\documentclass[journal]{IEEEtran}
\usepackage{graphicx} 
\usepackage{array}
\usepackage{arydshln}
\usepackage{xcolor,colortbl}
\usepackage{cite}
\usepackage{tikz,  pgfplots}
\usepackage{mathrsfs}
\usepackage{footmisc}
\usepackage{balance}
\usepackage{amsmath}
\usepackage{amssymb}
\usepackage{amsfonts}
\usepackage{amsthm}
\usepackage{mathrsfs}
\usepackage{xspace}
\usepackage{bm}
\usepackage{textcomp}
\usepackage[Symbol]{upgreek}
\usepackage{mleftright}
\usepackage{enumerate}
\usepackage{mathtools}
\usepackage{bbm}

\makeatletter
\if@twocolumn
\newcommand{\mysetminus}{\!\setminus\!} 
\else
\newcommand{\mysetminus}{\!\setminus\!} 
\fi
\makeatother

\newcommand{\safemath}[2]{\newcommand{#1}{\ensuremath{#2}\xspace}}

\newcommand{\diam}{\operatorname{diam}}

\newcommand{\lebmeasure}{\lambda}

\newcommand{\ssa}{\mathsf{a}}
\newcommand{\ssb}{\mathsf{b}}
\newcommand{\ssc}{\mathsf{c}}
\newcommand{\ssd}{\mathsf{d}}
\newcommand{\sse}{\mathsf{e}}
\newcommand{\ssf}{\mathsf{f}}
\newcommand{\ssg}{\mathsf{g}}
\newcommand{\ssh}{\mathsf{h}}
\newcommand{\ssi}{\mathsf{i}}
\newcommand{\ssj}{\mathsf{j}}
\newcommand{\ssk}{\mathsf{k}}
\newcommand{\ssl}{\mathsf{l}}
\newcommand{\ssm}{\mathsf{m}}
\newcommand{\ssn}{\mathsf{n}}
\newcommand{\sso}{\mathsf{o}}
\newcommand{\ssp}{\mathsf{p}}
\newcommand{\ssq}{\mathsf{q}}
\newcommand{\ssr}{\mathsf{r}}
\newcommand{\sss}{\mathsf{s}}
\newcommand{\sst}{\mathsf{t}}
\newcommand{\ssu}{\mathsf{u}}
\newcommand{\ssv}{\mathsf{v}}
\newcommand{\ssw}{\mathsf{w}}
\newcommand{\ssx}{\mathsf{x}}
\newcommand{\ssy}{\mathsf{y}}
\newcommand{\ssz}{\mathsf{z}}

\safemath{\bmsa}{\bm{\ssa}}
\safemath{\bmsb}{\bm{\ssb}}
\safemath{\bmsc}{\bm{\ssc}}
\safemath{\bmsd}{\bm{\ssd}}
\safemath{\bmse}{\bm{\sse}}
\safemath{\bmsf}{\bm{\ssf}}
\safemath{\bmsg}{\bm{\ssg}}
\safemath{\bmsh}{\bm{\ssh}}
\safemath{\bmsi}{\bm{\ssi}}
\safemath{\bmsj}{\bm{\ssj}}
\safemath{\bmsk}{\bm{\ssk}}
\safemath{\bmsl}{\bm{\ssl}}
\safemath{\bmsm}{\bm{\ssm}}
\safemath{\bmsn}{\bm{\ssn}}
\safemath{\bmso}{\bm{\sso}}
\safemath{\bmsp}{\bm{\ssp}}
\safemath{\bmsq}{\bm{\ssq}}
\safemath{\bmsr}{\bm{\ssr}}
\safemath{\bmss}{\bm{\sss}}
\safemath{\bmst}{\bm{\sst}}
\safemath{\bmsu}{\bm{\ssu}}
\safemath{\bmsv}{\bm{\ssv}}
\safemath{\bmsw}{\bm{\ssw}}
\safemath{\bmsx}{\bm{\ssx}}
\safemath{\bmsy}{\bm{\ssy}}
\safemath{\bmsz}{\bm{\ssz}}

\bmdefine{\bmualphad}{\upalpha}
\bmdefine{\bmubetad}{\upbeta}
\bmdefine{\bmuchid}{\upchi}
\bmdefine{\bmudeltad}{\updelta}
\bmdefine{\bmuepsilond}{\upepsilon}
\bmdefine{\bmuvarepsilond}{\upvarepsilon}
\bmdefine{\bmuetad}{\upeta}
\bmdefine{\bmugammad}{\upgamma}
\bmdefine{\bmuiotad}{\upiota}
\bmdefine{\bmukappad}{\upkappa}
\bmdefine{\bmulambdad}{\uplambda}
\bmdefine{\bmumud}{\upmu}
\bmdefine{\bmunud}{\upnu}
\bmdefine{\bmuomegad}{\upomega}
\bmdefine{\bmuphid}{\upphi}
\bmdefine{\bmuvarphid}{\upvarphi}
\bmdefine{\bmupid}{\uppi}
\bmdefine{\bmuvarpid}{\upvarpi}
\bmdefine{\bmupsid}{\uppsi}
\bmdefine{\bmurhod}{\uprho}
\bmdefine{\bmuvarrhod}{\upvarrho}
\bmdefine{\bmusigmad}{\upsigma}
\bmdefine{\bmuvarsigmad}{\upvarsigma}
\bmdefine{\bmutaud}{\uptau}
\bmdefine{\bmuthetad}{\uptheta}
\bmdefine{\bmuvarthetad}{\upvartheta}
\bmdefine{\bmuupsilond}{\upupsilon}
\bmdefine{\bmuxid}{\upxi}
\bmdefine{\bmuzetad}{\upzeta}

\safemath{\bmua}{\mathbf{a}}
\safemath{\bmub}{\mathbf{b}}
\safemath{\bmuc}{\mathbf{c}}
\safemath{\bmud}{\mathbf{d}}
\safemath{\bmue}{\mathbf{e}}
\safemath{\bmuf}{\mathbf{f}}
\safemath{\bmug}{\mathbf{g}}
\safemath{\bmuh}{\mathbf{h}}
\safemath{\bmui}{\mathbf{i}}
\safemath{\bmuj}{\mathbf{j}}
\safemath{\bmuk}{\mathbf{k}}
\safemath{\bmul}{\mathbf{l}}
\safemath{\bmum}{\mathbf{m}}
\safemath{\bmun}{\mathbf{n}}
\safemath{\bmuo}{\mathbf{o}}
\safemath{\bmup}{\mathbf{p}}
\safemath{\bmuq}{\mathbf{q}}
\safemath{\bmur}{\mathbf{r}}
\safemath{\bmus}{\mathbf{s}}
\safemath{\bmut}{\mathbf{t}}
\safemath{\bmuu}{\mathbf{u}}
\safemath{\bmuv}{\mathbf{v}}
\safemath{\bmuw}{\mathbf{w}}
\safemath{\bmux}{\mathbf{x}}
\safemath{\bmuy}{\mathbf{y}}
\safemath{\bmuz}{\mathbf{z}}

\safemath{\bmualpha}{\bmualphad}
\safemath{\bmubeta}{\bmubetad}
\safemath{\bmuchi}{\bumchid}
\safemath{\bmudelta}{\bmudeltad}
\safemath{\bmuepsilon}{\bmuepsilond}
\safemath{\bmuvarepsilon}{\bmuvarepsilond}
\safemath{\bmueta}{\bmuetad}
\safemath{\bmugamma}{\bmugammad}
\safemath{\bmuiota}{\bmuiotad}
\safemath{\bmukappa}{\bmukappad}
\safemath{\bmulambda}{\bmulambdad}
\safemath{\bmumu}{\bmumud}
\safemath{\bmunu}{\bmunud}
\safemath{\bmuomega}{\bmuomegad}
\safemath{\bmuphi}{\bmuphid}
\safemath{\bmuvarphi}{\bmuvarphid}
\safemath{\bmupi}{\bmupid}
\safemath{\bmuvarpi}{\bmuvarpid}
\safemath{\bmupsi}{\bmupsid}
\safemath{\bmurho}{\bmurhod}
\safemath{\bmuvarrho}{\bmuvarrhod}
\safemath{\bmusigma}{\bmusigmad}
\safemath{\bmuvarsigma}{\bmuvarsigmad}
\safemath{\bmutau}{\bmutaud}
\safemath{\bmutheta}{\bmuthetad}
\safemath{\bmuvartheta}{\bmuvarthetad}
\safemath{\bmuupsilon}{\bmuupsilond}
\safemath{\bmuxi}{\bmuxid}
\safemath{\bmuzeta}{\bmuzetad}

\bmdefine{\bmiad}{a}
\bmdefine{\bmibd}{b}
\bmdefine{\bmicd}{c}
\bmdefine{\bmidd}{d}
\bmdefine{\bmied}{e}
\bmdefine{\bmifd}{f}
\bmdefine{\bmigd}{g}
\bmdefine{\bmihd}{h}
\bmdefine{\bmiid}{i}
\bmdefine{\bmijd}{j}
\bmdefine{\bmikd}{k}
\bmdefine{\bmild}{l}
\bmdefine{\bmimd}{m}
\bmdefine{\bmind}{n}
\bmdefine{\bmiod}{o}
\bmdefine{\bmipd}{p}
\bmdefine{\bmiqd}{q}
\bmdefine{\bmird}{r}
\bmdefine{\bmisd}{s}
\bmdefine{\bmitd}{t}
\bmdefine{\bmiud}{u}
\bmdefine{\bmivd}{v}
\bmdefine{\bmiwd}{w}
\bmdefine{\bmixd}{x}
\bmdefine{\bmiyd}{y}
\bmdefine{\bmizd}{z}

\bmdefine{\bmialphad}{\alpha}
\bmdefine{\bmibetad}{\beta}
\bmdefine{\bmichid}{\chi}
\bmdefine{\bmideltad}{\delta}
\bmdefine{\bmiepsilond}{\epsilon}
\bmdefine{\bmivarepsilond}{\varepsilon}
\bmdefine{\bmietad}{\eta}
\bmdefine{\bmigammad}{\gamma}
\bmdefine{\bmiiotad}{\iota}
\bmdefine{\bmikappad}{\kappa}
\bmdefine{\bmivarkappad}{\varkappa}
\bmdefine{\bmilambdad}{\lambda}
\bmdefine{\bmimud}{\mu}
\bmdefine{\bminud}{\nu}
\bmdefine{\bmiomegad}{\omega}
\bmdefine{\bmiphid}{\phi}
\bmdefine{\bmivarphid}{\varphi}
\bmdefine{\bmipid}{\pi}
\bmdefine{\bmivarpid}{\varpi}
\bmdefine{\bmipsid}{\psi}
\bmdefine{\bmirhod}{\rho}
\bmdefine{\bmivarrhod}{\varrho}
\bmdefine{\bmisigmad}{\sigma}
\bmdefine{\bmivarsigmad}{\varsigma}
\bmdefine{\bmitaud}{\tau}
\bmdefine{\bmithetad}{\theta}
\bmdefine{\bmivarthetad}{\vartheta}
\bmdefine{\bmiupsilond}{\upsilon}
\bmdefine{\bmixid}{\xi}
\bmdefine{\bmizetad}{\zeta}

\safemath{\bmia}{\bmiad}
\safemath{\bmib}{\bmibd}
\safemath{\bmic}{\bmicd}
\safemath{\bmid}{\bmidd}
\safemath{\bmie}{\bmied}
\safemath{\bmif}{\bmifd}
\safemath{\bmig}{\bmigd}
\safemath{\bmih}{\bmihd}
\safemath{\bmii}{\bmiid}
\safemath{\bmij}{\bmijd}
\safemath{\bmik}{\bmikd}
\safemath{\bmil}{\bmild}
\safemath{\bmim}{\bmimd}
\safemath{\bmin}{\bmind}
\safemath{\bmio}{\bmiod}
\safemath{\bmip}{\bmipd}
\safemath{\bmiq}{\bmiqd}
\safemath{\bmir}{\bmird}
\safemath{\bmis}{\bmisd}
\safemath{\bmit}{\bmitd}
\safemath{\bmiu}{\bmiud}
\safemath{\bmiv}{\bmivd}
\safemath{\bmiw}{\bmiwd}
\safemath{\bmix}{\bmixd}
\safemath{\bmiy}{\bmiyd}
\safemath{\bmiz}{\bmizd}

\safemath{\bmialpha}{\bmialphad}
\safemath{\bmibeta}{\bmibetad}
\safemath{\bmichi}{\bmichid}
\safemath{\bmidelta}{\bmideltad}
\safemath{\bmiepsilon}{\bmiepsilond}
\safemath{\bmivarepsilon}{\bmivarepsilond}
\safemath{\bmieta}{\bmietad}
\safemath{\bmigamma}{\bmigammad}
\safemath{\bmiiota}{\bmiiotad}
\safemath{\bmikappa}{\bmikappad}
\safemath{\bmivarkappa}{\bmivarkappad}
\safemath{\bmilambda}{\bmilambdad}
\safemath{\bmimu}{\bmimud}
\safemath{\bminu}{\bminud}
\safemath{\bmiomega}{\bmiomegad}
\safemath{\bmiphi}{\bmiphid}
\safemath{\bmivarphi}{\bmivarphid}
\safemath{\bmipi}{\bmipid}
\safemath{\bmivarpi}{\bmivarpid}
\safemath{\bmipsi}{\bmipsid}
\safemath{\bmirho}{\bmirhod}
\safemath{\bmivarrho}{\bmivarrhod}
\safemath{\bmisigma}{\bmisigmad}
\safemath{\bmivarsigma}{\bmivarsigmad}
\safemath{\bmitau}{\bmitaud}
\safemath{\bmitheta}{\bmithetad}
\safemath{\bmivartheta}{\bmivarthetad}
\safemath{\bmiupsilon}{\bmiupsilond}
\safemath{\bmixi}{\bmixid}
\safemath{\bmizeta}{\bmizetad}

\bmdefine{\bmuDeltad}{\Updelta}
\bmdefine{\bmuGammad}{\Upgamma}
\bmdefine{\bmuLambdad}{\Uplambda}
\bmdefine{\bmuOmegad}{\Upomega}
\bmdefine{\bmuPhid}{\Upphi}
\bmdefine{\bmuPid}{\Uppi}
\bmdefine{\bmuPsid}{\Uppsi}
\bmdefine{\bmuSigmad}{\Upsigma}
\bmdefine{\bmuThetad}{\Uptheta}
\bmdefine{\bmuUpsilond}{\Upupsilon}
\bmdefine{\bmuXid}{\Upxi}

\safemath{\bmuA}{\mathbf{A}}
\safemath{\bmuB}{\mathbf{B}}
\safemath{\bmuC}{\mathbf{C}}
\safemath{\bmuD}{\mathbf{D}}
\safemath{\bmuE}{\mathbf{E}}
\safemath{\bmuF}{\mathbf{F}}
\safemath{\bmuG}{\mathbf{G}}
\safemath{\bmuH}{\mathbf{H}}
\safemath{\bmuI}{\mathbf{I}}
\safemath{\bmuJ}{\mathbf{J}}
\safemath{\bmuK}{\mathbf{K}}
\safemath{\bmuL}{\mathbf{L}}
\safemath{\bmuM}{\mathbf{M}}
\safemath{\bmuN}{\mathbf{N}}
\safemath{\bmuO}{\mathbf{O}}
\safemath{\bmuP}{\mathbf{P}}
\safemath{\bmuQ}{\mathbf{Q}}
\safemath{\bmuR}{\mathbf{R}}
\safemath{\bmuS}{\mathbf{S}}
\safemath{\bmuT}{\mathbf{T}}
\safemath{\bmuU}{\mathbf{U}}
\safemath{\bmuV}{\mathbf{V}}
\safemath{\bmuW}{\mathbf{W}}
\safemath{\bmuX}{\mathbf{X}}
\safemath{\bmuY}{\mathbf{Y}}
\safemath{\bmuZ}{\mathbf{Z}}

\safemath{\bmuZero}{\mathbf{0}}
\safemath{\bmuOne}{\mathbf{1}}

\safemath{\bmuDelta}{\bmuDeltad}
\safemath{\bmuGamma}{\bmuGammad}
\safemath{\bmuLambda}{\bmuLambdad}
\safemath{\bmuOmega}{\bmuOmegad}
\safemath{\bmuPhi}{\bmuPhid}
\safemath{\bmuPi}{\bmuPid}
\safemath{\bmuPsi}{\bmuPsid}
\safemath{\bmuSigma}{\bmuSigmad}
\safemath{\bmuTheta}{\bmuThetad}
\safemath{\bmuUpsilon}{\bmuUpsilond}
\safemath{\bmuXi}{\bmuXid}

\bmdefine{\bmiAd}{A}
\bmdefine{\bmiBd}{B}
\bmdefine{\bmiCd}{C}
\bmdefine{\bmiDd}{D}
\bmdefine{\bmiEd}{E}
\bmdefine{\bmiFd}{F}
\bmdefine{\bmiGd}{G}
\bmdefine{\bmiHd}{H}
\bmdefine{\bmiId}{I}
\bmdefine{\bmiJd}{J}
\bmdefine{\bmiKd}{K}
\bmdefine{\bmiLd}{L}
\bmdefine{\bmiMd}{M}
\bmdefine{\bmiOd}{N}
\bmdefine{\bmiPd}{O}
\bmdefine{\bmiQd}{P}
\bmdefine{\bmiRd}{R}
\bmdefine{\bmiSd}{S}
\bmdefine{\bmiTd}{T}
\bmdefine{\bmiUd}{U}
\bmdefine{\bmiVd}{V}
\bmdefine{\bmiWd}{W}
\bmdefine{\bmiXd}{X}
\bmdefine{\bmiYd}{Y}
\bmdefine{\bmiZd}{Z}

\bmdefine{\bmiDeltad}{\Delta}
\bmdefine{\bmiGammad}{\Gamma}
\bmdefine{\bmiLambdad}{\Lambda}
\bmdefine{\bmiOmegad}{\Omega}
\bmdefine{\bmiPhid}{\Phi}
\bmdefine{\bmiPid}{\Pi}
\bmdefine{\bmiPsid}{\Psi}
\bmdefine{\bmiSigmad}{\Sigma}
\bmdefine{\bmiThetad}{\Theta}
\bmdefine{\bmiUpsilond}{\Upsilon}
\bmdefine{\bmiXid}{\Xi}

\safemath{\bmiA}{\bmiAd}
\safemath{\bmiB}{\bmiBd}
\safemath{\bmiC}{\bmiCd}
\safemath{\bmiD}{\bmiDd}
\safemath{\bmiE}{\bmiEd}
\safemath{\bmiF}{\bmiFd}
\safemath{\bmiG}{\bmiGd}
\safemath{\bmiH}{\bmiHd}
\safemath{\bmiI}{\bmiId}
\safemath{\bmiJ}{\bmiJd}
\safemath{\bmiK}{\bmiKd}
\safemath{\bmiL}{\bmiLd}
\safemath{\bmiM}{\bmiMd}
\safemath{\bmiN}{\bmiNd}
\safemath{\bmiO}{\bmiOd}
\safemath{\bmiP}{\bmiPd}
\safemath{\bmiQ}{\bmiQd}
\safemath{\bmiR}{\bmiRd}
\safemath{\bmiS}{\bmiSd}
\safemath{\bmiT}{\bmiTd}
\safemath{\bmiU}{\bmiUd}
\safemath{\bmiV}{\bmiVd}
\safemath{\bmiW}{\bmiWd}
\safemath{\bmiX}{\bmiXd}
\safemath{\bmiY}{\bmiYd}
\safemath{\bmiZ}{\bmiZd}

\safemath{\bmiDelta}{\bmiDeltad}
\safemath{\bmiGamma}{\bmiGammad}
\safemath{\bmiLambda}{\bmiLambdad}
\safemath{\bmiOmega}{\bmiOmegad}
\safemath{\bmiPhi}{\bmiPhid}
\safemath{\bmiPi}{\bmiPid}
\safemath{\bmiPsi}{\bmiPsid}
\safemath{\bmiSigma}{\bmiSigmad}
\safemath{\bmiTheta}{\bmiThetad}
\safemath{\bmiUpsilon}{\bmiUpsilond}
\safemath{\bmiXi}{\bmiXid}

\safemath{\evA}{\mathcal{A}}
\safemath{\evB}{\mathcal{B}}
\safemath{\evC}{\mathcal{C}}
\safemath{\evD}{\mathcal{D}}
\safemath{\evE}{\mathcal{E}}
\safemath{\evF}{\mathcal{F}}
\safemath{\evG}{\mathcal{G}}
\safemath{\evH}{\mathcal{H}}
\safemath{\evI}{\mathcal{I}}
\safemath{\evJ}{\mathcal{J}}
\safemath{\evK}{\mathcal{K}}
\safemath{\evL}{\mathcal{L}}
\safemath{\evM}{\mathcal{M}}
\safemath{\evN}{\mathcal{N}}
\safemath{\evO}{\mathcal{O}}
\safemath{\evP}{\mathcal{P}}
\safemath{\evQ}{\mathcal{Q}}
\safemath{\evR}{\mathcal{R}}
\safemath{\evS}{\mathcal{S}}
\safemath{\evT}{\mathcal{T}}
\safemath{\evU}{\mathcal{U}}
\safemath{\evV}{\mathcal{V}}
\safemath{\evW}{\mathcal{W}}
\safemath{\evX}{\mathcal{X}}
\safemath{\evY}{\mathcal{Y}}
\safemath{\evZ}{\mathcal{Z}}

\safemath{\setA}{\mathcal{A}}
\safemath{\setB}{\mathcal{B}}
\safemath{\setC}{\mathcal{C}}
\safemath{\setD}{\mathcal{D}}
\safemath{\setE}{\mathcal{E}}
\safemath{\setF}{\mathcal{F}}
\safemath{\setG}{\mathcal{G}}
\safemath{\setH}{\mathcal{H}}
\safemath{\setI}{\mathcal{I}}
\safemath{\setJ}{\mathcal{J}}
\safemath{\setK}{\mathcal{K}}
\safemath{\setL}{\mathcal{L}}
\safemath{\setM}{\mathcal{M}}
\safemath{\setN}{\mathcal{N}}
\safemath{\setO}{\mathcal{O}}
\safemath{\setP}{\mathcal{P}}
\safemath{\setQ}{\mathcal{Q}}
\safemath{\setR}{\mathcal{R}}
\safemath{\setS}{\mathcal{S}}
\safemath{\setT}{\mathcal{T}}
\safemath{\setU}{\mathcal{U}}
\safemath{\setV}{\mathcal{V}}
\safemath{\setW}{\mathcal{W}}
\safemath{\setX}{\mathcal{X}}
\safemath{\setY}{\mathcal{Y}}
\safemath{\setZ}{\mathcal{Z}}
\safemath{\emptySet}{\varnothing}

\safemath{\colA}{\mathscr{A}}
\safemath{\colB}{\mathscr{B}}
\safemath{\colC}{\mathscr{C}}
\safemath{\colD}{\mathscr{D}}
\safemath{\colE}{\mathscr{E}}
\safemath{\colF}{\mathscr{F}}
\safemath{\colG}{\mathscr{G}}
\safemath{\colH}{\mathscr{H}}
\safemath{\colI}{\mathscr{I}}
\safemath{\colJ}{\mathscr{J}}
\safemath{\colK}{\mathscr{K}}
\safemath{\colL}{\mathscr{L}}
\safemath{\colM}{\mathscr{M}}
\safemath{\colN}{\mathscr{N}}
\safemath{\colO}{\mathscr{O}}
\safemath{\colP}{\mathscr{P}}
\safemath{\colQ}{\mathscr{Q}}
\safemath{\colR}{\mathscr{R}}
\safemath{\colS}{\mathscr{S}}
\safemath{\colT}{\mathscr{T}}
\safemath{\colU}{\mathscr{U}}
\safemath{\colV}{\mathscr{V}}
\safemath{\colW}{\mathscr{W}}
\safemath{\colX}{\mathscr{X}}
\safemath{\colY}{\mathscr{Y}}
\safemath{\colZ}{\mathscr{Z}}

\safemath{\opA}{\operatorname{A}}
\safemath{\opB}{\operatorname{B}}
\safemath{\opC}{\operatorname{C}}
\safemath{\opD}{\operatorname{D}}
\safemath{\opE}{\operatorname{E}}
\safemath{\opF}{\operatorname{F}}
\safemath{\opG}{\operatorname{G}}
\safemath{\opH}{\operatorname{H}}
\safemath{\opI}{\operatorname{I}}
\safemath{\opJ}{\operatorname{J}}
\safemath{\opK}{\operatorname{K}}
\safemath{\opL}{\operatorname{L}}
\safemath{\opM}{\operatorname{M}}
\safemath{\opN}{\operatorname{N}}
\safemath{\opO}{\operatorname{O}}
\safemath{\opP}{\operatorname{P}}
\safemath{\opQ}{\operatorname{Q}}
\safemath{\opR}{\operatorname{R}}
\safemath{\opS}{\operatorname{S}}
\safemath{\opT}{\operatorname{T}}
\safemath{\opU}{\operatorname{U}}
\safemath{\opV}{\operatorname{V}}
\safemath{\opW}{\operatorname{W}}
\safemath{\opX}{\operatorname{X}}
\safemath{\opY}{\operatorname{Y}}
\safemath{\opZ}{\operatorname{Z}}
\safemath{\opZero}{\operatorname{O}}
\safemath{\identityop}{\opI}

\safemath{\sca}{a}
\safemath{\scb}{b}
\safemath{\scc}{c}
\safemath{\scd}{d}
\safemath{\sce}{e}
\safemath{\scf}{f}
\safemath{\scg}{g}
\safemath{\sch}{h}
\safemath{\sci}{i}
\safemath{\scj}{j}
\safemath{\sck}{k}
\safemath{\scl}{l}
\safemath{\scm}{m}
\safemath{\scn}{n}
\safemath{\sco}{o}
\safemath{\scp}{p}
\safemath{\scq}{q}
\safemath{\scr}{r}
\safemath{\scs}{s}
\safemath{\sct}{t}
\safemath{\scu}{u}
\safemath{\scv}{v}
\safemath{\scw}{w}
\safemath{\scx}{x}
\safemath{\scy}{y}
\safemath{\scz}{z}

\safemath{\scA}{A}
\safemath{\scB}{B}
\safemath{\scC}{C}
\safemath{\scD}{D}
\safemath{\scE}{E}
\safemath{\scF}{F}
\safemath{\scG}{G}
\safemath{\scH}{H}
\safemath{\scI}{I}
\safemath{\scJ}{J}
\safemath{\scK}{K}
\safemath{\scL}{L}
\safemath{\scM}{M}
\safemath{\scN}{N}
\safemath{\scO}{O}
\safemath{\scP}{P}
\safemath{\scQ}{Q}
\safemath{\scR}{R}
\safemath{\scS}{S}
\safemath{\scT}{T}
\safemath{\scU}{U}
\safemath{\scV}{V}
\safemath{\scW}{W}
\safemath{\scX}{X}
\safemath{\scY}{Y}
\safemath{\scZ}{Z}

\safemath{\scalpha}{\alpha}
\safemath{\scbeta}{\beta}
\safemath{\scchi}{\chi}
\safemath{\scdelta}{\delta}
\safemath{\scepsilon}{\epsilon}
\safemath{\scvarepsilon}{\varepsilon}
\safemath{\sceta}{\eta}
\safemath{\scgamma}{\gamma}
\safemath{\sciota}{\iota}
\safemath{\sckappa}{\kappa}
\safemath{\scvarkappa}{\varkappa}
\safemath{\sclambda}{\lambda}
\safemath{\scmu}{\mu}
\safemath{\scnu}{\nu}
\safemath{\scomega}{\omega}
\safemath{\scphi}{\phi}
\safemath{\scvarphi}{\varphi}
\safemath{\scpi}{\pi}
\safemath{\scvarpi}{\varpi}
\safemath{\scpsi}{\psi}
\safemath{\scrho}{\rho}
\safemath{\scvarrho}{\varrho}
\safemath{\scsigma}{\sigma}
\safemath{\scvarsigma}{\varsigma}
\safemath{\sctau}{\tau}
\safemath{\sctheta}{\theta}
\safemath{\scvartheta}{\vartheta}
\safemath{\scupsilon}{\upsilon}
\safemath{\scxi}{\xi}
\safemath{\sczeta}{\zeta}

\safemath{\veca}{{\boldsymbol{a}}}
\safemath{\vecb}{{\boldsymbol{b}}}
\safemath{\vecc}{{\boldsymbol{c}}}
\safemath{\vecd}{{\boldsymbol{d}}}
\safemath{\vece}{{\boldsymbol{e}}}
\safemath{\vecf}{{\boldsymbol{f}}}
\safemath{\vecg}{{\boldsymbol{g}}}
\safemath{\vech}{{\boldsymbol{h}}}
\safemath{\veci}{{\boldsymbol{i}}}
\safemath{\vecj}{{\boldsymbol{j}}}
\safemath{\veck}{{\boldsymbol{k}}}
\safemath{\vecl}{{\boldsymbol{l}}}
\safemath{\vecm}{{\boldsymbol{m}}}
\safemath{\vecn}{{\boldsymbol{n}}}
\safemath{\veco}{{\boldsymbol{o}}}
\safemath{\vecp}{{\boldsymbol{p}}}
\safemath{\vecq}{{\boldsymbol{q}}}
\safemath{\vecr}{{\boldsymbol{r}}}
\safemath{\vecs}{{\boldsymbol{s}}}
\safemath{\vect}{{\boldsymbol{t}}}
\safemath{\vecu}{{\boldsymbol{u}}}
\safemath{\vecv}{{\boldsymbol{v}}}
\safemath{\vecw}{{\boldsymbol{w}}}
\safemath{\vecx}{{\boldsymbol{x}}}
\safemath{\vecy}{{\boldsymbol{y}}}
\safemath{\vecz}{{\boldsymbol{z}}}
\safemath{\veczero}{{\boldsymbol{0}}}
\safemath{\vecone}{{\boldsymbol{1}}}

\safemath{\vecalpha}{\upalpha}
\safemath{\vecbeta}{\upbeta}
\safemath{\vecchi}{\upchi}
\safemath{\vecdelta}{\updelta}
\safemath{\vecepsilon}{\upepsilon}
\safemath{\vecvarepsilon}{\upvarepsilon}
\safemath{\veceta}{\upeta}
\safemath{\vecgamma}{\upgamma}
\safemath{\veciota}{\upiota}
\safemath{\veckappa}{\upkappa}
\safemath{\veclambda}{\uplambda}
\safemath{\vecmu}{\text{\textmu}}
\safemath{\vecnu}{\upnu}
\safemath{\vecomega}{\upomega}
\safemath{\vecphi}{\upphi}
\safemath{\vecvarphi}{\upvarphi}
\safemath{\vecpi}{\uppi}
\safemath{\vecvarpi}{\upvarpi}
\safemath{\vecpsi}{\uppsi}
\safemath{\vecrho}{\uprho}
\safemath{\vecvarrho}{\upvarrho}
\safemath{\vecsigma}{\upsigma}
\safemath{\vecvarsigma}{\upvarsigma}
\safemath{\vectau}{\uptau}
\safemath{\vectheta}{\uptheta}
\safemath{\vecvartheta}{\upvartheta}
\safemath{\vecupsilon}{\upupsilon}
\safemath{\vecxi}{\upxi}
\safemath{\veczeta}{\upzeta}

\safemath{\vecac}{a}
\safemath{\vecbc}{b}
\safemath{\veccc}{c}
\safemath{\vecdc}{d}
\safemath{\vecec}{e}
\safemath{\vecfc}{f}
\safemath{\vecgc}{g}
\safemath{\vechc}{h}
\safemath{\vecic}{i}
\safemath{\vecjc}{j}
\safemath{\veckc}{k}
\safemath{\veclc}{l}
\safemath{\vecmc}{m}
\safemath{\vecnc}{n}
\safemath{\vecoc}{o}
\safemath{\vecpc}{p}
\safemath{\vecqc}{q}
\safemath{\vecrc}{r}
\safemath{\vecsc}{s}
\safemath{\vectc}{t}
\safemath{\vecuc}{u}
\safemath{\vecvc}{v}
\safemath{\vecwc}{w}
\safemath{\vecxc}{x}
\safemath{\vecyc}{y}
\safemath{\veczc}{z}

\safemath{\matA}{{\boldsymbol{A}}}
\safemath{\matB}{{\boldsymbol{B}}}
\safemath{\matC}{{\boldsymbol{C}}}
\safemath{\matD}{{\boldsymbol{D}}}
\safemath{\matE}{{\boldsymbol{E}}}
\safemath{\matF}{{\boldsymbol{F}}}
\safemath{\matG}{{\boldsymbol{G}}}
\safemath{\matH}{{\boldsymbol{H}}}
\safemath{\matI}{{\boldsymbol{I}}}
\safemath{\matJ}{{\boldsymbol{J}}}
\safemath{\matK}{{\boldsymbol{K}}}
\safemath{\matL}{{\boldsymbol{L}}}
\safemath{\matM}{{\boldsymbol{M}}}
\safemath{\matN}{{\boldsymbol{N}}}
\safemath{\matO}{{\boldsymbol{O}}}
\safemath{\matP}{{\boldsymbol{P}}}
\safemath{\matQ}{{\boldsymbol{Q}}}
\safemath{\matR}{{\boldsymbol{R}}}
\safemath{\matS}{{\boldsymbol{S}}}
\safemath{\matT}{{\boldsymbol{T}}}
\safemath{\matU}{{\boldsymbol{U}}}
\safemath{\matV}{{\boldsymbol{V}}}
\safemath{\matW}{{\boldsymbol{W}}}
\safemath{\matX}{{\boldsymbol{X}}}
\safemath{\matY}{{\boldsymbol{Y}}}
\safemath{\matZ}{{\boldsymbol{Z}}}
\safemath{\matzero}{{\boldsymbol{0}}}

\safemath{\matDelta}{\Updelta}
\safemath{\matGamma}{\Upgammma}
\safemath{\matLambda}{\Uplambda}
\safemath{\matOmega}{\Upomega}
\safemath{\matPhi}{\Upphi}
\safemath{\matPi}{\Uppi}
\safemath{\matPsi}{\Uppsi}
\safemath{\matSigma}{\Upsigma}
\safemath{\matTheta}{\Uptheta}
\safemath{\matUpsilon}{\Upupsilon}
\safemath{\matXi}{\Upxi}

\safemath{\matidentity}{\matI}
\safemath{\vecunit}{\vece} % i-th unit vector
\safemath{\matone}{\matO}

\safemath{\matAc}{a}
\safemath{\matBc}{b}
\safemath{\matCc}{c}
\safemath{\matDc}{d}
\safemath{\matEc}{e}
\safemath{\matFc}{f}
\safemath{\matGc}{g}
\safemath{\matHc}{h}
\safemath{\matIc}{i}
\safemath{\matJc}{j}
\safemath{\matKc}{k}
\safemath{\matLc}{l}
\safemath{\matMc}{m}
\safemath{\matNc}{n}
\safemath{\matOc}{o}
\safemath{\matPc}{p}
\safemath{\matQc}{q}
\safemath{\matRc}{r}
\safemath{\matSc}{s}
\safemath{\matTc}{t}
\safemath{\matUc}{u}
\safemath{\matVc}{v}
\safemath{\matWc}{w}
\safemath{\matXc}{x}
\safemath{\matYc}{y}
\safemath{\matZc}{z}

\safemath{\rnda}{\mathsf{a}}
\safemath{\rndb}{\mathsf{b}}
\safemath{\rndc}{\mathsf{c}}
\safemath{\rndd}{\mathsf{d}}
\safemath{\rnde}{\mathsf{e}}
\safemath{\rndf}{\mathsf{f}}
\safemath{\rndg}{\mathsf{g}}
\safemath{\rndh}{\mathsf{h}}
\safemath{\rndi}{\mathsf{i}}
\safemath{\rndj}{\mathsf{j}}
\safemath{\rndk}{\mathsf{k}}
\safemath{\rndl}{\mathsf{l}}
\safemath{\rndm}{\mathsf{m}}
\safemath{\rndn}{\mathsf{n}}
\safemath{\rndo}{\mathsf{o}}
\safemath{\rndp}{\mathsf{p}}
\safemath{\rndq}{\mathsf{q}}
\safemath{\rndr}{\mathsf{r}}
\safemath{\rnds}{\mathsf{s}}
\safemath{\rndt}{\mathsf{t}}
\safemath{\rndu}{\mathsf{u}}
\safemath{\rndv}{\mathsf{v}}
\safemath{\rndw}{\mathsf{w}}
\safemath{\rndx}{\mathsf{x}}
\safemath{\rndy}{\mathsf{y}}
\safemath{\rndz}{\mathsf{z}}

\safemath{\rndA}{\bmiA}
\safemath{\rndB}{\bmiB}
\safemath{\rndC}{\bmiC}
\safemath{\rndD}{\bmiD}
\safemath{\rndE}{\bmiE}
\safemath{\rndF}{\bmiF}
\safemath{\rndG}{\bmiG}
\safemath{\rndH}{\bmiH}
\safemath{\rndI}{\bmiI}
\safemath{\rndJ}{\bmiJ}
\safemath{\rndK}{\bmiK}
\safemath{\rndL}{\bmiL}
\safemath{\rndM}{\bmiM}
\safemath{\rndN}{\bmiN}
\safemath{\rndO}{\bmiO}
\safemath{\rndP}{\bmiP}
\safemath{\rndQ}{\bmiQ}
\safemath{\rndR}{\bmiR}
\safemath{\rndS}{\bmiS}
\safemath{\rndT}{\bmiT}
\safemath{\rndU}{\bmiU}
\safemath{\rndV}{\bmiV}
\safemath{\rndW}{\bmiW}
\safemath{\rndX}{\bmiX}
\safemath{\rndY}{\bmiY}
\safemath{\rndZ}{\bmiZ}

\safemath{\rndalpha}{\bmialpha}
\safemath{\rndbeta}{\bmibeta}
\safemath{\rndchi}{\bmichi}
\safemath{\rnddelta}{\bmidelta}
\safemath{\rndepsilon}{\bmiepsilon}
\safemath{\rndvarepsilon}{\bmivarepsilon}
\safemath{\rndeta}{\bmieta}
\safemath{\rndgamma}{\bmigamma}
\safemath{\rndiota}{\bmiiota}
\safemath{\rndkappa}{\bmikappa}
\safemath{\rndlambda}{\bmilambda}
\safemath{\rndmu}{\bmimu}
\safemath{\rndnu}{\bminu}
\safemath{\rndomega}{\bmiomega}
\safemath{\rndphi}{\bmiphi}
\safemath{\rndvarphi}{\bmivarphi}
\safemath{\rndpi}{\bmipi}
\safemath{\rndvarpi}{\bmivarpi}
\safemath{\rndpsi}{\bmipsi}
\safemath{\rndrho}{\bmirho}
\safemath{\rndvarrho}{\bmivarrho}
\safemath{\rndsigma}{\bmisigma}
\safemath{\rndvarsigma}{\bmivarsigma}
\safemath{\rndtau}{\bmitau}
\safemath{\rndtheta}{\bmitheta}
\safemath{\rndvartheta}{\bmivartheta}
\safemath{\rndupsilon}{\bmiupsilon}
\safemath{\rndxi}{\bmixi}
\safemath{\rndzeta}{\bmizeta}

\safemath{\rveca}{{\boldsymbol{\mathsf{a}}}}
\safemath{\rvecb}{{\boldsymbol{\mathsf{b}}}}
\safemath{\rvecc}{{\boldsymbol{\mathsf{c}}}}
\safemath{\rvecd}{{\boldsymbol{\mathsf{d}}}}
\safemath{\rvece}{{\boldsymbol{\mathsf{e}}}}
\safemath{\rvecf}{{\boldsymbol{\mathsf{f}}}}
\safemath{\rvecg}{{\boldsymbol{\mathsf{g}}}}
\safemath{\rvech}{{\boldsymbol{\mathsf{h}}}}
\safemath{\rveci}{{\boldsymbol{\mathsf{i}}}}
\safemath{\rvecj}{{\boldsymbol{\mathsf{j}}}}
\safemath{\rveck}{{\boldsymbol{\mathsf{k}}}}
\safemath{\rvecl}{{\boldsymbol{\mathsf{l}}}}
\safemath{\rvecm}{{\boldsymbol{\mathsf{m}}}}
\safemath{\rvecn}{{\boldsymbol{\mathsf{n}}}}
\safemath{\rveco}{{\boldsymbol{\mathsf{o}}}}
\safemath{\rvecp}{{\boldsymbol{\mathsf{p}}}}
\safemath{\rvecq}{{\boldsymbol{\mathsf{q}}}}
\safemath{\rvecr}{{\boldsymbol{\mathsf{r}}}}
\safemath{\rvecs}{{\boldsymbol{\mathsf{s}}}}
\safemath{\rvect}{{\boldsymbol{\mathsf{t}}}}
\safemath{\rvecu}{{\boldsymbol{\mathsf{u}}}}
\safemath{\rvecv}{{\boldsymbol{\mathsf{v}}}}
\safemath{\rvecw}{{\boldsymbol{\mathsf{w}}}}
\safemath{\rvecx}{{\boldsymbol{\mathsf{x}}}}
\safemath{\rvecy}{{\boldsymbol{\mathsf{y}}}}
\safemath{\rvecz}{{\boldsymbol{\mathsf{z}}}}

\safemath{\rvecalpha}{\bmualpha}
\safemath{\rvecbeta}{\bmubeta}
\safemath{\rvecchi}{\bmuchi}
\safemath{\rvecdelta}{\bmudelta}
\safemath{\rvecepsilon}{\bmuepsilon}
\safemath{\rvecvarepsilon}{\bmuvarepsilon}
\safemath{\rveceta}{\bmueta}
\safemath{\rvecgamma}{\bmugamma}
\safemath{\rveciota}{\bmuiota}
\safemath{\rveckappa}{\bmukappa}
\safemath{\rveclambda}{\bmulambda}
\safemath{\rvecmu}{\bmumu}
\safemath{\rvecnu}{\bmunu}
\safemath{\rvecomega}{\bmuomega}
\safemath{\rvecphi}{\bmuphi}
\safemath{\rvecvarphi}{\bmuvarphi}
\safemath{\rvecpi}{\bmupi}
\safemath{\rvecvarpi}{\bmuvarpi}
\safemath{\rvecpsi}{\bmupsi}
\safemath{\rvecrho}{\bmurho}
\safemath{\rvecvarrho}{\bmuvarrho}
\safemath{\rvecsigma}{\bmusigma}
\safemath{\rvecvarsigma}{\bmuvarsigma}
\safemath{\rvectau}{\bmutau}
\safemath{\rvectheta}{\bmutheta}
\safemath{\rvecvartheta}{\bmuvartheta}
\safemath{\rvecupsilon}{\bmuupsilon}
\safemath{\rvecxi}{\bmuxi}
\safemath{\rveczeta}{\bmuzeta}

\safemath{\rvecac}{\rnda}
\safemath{\rvecbc}{\rndb}
\safemath{\rveccc}{\rndc}
\safemath{\rvecdc}{\rndd}
\safemath{\rvecec}{\rnde}
\safemath{\rvecfc}{\rndf}
\safemath{\rvecgc}{\rndg}
\safemath{\rvechc}{\rndh}
\safemath{\rvecic}{\rndi}
\safemath{\rvecjc}{\rndj}
\safemath{\rveckc}{\rndk}
\safemath{\rveclc}{\rndl}
\safemath{\rvecmc}{\rndm}
\safemath{\rvecnc}{\rndn}
\safemath{\rvecoc}{\rndo}
\safemath{\rvecpc}{\rndp}
\safemath{\rvecqc}{\rndq}
\safemath{\rvecrc}{\rndr}
\safemath{\rvecsc}{\rnds}
\safemath{\rvectc}{\rndt}
\safemath{\rvecuc}{\rndu}
\safemath{\rvecvc}{\rndv}
\safemath{\rvecwc}{\rndw}
\safemath{\rvecxc}{\rndx}
\safemath{\rvecyc}{\rndy}
\safemath{\rveczc}{\rndz}

\safemath{\rmatA}{{\boldsymbol{\mathsf{A}}}}
\safemath{\rmatB}{{\boldsymbol{\mathsf{B}}}}
\safemath{\rmatC}{{\boldsymbol{\mathsf{C}}}}
\safemath{\rmatD}{{\boldsymbol{\mathsf{D}}}}
\safemath{\rmatE}{{\boldsymbol{\mathsf{E}}}}
\safemath{\rmatF}{{\boldsymbol{\mathsf{F}}}}
\safemath{\rmatG}{{\boldsymbol{\mathsf{G}}}}
\safemath{\rmatH}{{\boldsymbol{\mathsf{H}}}}
\safemath{\rmatI}{{\boldsymbol{\mathsf{I}}}}
\safemath{\rmatJ}{{\boldsymbol{\mathsf{J}}}}
\safemath{\rmatK}{{\boldsymbol{\mathsf{K}}}}
\safemath{\rmatL}{{\boldsymbol{\mathsf{L}}}}
\safemath{\rmatM}{{\boldsymbol{\mathsf{M}}}}
\safemath{\rmatN}{{\boldsymbol{\mathsf{N}}}}
\safemath{\rmatO}{{\boldsymbol{\mathsf{O}}}}
\safemath{\rmatP}{{\boldsymbol{\mathsf{P}}}}
\safemath{\rmatQ}{{\boldsymbol{\mathsf{Q}}}}
\safemath{\rmatR}{{\boldsymbol{\mathsf{R}}}}
\safemath{\rmatS}{{\boldsymbol{\mathsf{S}}}}
\safemath{\rmatT}{{\boldsymbol{\mathsf{T}}}}
\safemath{\rmatU}{{\boldsymbol{\mathsf{U}}}}
\safemath{\rmatV}{{\boldsymbol{\mathsf{V}}}}
\safemath{\rmatW}{{\boldsymbol{\mathsf{W}}}}
\safemath{\rmatX}{{\boldsymbol{\mathsf{X}}}}
\safemath{\rmatY}{{\boldsymbol{\mathsf{Y}}}}
\safemath{\rmatZ}{{\boldsymbol{\mathsf{Z}}}}
\safemath{\rmatDelta}{\bmuDelta}
\safemath{\rmatGamma}{\bmuGamma}
\safemath{\rmatLambda}{\bmuLambda}
\safemath{\rmatOmega}{\bmuOmega}
\safemath{\rmatPhi}{\bmuPhi}
\safemath{\rmatPi}{\bmuPi}
\safemath{\rmatPsi}{\bmuPsi}
\safemath{\rmatSigma}{\bmuSigma}
\safemath{\rmatTheta}{\bmuTheta}
\safemath{\rmatUpsilon}{\bmuUpsilon}
\safemath{\rmatXi}{\bmuXi}

\safemath{\rmatAc}{\rnda}
\safemath{\rmatBc}{\rndb}
\safemath{\rmatCc}{\rndc}
\safemath{\rmatDc}{\rndd}
\safemath{\rmatEc}{\rnde}
\safemath{\rmatFc}{\rndf}
\safemath{\rmatGc}{\rndg}
\safemath{\rmatHc}{\rndh}
\safemath{\rmatIc}{\rndi}
\safemath{\rmatJc}{\rndj}
\safemath{\rmatKc}{\rndk}
\safemath{\rmatLc}{\rndl}
\safemath{\rmatMc}{\rndm}
\safemath{\rmatNc}{\rndn}
\safemath{\rmatOc}{\rndo}
\safemath{\rmatPc}{\rndp}
\safemath{\rmatQc}{\rndq}
\safemath{\rmatRc}{\rndr}
\safemath{\rmatSc}{\rnds}
\safemath{\rmatTc}{\rndt}
\safemath{\rmatUc}{\rndu}
\safemath{\rmatVc}{\rndv}
\safemath{\rmatWc}{\rndw}
\safemath{\rmatXc}{\rndx}
\safemath{\rmatYc}{\rndy}
\safemath{\rmatZc}{\rndz}

\newenvironment{textbmatrix}{	\setlength{\arraycolsep}{2.5pt}%
								\big[\begin{matrix}}{\end{matrix}\big]%
								\raisebox{0.08ex}{\vphantom{M}}}

 \def\btm{\begin{textbmatrix}}
 \def\etm{\end{textbmatrix}}

\DeclareMathOperator{\rank}{rank}			% rank of a matrix
\DeclareMathOperator{\adj}{adj}				% adjunct matrix
\safemath{\fun}{\scf}						% generic scalar function

\safemath{\vrbl}{x}						% generic vector variable
\safemath{\altvrbl}{y}						% alt generic vector variable
\safemath{\aaltvrbl}{z}						% aalt generic vector variable

\safemath{\vvrbl}{\vecx}						% generic vector variable
\safemath{\altvvrbl}{\vecy}						% alt generic vector variable
\safemath{\aaltvvrbl}{\vecz}						% aalt generic vector variable

\safemath{\altfun}{\scg}
\safemath{\aaltfun}{\sch}
\safemath{\bel}{\sce}					% basis element
\safemath{\altbel}{\sce}					% alternative basis element
\safemath{\frel}{g}					% frame element
\safemath{\altfrel}{g}					% alternative frame element
\safemath{\dfrel}{\tilde{g}}					% dual frame element
\safemath{\altdfrel}{\tilde{g}}					% alternative dual frame element

\safemath{\mat}{\matA}						% generic matrix
\safemath{\matc}{\matAc}						% components of a generic matrix
\safemath{\altmat}{\matB}						% alternative generic matrix
\safemath{\altmatc}{\matBc}						% alternative generic matrix

\safemath{\vectr}{\vecu}						% generic vector
\safemath{\vectrc}{\vecuc}						% components of a generic vector
\safemath{\altvectr}{\vecv}						% alternative generic vector
\safemath{\aaltvectr}{\vect}						% aalternative generic vector
\safemath{\altvectrc}{\vecvc}						% components of an alternative generic vector
\safemath{\genvar}{u}						% generic variable
\safemath{\altgenvar}{v}						% alternative generic variable 

\safemath{\rvectr}{\rvecu}						% random generic vector
\safemath{\rvectrc}{\rvecuc}						% random components of a generic vector
\safemath{\raltvectr}{\rvecv}						% random alternative generic vector
\safemath{\raaltvectr}{\rvect}						% random aalternative generic vector
\safemath{\raltvectrc}{\rvecvc}						% random components of an alternative generic vector
\safemath{\rgenvar}{\rndu}						% random generic variable
\safemath{\raltgenvar}{\rndv}						% random alternative generic variable 

\newcommand{\nullspace}{\setN}	 			% nullspace
\newcommand{\card}[1]{\lvert#1\rvert}			% cardinality of a set
\newcommand{\ind}[1]{\mathbbm{1}_{#1}}				% indicator function
\newcommand{\conj}[1]{\ensuremath{#1^{*}}} 	% conjugate 		
\newcommand{\tp}[1]{\ensuremath{#1^{\mathsf{T}}}} 		% transpose

\newcommand{\inv}[1]{\ensuremath{#1^{-1}}} 	% inverse
\safemath{\dirac}{\delta}					% Dirac delta
\safemath{\diracp}{\dirac(\time)}			% 	''	parametrized
\safemath{\krond}{\dirac}					% Kronecker delta
\safemath{\indfun}{I}						% Indicator function
\safemath{\stepfun}{u}						% step function at zero

\safemath{\upto}{\uparrow}
\safemath{\downto}{\downarrow}
\safemath{\iu}{\mathrm{i}}							% imaginary unit
\safemath{\maj}{\succ}
\newcommand{\dftmat}[1]{\matF_{#1}}			% DFT matrix
\safemath{\mdft}{\dftmat{}}					% 	''
\safemath{\runity}{\beta}					% root of unity
\safemath{\eval}{\lambda}					% eigenvalue

\safemath{\veval}{\veclambda}				% eigenvalue vector
\safemath{\littleo}{\sco}					% Landau\s little o

\let\im\undefined
\safemath{\re}{\Re}				% real part
\safemath{\im}{\Im}				% imaginary part
\safemath{\euclidspace}{\complexset}			% Euclidean space
\safemath{\confunspace}{\setC}				% space of continuous functions
\newcommand{\banachseqspace}[1]{l^{#1}}		% Banach sequence space
\safemath{\hilseqspace}{\banachseqspace{2}}	% Hilbert sequence space
\newcommand{\banachfunspace}[1]{\setL^{#1}}	% Banach function space
\safemath{\hilfunspace}{\banachfunspace{2}}	% Hilbert function space
\safemath{\hilfunspacep}{\hilfunspace(\complexset)}	% Hilbert function space parametrized
\safemath{\schwarzspace}{\setS}				% Schwarz space

\newcommand{\hadj}[1]{#1^{\star}}			% Hilbert adjoint operator

\safemath{\SNR}{\rho} 				% signal to noise ratio
\safemath{\SINR}{\text{\sc sinr}} 				% signal to interference plus noise ratio
\safemath{\No}{N_0}							% noise spectral density
\safemath{\Es}{E_s}							% energy per symbol
\safemath{\Eb}{E_b}							% energy per bit
\safemath{\EbNo}{\frac{\Eb}{\No}}
\safemath{\EsNo}{\frac{\Es}{\No}}
\safemath{\NoVar}{\variance}                 % noise variance

\let\time\undefined
\safemath{\time}{\sct}						% continuous time
\safemath{\dtime}{\sck}						% discrete time
\safemath{\delay}{\sctau}					% continuous delay
\safemath{\ddelay}{\scl}						% discrete delay
\safemath{\doppler}{\scnu}					% continuous doppler
\safemath{\ddoppler}{\scm}					% discrete doppler
\safemath{\freq}{\scf}						% frequency
\safemath{\dfreq}{\scn}						% discrete frequency
\safemath{\Dtime}{\Delta\time}
\safemath{\Dfreq}{\Delta\freq}
\safemath{\Ddtime}{\dtime}
\safemath{\Ddfreq}{\dfreq}
\safemath{\bandwidth}{\scB}
\safemath{\maxdoppler}{\doppler_{0}}			% maximum Doppler shift
\safemath{\maxdelay}{\delay_{0}}				% maximum delay
\safemath{\spread}{\Delta_{\CHop}}			% total channel spread
\DeclareMathOperator{\CHop}{\ensuremath{\opH}} % channel operator
\safemath{\kernel}{\rndk_{\CHop}}			% operator kernel
\safemath{\kernelp}{\kernel(\time,\time')}	% 	''	parametrized
\safemath{\tvir}{\rndh_{\CHop}}				% time-varying impulse response
\safemath{\tvirp}{\tvir(\time,\delay)}		%	''	parametrized
\safemath{\tvirc}{\conj{\rndh}_{\CHop}}		% 	''	parametrized
\safemath{\tvtf}{\rndl_{\CHop}}				% time-varying transfer function
\safemath{\tvtfp}{\tvtf(\time,\freq)}			%	''	parametrized
\safemath{\tvtfc}{\conj{\rndl}_{\CHop}}		%	''	parametrized
\safemath{\spf}{\rnds_{\CHop}}				% spreading function
\safemath{\spfp}{\spf(\doppler,\delay)}		%	''	parametrized
\safemath{\spfc}{\conj{\rnds}_{\CHop}}		%	''	parametrized
\safemath{\bff}{\rndb_{\CHop}}				% bi-freuqency function
\safemath{\bffp}{\bff(\doppler,\freq)}		%	''	parametrized

\safemath{\irc}{\scr_{\rndh}}				% impulse response correlation fn.
\safemath{\tfc}{\scr_{\rndl}}				% time-frequency correlation fn.
\safemath{\spc}{\scr_{\rnds}}				% spreading fn. correlation fn.
\safemath{\bfc}{\scr_{\rndb}}				% bi-frequency correlation fn.

\safemath{\scaf}{\scc_{\rnds}}				% scattering function
\safemath{\scafp}{\scaf(\doppler,\delay)}		% 	''	parametrized
\safemath{\ccf}{\scc_{\rndl}}				% WSSUS tvtf correlation
\safemath{\ccfp}{\ccf(\Dtime,\Dfreq)}			% 	''	parametrized
\safemath{\cic}{\scc_{\rndh}}				% WSSUS tvir correlation
\safemath{\cicp}{\cic(\Dtime,\delay)}			% 	''	parametrized

\safemath{\mi}{I}							% mutual information
\safemath{\capacity}{C}					% capacity

\DeclareMathOperator{\Prob}{\opP}		% probability of an event
\safemath{\normal}{\mathcal{N}}			% normal distribution
\safemath{\jpg}{\mathcal{CN}}			% jointly proper Gaussian
\safemath{\uniform}{\mathcal{U}}				% uniform distribution
\safemath{\mchain}{\leftrightarrow}		% Markov chain
\safemath{\dB}{\,\mathrm{dB}}
\safemath{\dBm}{\,\mathrm{dBm}}
\safemath{\Hz}{\,\mathrm{Hz}}
\safemath{\kHz}{\,\mathrm{kHz}}
\safemath{\MHz}{\,\mathrm{MHz}}
\safemath{\GHz}{\,\mathrm{GHz}}
\safemath{\s}{\,\mathrm{s}}
\safemath{\ms}{\,\mathrm{ms}}
\safemath{\mus}{\,\mathrm{\text{\textmu}s}}
\safemath{\ns}{\,\mathrm{ns}}
\safemath{\ps}{\,\mathrm{ps}}
\safemath{\meter}{\,\mathrm{m}}
\safemath{\mm}{\,\mathrm{mm}}
\safemath{\cm}{\,\mathrm{cm}}
\safemath{\m}{\,\mathrm{m}}
\safemath{\W}{\,\mathrm{W}}
\safemath{\mW}{\, \mathrm{mW}}
\safemath{\J}{\,\mathrm{J}}
\safemath{\K}{\,\mathrm{K}}
\safemath{\bit}{\,\mathrm{bit}}
\safemath{\nat}{\,\mathrm{nat}}

\safemath{\define}{\triangleq}					% definition

\safemath{\equivalent}{\sim}
\safemath{\distas}{\sim}					% distributed according to
\safemath{\sdiff}{\Delta}				% symmetric set difference
\safemath{\setdiff}{\setminus}				% set difference

\safemath{\reals}{\mathbb R}
\safemath{\positivereals}{\reals^{+}}
\safemath{\integers}{\mathbb Z}
\safemath{\posint}{\integers^{+}}
\safemath{\naturals}{\mathbb N}
\safemath{\posnaturals}{\naturals^{+}}
\safemath{\complexset}{\mathbb C}
\safemath{\rationals}{\mathbb Q}

\safemath{\iSet}{\setI}
\safemath{\rel}{\bowtie}					% relation
\safemath{\eqrel}{\sim}					% equivalence relation
\safemath{\rlord}{\leq}					% reflexive linear ordering
\safemath{\slord}{<}						% strict linear ordering
\safemath{\rpord}{\preceq}				% reflexive partial ordering
\safemath{\rrpord}{\succeq}				% reversed reflexive partial ordering
\safemath{\spord}{\prec}					% strict partial ordering
\safemath{\sig}{\sigma}					% sigma-{algebra, ring,...}
\safemath{\metric}{d}					% metric
\safemath{\setfun}{\Phi}					% set function
\safemath{\measure}{\mu}					% measure
\safemath{\altmeasure}{\lambda}					% measure
\newcommand{\outerm}[1]{#1^{\star}}		% marks outer measures.
\newcommand{\innerm}[1]{#1_{\star}}		% marks inner measures.
\safemath{\omeasure}{\outerm{\measure}}		% outer measure
\safemath{\imeasure}{\innerm{\measure}}		% inner measure	
\safemath{\aecol}{\colS^{\star}_{\measure}} % collection of almost equal sets
\safemath{\emeasure}{\bar{\measure}_{0}}	% measure extension
\safemath{\rmeasure}{\tilde{\measure}}	% restricted measure
\safemath{\bmeasure}{\measure_{0}}		% basic measure on a semiring
\safemath{\glength}{\measure_{\altfun}}	% generalized length
\safemath{\lebmea}{\lambda}				% Lebesgue length and measure
\safemath{\blebmea}{\lebmea_{0}}			% pre-lebesgue-measure
\safemath{\sfun}{s}						% simple function
\safemath{\absintspace}{\colL^{1}}		% space of abs. integrable functions
\safemath{\sqintspace}{\colL^{2}}		% space of square integrable functions
\safemath{\abssumspace}{l^{1}}		% space of abs. summable sequences
\safemath{\sqsumspace}{l^{2}}		% space of square summable sequences

\safemath{\sfield}{\setF}				% scalar field
\safemath{\vectors}{\setV}				% set of vectors
\safemath{\vecspace}{(\vectors,\sfield)}	% vector space
\safemath{\linspace}{\setV}				% linear space
\safemath{\clinspace}{(\linspace,\sfield)} % linear space
\safemath{\nspace}{\setU}				% normed space
\safemath{\metspace}{\setM}				% metric space
\safemath{\bspace}{\setB}				% Banach space
\safemath{\ipspace}{\setG}				% inner product space
\safemath{\hilspace}{\setH}				% Hilbert space
\safemath{\blospace}{\setG}				% set of bounded linear oprators
\safemath{\lop}{\opT}					% linear operator
\safemath{\altlop}{\opS}					% alternative linear operator
\safemath{\nullsp}{\nullspace(\lop)}		% null space of the linear operator
\safemath{\lfun}{l}						% linear functional
\safemath{\altlfun}{g}					% alternative linear functional
\newcommand{\dual}[1]{#1^{'}}			% dual space
\safemath{\dsum}{\oplus}					% direct sum
\safemath{\funspace}{\colL}				% function space
\renewcommand{\adj}[1]{#1^{\times}}		% adjoint operator generator
\safemath{\adjlop}{\adj{\lop}}			% adjoint operator
\safemath{\hadjlop}{\hadj{\lop}}			% Hilbert adjoint operator
\safemath{\tow}{\xrightarrow{w}}			% weak convergence
\safemath{\tows}{\xrightarrow{w^{*}}}		% weak* convergence
\safemath{\cparam}{\lambda}
\safemath{\lopl}{\lop_{\cparam}}		
\safemath{\iop}{\opI}					% identity operator
\safemath{\resolop}{\opR}				% resolvent operator
\safemath{\resolvent}{\resolop_{\cparam}(\lop)}	% resolvent operator
\safemath{\reset}{\setQ}
\safemath{\spectrum}{\setS}
\safemath{\resolset}{\reset(\lop)}		% resolvent set
\safemath{\lopspec}{\spectrum(\lop)}		% spectrum of a linear operator
\safemath{\pspec}{\spectrum_{p}(\lop)}	% point spectrum
\safemath{\cspec}{\spectrum_{c}(\lop)}	% continuous spectrum
\safemath{\rspec}{\spectrum_{r}(\lop)}	% residual spectrum
\safemath{\ev}{\cparam}					% eigenvalue
\newcommand{\specrad}[1]{r_{#1}}			% spectral radius
\safemath{\lopsrad}{\specrad{\lop}}		% spectral radius
\safemath{\pop}{\opP}					% projection operator

\safemath{\specfam}{\colE}				% spectral family
\safemath{\specop}{\opE_{\cparam}}		% spectral projection operator
\safemath{\altspecop}{\opE_{\mu}}		% alternat spectral projection operator
\safemath{\vmulti}{\vecone}				% vector multiplicative identity
\safemath{\unitaryop}{\opU}				% unitary operator
\safemath{\sval}{\sigma}					% singular value
\safemath{\corrcoef}{\rho}				% canonical correlation coefficient
\safemath{\sangle}{\theta}				% angle between subspaces

\let\time\undefined
\safemath{\iset}{\setI}				% index set
\safemath{\shift}{\nu}
\safemath{\scale}{\alpha}
\safemath{\time}{t}
\safemath{\specfreq}{\theta}	
\newcommand{\transopgen}[1]{\opT_{#1}} % translation operator
\safemath{\transop}{\transopgen{\delay}}
\newcommand{\modopgen}[1]{\opM_{#1}}	% modulation operator
\safemath{\modop}{\modopgen{\shift}}
\newcommand{\dilopgen}[1]{\opD_{#1}}	% dilation operator
\safemath{\dilop}{\dilopgen{\scale}}
\safemath{\fram}{\setF}				% frame
\safemath{\dfram}{\dual{\fram}}		% dual frame
\safemath{\ufb}{B}					% upper frame bound
\safemath{\lfb}{A}					% lower frame bound
\safemath{\sop}{\hadj{\aop}}				% frame synthesis operator
\safemath{\aop}{\opT}			% frame analysis operator 		// modifide by christoph bunte
\safemath{\fop}{\opS}				% frame operator 			// modified by christoph bunte 
\safemath{\daop}{\tilde\opT}			% dual frame analysis operator 		// added by christoph bunte
\safemath{\dsop}{\hadj{\tilde\opT}}				% dual frame synthesis operator 			// added by christoph bunte 

\safemath{\ifop}{\inv{\fop}}			% inverse frame operator
\safemath{\rifop}{\fop^{-1/2}}			% square root of inverse frame operator
\safemath{\transeq}{\setT}			% sequence of translates
\safemath{\nfun}{\Phi}				% ``Nyquist series''
\safemath{\funvec}{\vecf}			%  function as a vector
\safemath{\altfunvec}{\vecg}

\safemath{\samplespace}{\Omega}
\safemath{\probspace}{(\samplespace,\sfield,\Prob)}	% probability space
\safemath{\ccoef}{\rho}			% correlation coefficient
\safemath{\infstate}{\vecpi}				% steady state vector
\safemath{\typset}{\setA_{\epsilon}^{(N)}}	% typical set
\safemath{\expequal}{\doteq}				% equal to first order in the exponent
\safemath{\code}{C}						% code
\safemath{\dstringset}{\setD^{\star}}		% set of finite length D-ary strings
\safemath{\cwlength}{l}					% codeword length
\safemath{\elength}{L}					% expected codeword length
\safemath{\extension}{C^{\star}}			% code extension
\safemath{\approaches}{\rightarrow}		% i.e., x_{n} -> x
\safemath{\evnt}{\setA}					% event A
\safemath{\altevnt}{\setB}					% event B
\safemath{\rv}{\rndx}					% random variable X
\safemath{\altrv}{\rndy}					% random variable Y
\safemath{\complexrv}{\rndu}					% complex random variable U
\safemath{\altcrv}{\rndv}				% complex random variable V
\safemath{\rvec}{\rvecx}					% random vector X
\safemath{\altrvec}{\rvecy}				% random vector Y
\safemath{\crvec}{\rvecu}				% complex random vector U
\safemath{\altcrvec}{\rvecv}				% complex random vector V
\safemath{\variance}{\sigma^{2}}			% variance
\safemath{\map}{T}						% mapping
\safemath{\jacobian}{\matJ}					% jacobian
\safemath{\wvec}{\rvecw}					% white random vector
\safemath{\wrv}{\rndw}					% white noise process
\safemath{\orthmat}{\matQ}				% orthogonal matrix
\safemath{\evmat}{\matLambda}			% eigenvalue matrix (diagonal)
\safemath{\identity}{\matidentity}		% identity matrix
\safemath{\innovec}{\vecv}				% innovations vector
\safemath{\convas}{\xrightarrow{\text{a.s.}}}	% almost sure convergence
\safemath{\convr}{\xrightarrow{\text{r}}}	% convergence in r-th mean
\safemath{\convp}{\xrightarrow{\text{P}}}	% convergence in probability
\safemath{\convd}{\xrightarrow{\text{D}}}	% convergence in distribution
\safemath{\ltis}{\opL}				% LTI system
\safemath{\ir}{h}					% impulse response
\safemath{\tf}{\MakeUppercase{\ir}}	% transfer function
\theoremstyle{definition}
\newtheorem{thm}{Theorem}[section]
\newtheorem{lem}{Lemma}[section]
\newtheorem{prp}{Proposition}[section]
\newtheorem{cor}{Corollary}[section]

\newtheorem{exa}{Example}[section]
\newtheorem{dfn}{Definition}[section]

\usepackage{cancel}
\usepackage{enumitem}
\usepackage[utf8]{inputenc}

\allowdisplaybreaks[3]
\usetikzlibrary{cd,matrix,patterns,calc,decorations.fractals,arrows,external,pgfplots.colormaps}
\definecolor{tblblue}{rgb}{0.93,0.93,1.0}
\definecolor{tblred}{rgb}{1,0.93,0.93}
\definecolor{darkblue}{rgb}{0,0,0.7} 
\definecolor{darkgreen}{RGB}{20,120,43} 
\definecolor{darkred}{rgb}{0.8,0,0} 
\definecolor{lightblue}{RGB}{101,124,191}
\definecolor{skyblue}{RGB}{135,206,235}
\definecolor{gold}{RGB}{204,168,66}
\definecolor{strongblue}{RGB}{60,146,228}
\definecolor{lightgray}{gray}{0.5}
\definecolor{verylightgray}{RGB}{101,124,191}
\definecolor{mistyrose}{RGB}{238,213,210}
\definecolor{firebrick3}{RGB}{205,38,38}
\newcommand{\mydots}{\dots\,}

\author{\IEEEauthorblockN{
Giovanni Alberti,
Helmut B\"olcskei,
Camillo De Lellis,\\ 
G\"unther Koliander,   
and 
Erwin Riegler
}
\thanks{G. Alberti is with the Dept.~of~Mathematics, University of Pisa, Italy (e-mail: galberti1@dm.unipi.it).}
\thanks{H. B\"olcskei is with the Dept.~of~IT~\&~EE and the Dept. of Mathematics, ETH Zurich, Switzerland (e-mail: hboelcskei@ethz.ch).}
\thanks{C. De Lellis is with the IAS, Princeton, NJ, USA and the University of Zurich, Switzerland (e-mail: camillo.delellis@math.ias.edu).}
\thanks{G. Koliander is with the ARI, Austrian Academy of Sciences, Vienna, Austria (e-mail: guenther.koliander@kfs.oeaw.ac.at). His work was supported by the Vienna Science and Technology Fund (WWTF): MA16-053.}
\thanks{E. Riegler is with the Dept.~of~IT~\&~EE, ETH Zurich, Switzerland (e-mail:  eriegler@mins.ee.ethz.ch).}
\thanks{The material in this paper was presented in part at the 2016 IEEE International Symposium on Information Theory \cite{albdekori16}.}
}

\begin{document}

\title{Lossless Analog Compression}
\maketitle

\begin{abstract}
We establish the fundamental limits of lossless analog compression by considering the recovery of arbitrary random vectors $\rvecx\in\reals^m$ from the noiseless linear measurements $\rvecy=\matA\rvecx$ with measurement matrix $\matA\in\reals^{n\times m}$. Our theory is inspired by the groundbreaking work of Wu and Verd\'u (2010) on almost lossless analog compression, but applies to  the  nonasymptotic, i.e., fixed-$m$ case, and considers zero error probability. Specifically, our achievability result states that, for Lebesgue-almost all $\matA$, the random vector $\rvecx$ can be recovered with zero error probability provided that $n>K(\rvecx)$, where $K(\rvecx)$ is given by the infimum of the lower modified Minkowski dimension over all support sets $\setU$ of $\rvecx$ (i.e., sets $\setU\subseteq\reals^m$ with $\opP[\rvecx\in\setU]=1$). We  then particularize this achievability result to the class of $s$-rectifiable random vectors as introduced in Koliander \emph{et al.} (2016); these  are  random vectors of absolutely continuous distribution---with respect to  the $s$-dimensional Hausdorff measure---supported on countable unions of $s$-dimensional $C^1$-submanifolds of $\reals^m$. Countable unions of $C^1$-submanifolds include essentially all signal models used in the compressed sensing literature such as the standard union of subspaces model underlying much of compressed sensing theory and spectrum-blind sampling, smooth submanifolds, block-sparsity, and low-rank matrices as considered in the  matrix completion problem. Specifically, we prove that, for Lebesgue-almost all $\matA$, $s$-rectifiable random vectors $\rvecx$ can be recovered with zero error probability from $n>s$ linear measurements. This threshold is, however, found not to be tight as exemplified by the construction of an $s$-rectifiable random vector that can be recovered with zero error probability from $n<s$ linear measurements. Motivated by this observation, we introduce the new class of $s$-analytic random vectors, which admit a strong converse in the sense of $n\geq s$ being necessary for recovery with probability of error smaller than one. The central conceptual tools in the development of our theory are geometric measure theory and the theory of real analytic functions.  
\end{abstract}


\section{Introduction}\label{sec:Introduction}

Compressed sensing 
%\cite{dost89,grni03,elbr02,dohu01,doel03,tr04,do06,carota06,cata06,ca08,cawa08,fora13} 
\cite{dost89,do06,carota06,caelnera11,kudubo12} 
%grni03,elbr02,dohu01,doel03,tr04,,carota06,cata06,ca08,cawa08,fora13
deals with the recovery of unknown sparse vectors $\vecx\in\reals^m$ from a small (relative to $m$) number, $n$, of linear measurements of the form $\vecy=\matA\vecx$, where $\matA\in\reals^{n\times m}$ is  the measurement matrix.\footnote{Throughout the paper, we assume that $n\leq m$.} 
Known recovery guarantees  can be  
categorized  as deterministic, probabilistic, and information-theoretic.  The literature in all three categories is abundant and the ensuing overview is hence necessarily highly incomplete, yet representative. 

Deterministic results, such as  those in \cite{dost89,grni03,elbr02,doel03, tr04,kudubo12}, are uniform in the sense of applying to all $s$-sparse vectors $\vecx\in\reals^m$,  i.e.,  vectors $\vecx$ %having at most $s$ nonzero entries and are, therefore, 
that are supported on 
%that are supported on 
a finite  union of   $s$-dimensional linear subspaces of $\reals^m$,  
and  for a fixed  measurement matrix $\matA$.  Typical  recovery guarantees  say that   $s$-sparse vectors $\vecx$ can be recovered  through convex optimization algorithms or greedy algorithms  provided that \cite[Chapter 3]{el10}
$s<\frac{1}{2}(1+1/\mu)$, where $\mu$ denotes the coherence of $\matA$, i.e., the largest (in absolute value) inner product of any two different columns of $\matA$. 
The Welch bound \cite{we74} implies that the minimum number of linear measurements, $n$, required for uniform recovery is of order $s^2$, a result known as the  ``square-root bottleneck''  all coherence-based recovery thresholds suffer from.

Probabilistic results are either based on random 
%\footnote{Random quantities are denoted by sans-serif letters, throughout.}
measurement matrices $\rmatA$  (\!\!\cite{carota06,do06,cata06,tr08,pobrst13,badarowa08,caelnera11}) 
or deterministic $\matA$ and random $s$-sparse vectors $\rvecx$ (\!\!\cite{tr08,pobrst13}),
and typically state that $s$-sparse vectors can be recovered, again using convex optimization algorithms or greedy algorithms, with high probability, provided that $n$ is of order $s\log m$. 
%Thresholds of the same order can also be obtained by taking the measurement matrix $\matA$ deterministic and randomizing the sparse vector \cite{tr08,pobrst13}.  

An information-theoretic framework for compressed sensing,  fashioned as an almost lossless analog compression problem, was developed by Wu and Verd\'u \cite{wuve10,wuve12}. 
Specifically, \cite{wuve10} derives asymptotic (in $m$) achievability results and converses for linear encoders and measurable decoders, measurable encoders and Lipschitz continuous decoders, and continuous encoders and continuous decoders. 
For the particular case of linear encoders and measurable decoders,  \cite{wuve10}
shows that, asymptotically in $m$, 
for Lebesgue almost all (a.a.) measurement matrices $\matA$, the random  vector $\rvecx$ can be recovered with arbitrarily small probability of error from $n=\lfloor R m\rfloor$ linear measurements, provided that $R>R_\mathrm{B}$, where $R_\mathrm{B}$ denotes the Minkowski dimension compression rate \cite[Definition 10]{wuve10} of the random process generating $\rvecx$. 
For the special case of $\rvecx$ with independent and identically distributed (i.i.d.)  discrete-continuous mixture entries, a matching converse  exists.  
%$R\geq R_\mathrm{B}$ being necessary for the existence of a measurement matrix $\matA$ such that $\rvecx$ can be recovered with probability of error strictly smaller than one for $m$ sufficiently large. 
Discrete-continuous mixture distributions $\rho\mu^\mathrm{c}+(1-\rho)\mu^\mathrm{d}$ are relevant as 
%$\lfloor\rho m\rfloor$---by the law of large numbers---can be interpreted as the sparsity level of $\rvecx$ and $R_\mathrm{B}=\rho$.  
they mimic sparse vectors for large $m$.  
In particular, if the discrete part $\mu^\mathrm{d}$ is a Dirac measure at $0$, then the nonzero entries of $\rvecx$ can be generated only by the continuous part  $\mu^\mathrm{c}$, and the fraction of nonzero entries in $\rvecx$ converges---in probability---to $\rho$ as $m$ tends to infinity. 
A nonasymptotic, i.e., fixed-$m$, statement in  \cite{wuve10} says that a.a. (with respect to a  $\sigma$-finite Borel measure on $\reals^m$) $s$-sparse random vectors can be recovered with zero error probability  provided that $n>s$. 
Again, this result holds for Lebesgue a.a. measurement matrices $\matA\in\reals^{n\times m}$. A corresponding converse does not seem to be available. For recent work on the connection  between lossy  compression of stochastic processes under distortion constraints and mean dimension theory for dynamical systems, we refer the interested reader to \cite{lits18,veve18,gusp18}.

\emph{Contributions.} We establish the fundamental limits of lossless, i.e., zero error probability, analog compression in the nonasymptotic, i.e., fixed-$m$,  regime for arbitrary random vectors $\rvecx\in\reals^m$.  
Specifically, we show that $\rvecx$ can be recovered with zero error probability provided that $n>K(\rvecx)$, with  $K(\rvecx)$ given by the infimum of the lower modified Minkowski dimension over all support sets $\setU$ of $\rvecx$, i.e., all sets $\setU\subseteq\reals^m$ with $\opP[\rvecx\in\setU]=1$. 
This statement holds for Lebesgue-a.a. measurement matrices. 
Lower modified Minkowski dimension vastly generalizes the notion of  $s$-sparsity, and allows for arbitrary support sets that are not necessarily finite unions of $s$-dimensional linear subspaces. For $s$-sparse vectors, we get the recovery guarantee $n>s$ showing that our information-theoretic thresholds  suffer neither  from the square-root bottleneck \cite{we74}  nor from a $\log m$-factor \cite{tr08,pobrst13}. 
We hasten to add, however, that we do not 
specify explicit decoders that achieve these thresholds, rather we provide 
existence results  absent computational considerations. 
The central conceptual element in the proof of our achievability result is the probabilistic null-space property first reported in \cite{striagbo15}. We emphasize that it is the usage of modified Minkowski dimension, as opposed to  Minkowski dimension  \cite{wuve10,striagbo15}, that allows us to obtain achievability results for zero  error probability. 
The asymptotic achievability result for  linear encoders in \cite{wuve10} can be recovered in our framework.  

We  particularize our achievability result to $s$-rectifiable random vectors $\rvecx$ as introduced in \cite{kopirihl15}; these are  random vectors 
supported on  countable unions of $s$-dimensional $C^1$-submanifolds of $\reals^m$ and of  
absolutely continuous---with respect to  $s$-dimensional Hausdorff measure---distribution.  
Countable unions of $C^1$-submanifolds  include numerous signal models prevalent in the compressed sensing literature, namely, the standard union of subspaces model underlying much of compressed sensing theory \cite{elmi09,baceduhe10} and spectrum-blind sampling \cite{febr96,miel09}, smooth submanifolds \cite{bawa09}, block-sparsity \cite{stpaha09,elkubo10,huzh10}, and  low-rank matrices as considered in the  matrix completion problem \cite{care09,capl11,ristbo15}. 
Our achievability result shows that $s$-rectifiable random vectors can be recovered with zero error probability provided that $n>s$. Again, 
this statement holds for Lebesgue-a.a. measurement matrices. Absolute continuity   with respect to  $s$-dimensional Hausdorff  measure is a regularity condition ensuring  that the distribution %of an $s$-rectifiable random vector 
is not  too concentrated; in particular,  sets of Hausdorff dimension $t<s$ are guaranteed to carry zero probability mass. % with respect to the distribution of an $s$-rectifiable random vector.    
%We provide several examples of $s$-rectifiable random vectors with interesting structural properties. 
One would therefore expect $n\geq s$ to be necessary for zero error recovery of $s$-rectifiable random vectors. 
It turns out, however, that, perhaps surprisingly, this is  not  the case in general. 
An example elucidating this phenomenon constructs  a set $\setG\subseteq\reals^2$ of positive $2$-dimensional Hausdorff measure  that can be compressed linearly in a one-to-one fashion into $\reals$. 
This will then be seen to lead to the statement that every $2$-rectifiable random vector of distribution  absolutely continuous with respect to   $2$-dimensional Hausdorff measure restricted to $\setG$ can be recovered 
with zero error probability from a single  linear measurement.  
What renders this result  surprising is that all this is possible even though   $\setG$ contains 
the image of a Borel set in $\reals^2$ of positive  Lebesgue measure under a $C^\infty$-embedding. 
The picture changes completely when the embedding is real analytic. 
Specifically, we show that if a set $\setU\subseteq\reals^m$ contains 
the  real analytic embedding of a Borel set in $\reals^s$ of positive  Lebesgue measure,  it cannot be compressed linearly (in fact, not even  through a nonlinear real analytic mapping) in a one-to-one fashion into $\reals^n$ with $n<s$. This leads to the new concept of $s$-analytic random vectors,  which allows  a strong converse in the sense of $n\geq s$ being necessary for recovery of $\rvecx$  with probability of error smaller than one. The qualifier ``strong''  refers to the fact that recovery from 
$n< s$ linear measurements is not possible even if we allow an arbitrary positive error probability  strictly smaller than one. 
The only strong converse available in the literature applies to random vectors $\rvecx$ with i.i.d. discrete-continuous mixture entries \cite{wuve10}. 
%We note that the construction of the $2$-rectifiable random vector mentioned above is highly nontrivial and all the other $s$-rectifiable random vectors  
%in our examples satisfy the regularity condition of being also  $s$-analytic. 

\emph{Organization of the paper.}
In Section \ref{sec.ach}, we present our achievability results, with the central statement  in Theorem \ref{th1}. 
Section \ref{sec:rectifiable} particularizes these  results to $s$-rectifiable random vectors, and presents 
an example of a $2$-rectifiable random vector that can be recovered from a single linear measurement with zero error probability. 
%the   $2$-rectifiable example random vectors that can be recovered with zero error probability 
%from a single  linear measurement.  
%The main technical contribution here is Proposition  \ref{prpcounter}, a preparatory result for the construction of this example. 
In Section \ref{sec:analytic}, we introduce  and characterize the new class of $s$-analytic random vectors and we derive a corresponding  strong converse, stated in Theorem \ref{thm.converse}. 
Sections  \ref{th1.proof}--\ref{proof.converse}  contain the proofs of the main technical results stated in Sections  \ref{sec.ach}--\ref{sec:analytic}. 
%Theorem \ref{th1},  Proposition \ref{prpcounter}, and Theorem  \ref{thm.converse}, respectively. 
Appendices \ref{ProofLemmaExarec}--\ref{auxlemma1proof} contain   proofs of further technical results stated in the main body of the paper. 
In Appendices \ref{app:geo}--\ref{app:rea}, we summarize concepts and basic  results from (geometric) measure theory, the theory of set-valued functions,  sequences of functions in several variables, and  real analytic mappings, all needed throughout the paper. 
The reader not familiar with these results is advised to consult  the  corresponding appendices before studying the proofs of our main results.  
These appendices also contain new results, which are highly specific and would disrupt the flow of the paper if presented in the main body.

\emph{Notation.} 
We use capital boldface roman letters $\matA,\matB,\dots$  to denote deterministic matrices and 
lower-case boldface roman letters $\veca,\vecb,\dots$ to designate  deterministic vectors. 
Random matrices and vectors are set in sans-serif font, e.g., $\rmatA$ and $\rvecx$. 
%We use  letters, e.g.\ , for random quantities and roman letters, e.g. $\matA,\vecx$, for deterministic quantities. 
%$\rvecx\in\reals^{m}$ emphasizes that $\rvecx$ takes values in $\reals^m$. 
The $m\times m$ identity matrix is denoted by $\matI_m$. 
We write $\rank(\matA)$ and $\ker(\matA)$ for the rank and the kernel of $\matA$, respectively. 
%For the matrix $\matA$, $\spark(\matA)$ is   the smallest number $k$ such that there exists a set of $k$ column vectors  of $\matA$ that are linearly dependent. 
The superscript  $\tp{}$ stands for transposition. 
The $i$-th unit vector  is denoted by $\vece_i$.
For a vector $\vecx\in\reals^m$, 
$\|\vecx\|_2=\sqrt{\tp{\vecx}\vecx}$ is its Euclidean norm  and $\|\vecx\|_0$ denotes the number of nonzero entries of $\vecx$. 
For the set $\setA$, we write  $\operatorname{card}(\setA)$ for its cardinality, $\overline{\setA}$ for its closure,  
$2^{\setA}$ for its power set, and 
$\ind{\setA}$ for the indicator function on $\setA$. 
With $\setA\subseteq\reals^m$ and $\setB\subseteq\reals^n$, we let   
$\setA\times\setB=\{(\veca, \vecb):\veca\in\setA, \vecb\in\setB\}$ 
 and 
$\setA\otimes \setB=\{\veca\otimes\vecb:\veca\in\setA, \vecb\in\setB\}$, where $\otimes$ denotes the Kronecker product (see \cite[Definition 4.2.1]{hojo91}).  
For $\setA,\setB\subseteq\reals^m$, we write $\setA\subsetneq\setB$ to express  strict inclusion according to    
$\setA\subseteq\setB$ with $\setA\neq\setB$ and we let $\setA-\setB=\{\veca-\vecb:\veca\in\setA,\vecb\in\setB\}$.   
We set $\overline{\reals}=\reals\cup\{-\infty,\infty\}$. 
For the Euclidean space $(\reals^k,\|\cdot\|_2)$, we designate the open ball of radius $\rho$ centered at $\vecu\in \reals^k$ by $\setB_k(\vecu,\rho)$.
We write $\colS(\setX)$ for a general $\sigma$-algebra on $\setX$, $\colB(\setX)$ for the Borel $\sigma$-algebra on a topological space $\setX$, 
and $\colL(\reals^m)$ for the Lebesgue $\sigma$-algebra on $\reals^m$.   
The product $\sigma$-algebra of $\colS(\setX)$ and $\colS(\setY)$ is denoted by $\colS(\setX)\otimes\colS(\setY)$.  
For  measures $\mu$ and $\nu$ on the same measurable space, we denote absolute continuity of  $\mu$ with respect to $\nu$ by $\mu\ll\nu$. 
We write $\mu\times\nu$ for the product measure of $\mu$ and $\nu$.  
Throughout we assume, without loss of generality (w.l.o.g.), that a measure on a measurable space is defined on all subsets of the measurable space (see \cite[Remark 1.2.6]{krpa08}).   
The Lebesgue measure on $\reals^k$ and $\reals^{k\times l}$ is designated as $\lebmeasure^{k}$ and  $\lebmeasure^{k\times l}$, respectively. 
The distribution of a random vector $\rvecx$ is denoted by $\mu_\rvecx$. 
If $f\colon\reals^k\to\reals^l$ is differentiable, we write $D\!f(\vecv)\in\reals^{l\times k}$ for its differential at $\vecv\in\reals^k$ and  define the $\min\{k,l\}$-dimensional Jacobian $J\! f(\vecv)$ at $\vecv\in\reals^k$ by 
\begin{align}
J\! f(\vecv)=
\begin{cases}\label{eq:Jacobiandef}
\sqrt{\det(\tp{(Df(\vecv))}Df(\vecv))}& \text{if $l\geq k$}\\ 
\sqrt{\det(Df(\vecv)\tp{(Df(\vecv))})}& \text{else.}
\end{cases}
\end{align}
For an open set $\setU\subseteq\reals^k$, a differentiable mapping $f\colon\setU\to\reals^l$, where $l\geq k$, is called an  immersion if 
$J\! f(\vecv)>0$ for all $\vecv\in\setU$. 
A one-to-one immersion is referred to as an embedding. 
For a mapping $f$, we write $f\equiv \veczero$  if it is identically zero and 
 $f\not\equiv \veczero$  otherwise. 
For $f\colon\setU\to\setV$ and $g\colon\setV\to\setW$, the composition $g\circ f\colon\setU\to\setW$ is defined as $(g\circ f)(x)=g(f(x))$ for all $x\in\setU$. 
For $f\colon\setU\to\setV$ and $\setA\subseteq \setU$, $f|_\setA$ denotes the restriction of $f$ to $\setA$. 
For $f\colon\setU\to\setV$ and $\setB\subseteq \setV$, we set $f^{-1}(\setB)=\{x\in\setU:f(x)\in\setB\}$. 

\section{Achievability}\label{sec.ach}

In classical compressed sensing theory \cite{do06,carota06,caelnera11,kudubo12},  one typically deals with the recovery of $s$-sparse vectors $\vecx\in\reals^m$, i.e.,  vectors $\vecx$ that are supported on a finite  union of   $s$-dimensional linear subspaces of $\reals^m$. %  obtained by setting $m-s$ coordinates equal to zero. 
The purpose of this paper is the development of a  comprehensive theory of signal recovery in the sense of allowing arbitrary support sets $\setU$, which  are not necessarily unions of (a finite number of)  linear subspaces of $\reals^m$. 
Formalizing this idea requires a suitable dimension measure for general nonempty sets. There is a rich variety of dimension measures available in the literature \cite{ro11,fa14,kade94}.  
Our choice  will be guided by the requirement of information-theoretic operational significance. Specifically, the dimension measure should allow the formulation of  nonasymptotic, i.e., fixed-$m$, 
achievability results with 
zero error probability. % for arbitrary random vectors $\rvecx\in\reals^m$. 
The modified  Minkowski dimension will turn out to meet  these requirements.  

We first recall the definitions of   Minkowski dimension and modified Minkowski dimension, compare the two concepts, and state the basic properties  of modified Minkowski dimension needed in the remainder of the paper.

\begin{dfn}(Minkowski dimension\footnote{Minkowski dimension is sometimes also referred to as box-counting dimension, which is the origin of the subscript B in the notation $\dim_\mathrm{B}(\cdot)$ used henceforth.})\cite[Equivalent definitions 2.1]{fa14} \label{dfndim}
For $\setU\subseteq \reals^{m}$ nonempty, the lower and upper Minkowski dimension of $\setU$ is defined as 
\begin{align}
\underline{\dim}_\mathrm{B}(\setU)&=\liminf_{\rho\to 0} \frac{\log N_\setU(\rho)}{\log \frac{1}{\rho}}\label{eq:lowerB}
\end{align} 
and
\begin{align} 
\overline{\dim}_\mathrm{B}(\setU)&=\limsup_{\rho\to 0} \frac{\log N_\setU(\rho)}{\log \frac{1}{\rho}},\label{eq:upperB}
\end{align} 
respectively, where 
\begin{align}\label{eq:coveringnumber}
N_\setU(\rho)=\min\Big\{k \in\naturals : \setU\subseteq\hspace*{-4truemm} \bigcup_{i\in\{1,\dots,k\}}\hspace*{-4truemm} \setB_{m}(\vecu_i,\rho),\ \vecu_i\in \reals^{m}\Big\}
\end{align}
is the covering number of $\setU$ for radius $\rho$.
If $\underline{\dim}_\mathrm{B}(\setU)=\overline{\dim}_\mathrm{B}(\setU)$, this common value, denoted by  $\dim_\mathrm{B}(\setU)$, is   
the Minkowski dimension of $\setU$.
\end{dfn}

\begin{dfn}(Modified Minkowski dimension)\cite[p. 37]{fa14} \label{dfndimlocal} 
For $\setU\subseteq \reals^{m}$ nonempty, the lower and upper modified Minkowski dimension of $\setU$ is defined as 
\begin{align}\label{eq:MBlower}
\underline{\dim}_\mathrm{MB}(\setU)=\inf\mleft\{\sup_{i\in\naturals} \underline{\dim}_\mathrm{B}(\setU_i) : \setU\subseteq \bigcup_{i\in\naturals}\setU_i\mright\}
\end{align} 
and
\begin{align} \label{eq:MBupper}
\overline{\dim}_\mathrm{MB}(\setU)=\inf\mleft\{\sup_{i\in\naturals} \overline{\dim}_\mathrm{B}(\setU_i) : \setU\subseteq \bigcup_{i\in\naturals}\setU_i\mright\},
\end{align}
respectively, where in \eqref{eq:MBlower}  and \eqref{eq:MBupper}  the infima are over all possible  coverings $\{\setU_i\}_{i\in\naturals}$ of $\setU$ by nonempty bounded sets $\setU_i$. 
If $\underline{\dim}_\mathrm{MB}(\setU)=\overline{\dim}_\mathrm{MB}(\setU)$, this common value, denoted by  $\dim_\mathrm{MB}(\setU)$, is   
the modified Minkowski dimension of $\setU$.
\end{dfn}

The main properties of   modified Minkowski dimension are summarized in Lemma \ref{lem:Mpropmindim}. 
In particular, 
\begin{align}\label{eq:ineqMB1}
\underline{\dim}_\mathrm{MB}(\cdot)\leq \underline{\dim}_\mathrm{B}(\cdot)
\end{align}
and 
\begin{align}\label{eq:ineqMB2}
\overline{\dim}_\mathrm{MB}(\cdot)\leq \overline{\dim}_\mathrm{B}(\cdot).
\end{align}
Both lower and upper modified Minkowski dimension have  the advantage of being countably stable, a key property we will use frequently. 
In contrast, upper Minkowski dimension is only finitely stable, and lower Minkowski dimension is not even finitely stable (see \cite[p. 34]{fa14}). 
For example, all countable subsets of $\reals^{m}$ have  modified Minkowski dimension equal to zero (the Minkowski dimension  of a single point in $\reals^{m}$ equals zero),    
but there exist infinitely countable sets with nonzero Minkowski dimension: 
\begin{exa}\cite[Example 2.7]{fa14}\label{exa:MBB}
Let $\setF=\{0,1/2,1/3,\dots\}$. Then, ${\dim}_\mathrm{MB}(\setF)=0<{\dim}_\mathrm{B}(\setF)=1/2$.
\end{exa}

Minkowski dimension and modified Minkowski dimension also behave  differently for unbounded sets. 
Specifically, by  monotonicity of (upper) modified Minkowski dimension,  
$\underline{\dim}_\mathrm{MB}(\setA)\leq \overline{\dim}_\mathrm{MB}(\setA)\leq{\dim}_\mathrm{MB}(\reals^m)= m$ for all $\setA\subseteq\reals^m$, in particular also for unbounded sets, whereas $\underline{\dim}_\mathrm{B}(\setA)=\overline{\dim}_\mathrm{B}(\setA)=\infty$ for all unbounded sets $\setA$ as a consequence of $N_\setA(\rho)=\infty$ for all $\rho\in(0,\infty)$. 
Working with lower modified Minkowski dimension will allow us to consider  arbitrary random vectors, regardless of whether they admit bounded support sets or not. 

The following example shows that  the modified Minkowski dimension agrees with the  sparsity notion   used 
%for the union of linear subspaces model 
in classical compressed sensing theory.

\begin{exa}\label{exa:subspace}
For $\setI$ finite or countable infinite, let $\setT_i$, $i\in\setI$,  be linear subspaces with their Euclidean dimensions $\dim(\setT_i)$ satisfying 
\begin{align}\label{eq:union1}
\max_{i\in\setI}\dim(\setT_i)=s,
\end{align}
and consider the union of subspaces 
\begin{align}\label{eq:union2}%\label{eq:setS}\label{eq:union1}
\setU&=\bigcup_{i\in\setI}\setT_i.%\{\matB\vecz : \vecz\in\reals^k, \|\vecz\|_0\leq s\}. 
\end{align} 
%Let $\matB\in\reals^{m\times k}$  with $\spark(\matB)>s$ and  
%consider the set  of vectors 
%that are $s$-sparse in a certain  basis $\matB\in\reals^{m\times m}$, i.e.,  
%For $s\in\{1,\dots,m\}$, set   
 %The set $\setU^\matB_s$  is the finite  union  of  $s$-dimensional linear  subspaces. 
As every  linear subspace is a smooth  submanifold of $\reals^m$, it  follows from Properties \ref{Mpropmanifold}  and 
\ref{Mpropcountablestable} of Lemma  \ref{lem:Mpropmindim}  that $\dim_\mathrm{MB}(\setU)=s$. 
In the union of subspaces model, prevalent in compressed sensing theory \cite{do06,carota06,caelnera11,kudubo12}, 
$|\setI|={m\choose s}$ and the subspaces $\setT_i$ correspond to different sparsity patterns, each of cardinality equal to the sparsity level $s$.  
%This model comprises the classical compressed sensing signal model  for $s$-sparse vectors .  

%Finite unions of linear subspaces cover important signal models in classical compressed sensing theory including vectors that are $s$-sparse in the identity basis, i.e., $\|\vecx\|_0\leq s$ \cite{do06,carota06}, or vectors that are $s$-sparse with respect to a (possibly redundant) dictionary $\matB$, i.e., $\|\matB\vecx\|_0\leq s$ with $\matB\in\reals^{m\times l}$ for $l\in\naturals$ and $\rank(\matB)=s$ . 
\end{exa}

The aim of the present paper is to  develop a theory for lossless analog compression of arbitrary  random vectors $\rvecx\in\reals^m$. 
%This requires a suitable definition of the stochastic equivalence of the sparsity level $s$. 
An obvious choice for the stochastic equivalent of the sparsity level $s$ 
is  the stochastic sparsity level $S(\rvecx)$ defined as  %$S(\rvecx)$ of $\rvecx$, defined as the smallest $s\in\naturals_0$ such that 
%$\rvecx$ is $s$-sparse with probability one, that is, 
\begin{align}\label{def:sl}
S(\rvecx)=\min\mleft\{s: \exists\setT_1,\dots,\setT_k\ \text{with}   \opP\mleft[\rvecx\in\bigcup_{i=1}^k\setT_i\mright]=1\ \mright\}, 
\end{align}
where every $\setT_i$ is a linear subspace of $\reals^m$ of dimension $\dim(\setT_i)\leq s$ and $k\in\naturals$.  
This definition is, however, specific to the (finite) union of linear subspaces structure.  
The  theory we develop here requires a more general notion of description complexity,  which we define in terms of the lower modified Minkowski dimension according to  
\begin{align}\label{def:dc}
K(\rvecx)= \inf\{\underline{\dim}_\mathrm{MB}(\setU):\setU\subseteq \reals^{m}, \opP[\rvecx\in\setU]=1\}.
\end{align}
Sets satisfying  $\opP[\rvecx\in\setU]=1$ are hereafter referred to as support sets of $\rvecx$. 
While the definition of $S(\rvecx)$ involves minimization of  Euclidean dimensions of linear subspaces, $K(\rvecx)$ is defined by minimizing the lower modified Minkowski dimension of general support sets. 
Definitions \eqref{def:sl} and \eqref{def:dc} imply directly that $K(\rvecx)\leq S(\rvecx)$ for all random vectors $\rvecx\in\reals^m$ (see Example \ref{exa:subspace}).   
We will see in Section 
\ref{sec:rectifiable} that this inequality can actually be strict. 

Next, we show that application of a locally Lipschitz mapping cannot increase a random vector's 
description complexity. This result will allow us to construct random vectors with low description complexity  out of existing ones simply by applying locally Lipschitz mappings. 
The formal statement is as follows. 

\begin{lem}\label{lem:dic}
Let $\rvecx\in\reals^k$ and $f\colon \reals^k\to \reals^m$  be locally Lipschitz. Then,   $K(f(\rvecx))\leq K(\rvecx)$. 
\end{lem}
\begin{proof}
\begin{align}
&K(f(\rvecx))\\
&= \inf\{\underline{\dim}_\mathrm{MB}(\setV):\setV\subseteq\reals^{m},\, \opP[f(\rvecx)\in\setV]=1\}\\
&= \inf\{\underline{\dim}_\mathrm{MB}(f(\setU)):\setU\subseteq\reals^{k},\, \opP[\rvecx\in \setU]=1\}\\
&\leq \inf\{\underline{\dim}_\mathrm{MB}(\setU):\setU\subseteq\reals^{k},\, \opP[\rvecx\in \setU]=1\} \label{eq:stepLMB}\\
&=K(\rvecx),
\end{align}
where \eqref{eq:stepLMB} follows from Property \ref{MpropL}  of Lemma  \ref{lem:Mpropmindim}.  
\end{proof}
When the mapping $f$ is invertible and both $f$ and $f^{-1}$ are locally Lipschitz, the description complexity remains unchanged: 

\begin{cor}\label{cor:dic}
Let $\rvecx\in\reals^m$ and consider an invertible  mapping $f\colon \reals^m\to \reals^m$. Suppose that $f$ and $f^{-1}$ are both locally Lipschitz. Then, $K(\rvecx)=K(f(\rvecx))$. 
\end{cor}
\begin{proof}
$K(\rvecx)= K((f^{-1}\circ f)(\rvecx))\leq K(f(\rvecx))\leq K(\rvecx)$, where we applied Lemma \ref{lem:dic} twice.    
\end{proof}
As a consequence of Corollary \ref{cor:dic}, the description complexity $K(\rvecx)$ is invariant under a basis change.
% a property not shared by the  stochastic sparsity level $S(\rvecx)$. Thus, 
%$K(\rvecx)=K(\matU\rvecx)\leq S(\matU\rvecx)$ for all regular matrices $\matU\in\reals^{m\times m}$: the random vector's description complexity is no larger than its stochastic sparsity level with respect to any  basis. 
Our main achievability result can now be formulated  as follows. 

\begin{thm}(Achievability)\label{th1}
For $\rvecx\in\reals^{m}$,  $n>K(\rvecx)$  is sufficient for the existence of a 
Borel measurable mapping $g\colon\reals^{n\times m}\times\reals^n\to\reals^m$, 
referred to as  (measurable) decoder, satisfying    
\begin{align}\label{eq:decoder}
\opP\mleft[g(\matA,\matA\rvecx)\neq\rvecx\mright]=0\quad \text{for $\lebmeasure^{n\times m}$-\text{a.a.}$\ \matA\in\reals^{n\times m}$.} 
\end{align}
\end{thm}
\begin{proof}See Section \ref{th1.proof}.\end{proof}

Theorem \ref{th1} generalizes the  achievability result for linear encoders in \cite{wuve10} in the sense of being  nonasymptotic (i.e., it applies for finite $m$)  and guaranteeing zero error probability.   

The central conceptual element in the proof of Theorem~\ref{th1} is the following probabilistic null-space property for arbitrary (possibly unbounded) nonempty sets, first reported in \cite{striagbo15} for  bounded sets and expressed in terms of  lower Minkowski dimension.   
If the lower modified Minkowski dimension of a nonempty  set $\setU$ is smaller than $n$, then,  for $\lebmeasure^{n\times m}$-a.a.  measurement matrices $\matA\in\reals^{n\times m}$, the set $\setU$ intersects the $(m-n)$-dimensional kernel of $\matA$  at most trivially. 
What is remarkable here is that the notions of Euclidean dimension (for the kernel of the linear mapping induced by $\matA$) and of lower modified  Minkowski dimension (for $\setU$) are compatible.  
The formal statement is as follows.  

\begin{prp}\label{prp:ns}  
Let $\setU\subseteq\reals^m$ be   nonempty with $\underline{\dim}_\mathrm{MB}(\setU)<n$. Then,   
\begin{align}\label{eq:proptoshow}
\ker(\matA)\cap(\setU\mysetminus\{\matzero\})=\emptyset\quad \text{for $\lebmeasure^{n\times m}$-\text{a.a.}$\ \matA\in\reals^{n\times m}$.} 
\end{align}
\end{prp}

\begin{proof}   
By  definition of lower modified Minkowski dimension, there exists a covering $\{\setU_i\}_{i\in\naturals}$ of $\setU$ by  nonempty bounded sets $\setU_i$ 
 with  $\underline{\dim}_\mathrm{B}(\setU_i) < n$ for all $i\in\naturals$. 
The countable subadditivity of Lebesgue measure now implies that 
\begin{align}
&\lebmeasure^{n\times m}\mleft\{\matA\in\reals^{n\times m} :\ker(\matA)\cap(\setU\mysetminus\{\matzero\})\neq \emptyset  \mright\}\\
&\leq \sum_{i\in\naturals}\lebmeasure^{n\times m}\mleft\{\matA\in\reals^{n\times m} : \ker(\matA)\cap(\setU_i\mysetminus\{\matzero\})\neq \emptyset \mright\}.\label{eq:sumzero1}
\end{align}
The proof is  concluded by noting that  \cite[Proposition 1]{striagbo15}  with 
$\underline{\dim}_\mathrm{B}(\setU_i)<n$ for all $i\in\naturals$ implies that every term in the sum  of \eqref{eq:sumzero1} equals zero. 
\end{proof}

We close this section by elucidating the level of generality of our theory through particularization of the achievability result Theorem \ref{th1} to random vectors supported on   attractor sets of systems of contractions as defined below.   
%Such sets include  the Cantor sets,   Sierpinski gaskets,  the  modified Koch curves, and many other  fractals.  
 The formal definition is as follows. 
 Let $\setA\subseteq\reals^m$ be closed. For  $i=1,\dots,k$, consider  $s_i\colon \setA\to\setA$  and $c_i\in (0,1)$ such that 
\begin{align}
\|s_i(\vecu)-s_i(\vecv)\|_2\leq c_i\|\vecu-\vecv\|_2\quad \text{for all}\ \vecu,\vecv\in\setA. \label{eq:contractions}
\end{align} 
Such mappings are called contractions. By \cite[Theorem 9.1]{fa14}, there exists a unique compact set $\setK\subseteq\setA$, referred to as an attractor set,  such that 
\begin{align}
\setK=\bigcup_{i=1}^k s_i(\setK). 
\end{align}
Thanks to \cite[Proposition 9.6]{fa14}, $\overline{\dim}_\mathrm{B}(\setK)\leq d$, where $d>0$ is the unique solution of 
\begin{align}\label{eq:defdimC}
\sum_{i=1}^k c_i^d=1. 
\end{align} 
If, in addition, $\setK$ satisfies the open set condition \cite[Equation (9.12)]{fa14}, then  ${\dim}_\mathrm{B}(\setK)= d$ by  \cite[Theorem 9.3]{fa14}. 
The middle-third Cantor set \cite[Example 9.1]{fa14}, Sierpi\'{n}ski gaskets   \cite[Example 9.4]{fa14},  and the  modified von Koch curves \cite[Example 9.5]{fa14} are all attractor sets (with their underlying contractions) that meet the open set condition \cite[Chapter 9]{fa14}. 
As $\underline{\dim}_\mathrm{MB}(\setK)\leq \overline{\dim}_\mathrm{B}(\setK)$ by Property \ref{MpropMltB} of Lemma \ref{lem:Mpropmindim}, every $\rvecx\in\reals^m$ such that  $\opP[\rvecx\in\setK]=1$ has description complexity $K(\rvecx)\leq d$ (see \eqref{def:dc}). 
For an excellent in-depth treatment of attractor sets of systems of contractions, the interested reader is referred to  \cite[Chapter 9]{fa14}.
We finally note that for self-similar distributions \cite[Section 3.E]{wuve10} on attractor sets $\setK$ satisfying the  open set condition \cite[Equation (9.12)]{fa14}, an achievability result in terms  of  information dimension  was reported in  \cite[Theorem 7]{wuve10}.

\section{Rectifiable Sets and Rectifiable Random Vectors}\label{sec:rectifiable}
The  signal models employed in classical compressed sensing \cite{do06,carota06},  model-based compressed sensing \cite{baceduhe10}, and block-sparsity \cite{stpaha09,elkubo10,huzh10}  all fall under the rubric of finite  unions of  linear subspaces.  
More general prevalent signal models in the theory of sparse signal recovery   include finite unions of smooth submanifolds, either in explicit form  as in   \cite{bawa09} or implicitly in the context of low-rank matrix recovery  \cite{care09,capl11,ristbo15}. 
All these models are subsumed by  the   
countable unions of $C^1$-submanifolds structure, formalized next using the notion of rectifiable sets.  
We start with the definition of rectifiable sets. 
 
\begin{dfn}(Rectifiable sets)\cite[Definition 3.2.14]{fed69}\label{Dfn:recset}
Let  $s\in\naturals$, and consider a measure $\mu$  on $\reals^m$.  A nonempty set $\setU\subseteq\reals^{m}$ is 
\begin{enumerate}[label=\roman*)]
\item $s$-rectifiable if there exist a compact set $\setA\subseteq\reals^s$ and a Lipschitz  mapping $\varphi\colon\setA\to \reals^{m}$  such that $\setU=\varphi(\setA)$, \label{recd1}
\item countably $s$-rectifiable if it is the countable union of $s$-rectifiable sets,\label{recd2}
\item countably $(\mu,s)$-rectifiable if it is $\mu$-measurable and there exists a countably $s$-rectifiable set $\setV$ such that $\mu(\setU\mysetminus\setV)=0$.\label{recd3}%, where $\mathscr{H}^{s}$ denotes the $s$-dimensional Hausdorff measure (see Definition \ref{dfnH}),
\item $(\mu,s)$-rectifiable if it is  countably $(\mu,s)$-rectifiable and $\mu(\setU)<\infty$.\label{recd4}
\end{enumerate}
\end{dfn}
Our definitions of $s$-rectifiability and countably $s$-rectifiability
differ from those of \cite[Definition 3.2.14]{fed69} as we require the $s$-rectifiable
set to be the Lipschitz image of a compact rather than a bounded set. 
The more restrictive definitions \ref{recd1} and \ref{recd2} above have
the advantage of $s$-rectifiable sets and countably $s$-rectifiable sets being 
guaranteed to be Borel. (This holds since the image of a compact set under a continuous mapping is a compact set, and a countable union of compact sets is Borel.) 
We note however that, by Lemma \ref{lem:equivrec},  our definitions of  $(\mu,s)$-rectifiable 
and countably  $(\mu,s)$-rectifiable (see Definition \ref{Dfn:recset}, Items \ref{recd3} and \ref{recd4}) are nevertheless equivalent to those in \cite[Definition 3.2.14]{fed69}. 

In what follows, we only need Items \ref{recd3} and \ref{recd4} in Definition \ref{Dfn:recset} for the specific case of $\mu= \mathscr{H}^s$  (see Definition \ref{dfnH}), in which measurability of $\setU$ is guaranteed for all Borel sets.   
 Therefore, $s$-rectifiable sets and countably $s$-rectifiable sets  are also   $\mathscr{H}^{s}$-measurable, 
which leads to the following chain of implications: 
\begin{center} 
$\setU$ is $s$-rectifiable $\Rightarrow$ $\setU$ is countably $s$-rectifiable $\Rightarrow$ $\setU$ is countably $(\mathscr{H}^{s},s)$-rectifiable.
\end{center} 
%Finally, we note that our definition of (countable) $(\mathscr{H}^{s},s)$-rectifiability is equivalent to the one in \cite[Definition 3.2.14]{fed69} as a consequence of.  

The following result  collects properties of (countably) $s$-rectifiable sets for later use.  

\begin{lem} \label{LemmaExarec}$ $
\begin{enumerate}[label=\roman*)]
\item \label{exarec}
If $\setU\subseteq\reals^m$ is $s$-rectifiable, then it is $t$-rectifiable for all $t\in\naturals$  with $t>s$. 
\item \label{exar1}
For  locally Lipschitz $\varphi_i$, $i\in\naturals$, the set
\begin{align} 
\setV=\bigcup_{i\in\naturals}\varphi_i(\reals^s)
\end{align}
is countably $s$-rectifiable. 
\item \label{exaProdrec}
If $\setU\subseteq\reals^m$ is countably  $s$-rectifiable and  $\setV\subseteq\reals^n$ is countably $t$-rectifiable, then 
\begin{align}
\setW=\{\tp{(\tp{\vecu}\ \tp{\vecv})}: \vecu\in\setU, \vecv\in\setV\}\subseteq\reals^{m+n}
\end{align} 
is countably $(s+t)$-rectifiable.  
\item \label{exaC1}
Every $s$-dimensional $C^1$-submanifold \cite[Definition 5.3.1]{krpa08} of $\reals^m$ is countably  $s$-rectifiable. In particular, every $s$-dimensional affine subspace of $\reals^m$ is countably $s$-rectifiable.
\item \label{exasum}
For  $\setA_i$  countably $s_i$-rectifiable and $s_i\leq s$, $i\in\naturals$,  the set
\begin{align}
\setA=\bigcup_{i\in\naturals}\setA_i 
\end{align}
is  countably $s$-rectifiable.  
\end{enumerate}
\end{lem}
\begin{proof} See Appendix \ref{ProofLemmaExarec}.\end{proof}

Countable unions of $s$-dimensional  $C^1$-submanifolds of $\reals^m$ are countably  $s$-rectifiable  by Properties \ref{exaC1} and  \ref{exasum} of Lemma 
 \ref{LemmaExarec}. For countably $(\mathscr{H}^{s},s)$-rectifiable sets we even get an equivalence result, namely: 

\begin{thm}\cite[Theorem 3.2.29]{fed69}\label{thm:C1}
A set $\setU\subseteq\reals^{m}$ is countably $(\mathscr{H}^{s},s)$-rectifiable if and only if $\mathscr{H}^{s}$-a.a. of $\setU$ is contained in the countable union of $s$-dimensional  $C^1$-submanifolds of $\reals^m$. 
\end{thm}

We now show that the upper modified Minkowski dimension of a countably $s$-rectifiable set is upper-bounded by $s$. 
This will allow us to conclude that, for a random vector $\rvecx$ admitting  a countably  $s$-rectifiable support set, $K(\rvecx)\leq s$.   
The formal statement is as follows. 

\begin{lem}\label{lem:mdimrec}
If  $\setU\subseteq\reals^{m}$ is countably $s$-rectifiable, then $\overline{\dim}_\mathrm{MB}(\setU)\leq s$.
\end{lem}
\begin{proof}
Suppose that $\setU\subseteq\reals^{m}$ is countably $s$-rectifiable. Then, 
Definition  \ref{Dfn:recset} implies that there exist nonempty compact sets $\setA_i\subseteq\reals^s$ and Lipschitz mappings $\varphi_i\colon\setA_i\to \reals^{m}$, with  $i\in\naturals$, such that 
\begin{align}\label{eq:recS}
\setU=\bigcup_{i\in\naturals}\varphi_i(\setA_i).
\end{align}
Thus,  
\begin{align}
\overline{\dim}_\mathrm{MB}(\setU)
&=\sup_{i\in\naturals}\overline{\dim}_\mathrm{MB}(\varphi_{i}(\setA_i))\label{eq:step1mbrec}\\
&\leq\sup_{i\in\naturals}\overline{\dim}_\mathrm{MB}(\setA_i)\label{eq:step2mbrec}\\
&\leq\sup_{i\in\naturals}\overline{\dim}_\mathrm{MB}(\reals^s)\label{eq:step3mbrec}\\
&= s, \label{eq:leqs1}
\end{align}
where the individual steps follow from  properties of  Lemma \ref{lem:Mpropmindim}, namely,  
\eqref{eq:step1mbrec} is by Property \ref{Mpropcountablestable}, 
\eqref{eq:step2mbrec} by Property \ref{MpropL2},
\eqref{eq:step3mbrec} by Property \ref{Mpropmonotonic},
and 
\eqref{eq:leqs1}  by Property \ref{Mpropmanifold}.  
\end{proof}

We next  investigate the effect of locally Lipschitz mappings on (countably) $s$-rectifiable and countably $(\mathscr{H}^{s},s)$-rectifiable sets. 

\begin{lem}\label{lem:recsetlip}
Let $\setU\subseteq\reals^m$, and consider a  locally Lipschitz $f\colon\reals^m\to\reals^n$.  If $\setU$ is 
\begin{enumerate}[label=\roman*)]
\item $s$-rectifiable, then $f(\setU)$ is $s$-rectifiable,\label{recf1}
\item countably $s$-rectifiable, then $f(\setU)$ is countably $s$-rectifiable,\label{recf2}
\item countably $(\mathscr{H}^{s},s)$-rectifiable and Borel, then $f(\setU)=\setA\cup\setB$, where $\setA$ is a countably $(\mathscr{H}^{s},s)$-rectifiable  Borel set and 
$\mathscr{H}^{s}(\setB)=0$.\label{recf3}
\end{enumerate}
\end{lem}
\begin{proof}
See Appendix \ref{lem:recsetlipproof}.
\end{proof}
A slightly weaker version of this statement, valid for Lipschitz mappings, was derived previously in \cite[Lemma 4]{kopirihl15}. 
%Lemma \ref{lem:recsetlip} will allow us to construct  nontrivial examples of rectifiable sets  out of given  rectifiable sets  simply by applying locally Lipschitz mappings. 

We are now ready to define rectifiable random vectors. 

\begin{dfn}(Rectifiable random vectors)\cite[Definition 11]{kopirihl15}\label{dfnXrec}
A random vector $\rvecx\in\reals^{m}$ is $s$-rectifiable if there exists a countably $(\mathscr{H}^{s},s)$-rectifiable set $\setU\subseteq \reals^{m}$ such that  $\mu_\rvecx\ll \mathscr{H}^{s}|_\setU$. The corresponding value $s$ is the rectifiability parameter. 
\end{dfn}

It turns out that an  $s$-rectifiable random vector $\rvecx$ always admits a countably $s$-rectifiable support set and, therefore, has description complexity $K(\rvecx)\leq s$ by Lemma \ref{lem:mdimrec}. The formal statement is as follows.

\begin{lem}\label{lem:recsup}
Every $s$-rectifiable random vector $\rvecx\in\reals^m$ admits a  countably $s$-rectifiable support set.  In particular, every  $s$-rectifiable random vector $\rvecx$ has description complexity  $K(\rvecx)\leq s$. 
\end{lem}
\begin{proof}
Suppose that $\rvecx$ is $s$-rectifiable. Then,   there exists a countably $(\mathscr{H}^{s},s)$-rectifiable set $\setU\subseteq\reals^m$ such that 
$\mu_\rvecx\ll\mathscr{H}^{s}|_\setU$. 
As $\setU$ is countably $(\mathscr{H}^{s},s)$-rectifiable, by Definition \ref{Dfn:recset}  there exists a  countably $s$-rectifiable set $\setV\subseteq\reals^m$ such that $\mathscr{H}^{s}(\setU\,\mysetminus\setV)=0$. Set $\setW=(\setU\mysetminus\setV)\cup\setV$ and note that  $\setU\subseteq\setW$ implies  
$\mathscr{H}^{s}|_\setU\ll\mathscr{H}^{s}|_\setW$ by monotonicity of $\mathscr{H}^{s}$. Moreover,  from the definition of $\setW$, the countable additivity of $\mathscr{H}^{s}$, and $\mathscr{H}^{s}(\setU\mysetminus\setV)=0$, it follows that $\mathscr{H}^{s}|_\setW=\mathscr{H}^{s}|_\setV$. Thus,  
$\mathscr{H}^{s}|_\setU\ll\mathscr{H}^{s}|_\setV$, and  
$\mu_\rvecx\ll\mathscr{H}^{s}|_\setU$ implies $\mu_\rvecx\ll\mathscr{H}^{s}|_\setV$. Therefore, as $\mathscr{H}^{s}|_\setV(\reals^m\mysetminus\setV)=0$, 
we can conclude that  
$\opP[\rvecx\in\reals^m\,\mysetminus\setV]=0$, which is equivalent to $\opP[\rvecx\in\setV]=1$, that is, the    countably $s$-rectifiable set $\setV$ is a support set of $\rvecx$. 
The proof is now concluded by noting that $K(\rvecx)\leq\overline{\dim}_\mathrm{MB}(\setV)\leq s$, where the latter inequality is by Lemma \ref{lem:mdimrec}.  
\end{proof}

In light of Theorem \ref{thm:C1}, an $s$-rectifiable random vector $\rvecx\in\reals^m$ is supported on a countable union of $s$-dimensional $C^1$-submanifolds of $\reals^{m}$. Absolute continuity of $\mu_\rvecx$ with respect to the $s$-dimensional Hausdorff measure is a regularity condition  guaranteeing  that $\rvecx$
cannot have positive probability measure on sets of Hausdorff dimension $t<s$ (see Property \ref{Hjump} of Lemma \ref{lem:prophausdorff}). 
This regularity condition together with the sharp transition behavior of Hausdorff measures (see Figure \ref{fig:hd} in Appendix \ref{app:geo}) implies uniqueness of the 
 rectifiability parameter. The corresponding formal statement is as follows. 

\begin{lem}\label{lem:uniques}
If $\rvecx\in\reals^m$  is both $s$ and $t$-rectifiable, then  $s=t$. 
\end{lem}
\begin{proof}
See Appendix \ref{proof:uniques}.
\end{proof}

We are now ready to particularize our achievability result to $s$-rectifiable random vectors. 

\begin{cor}\label{cor:rec}
For $s$-rectifiable $\rvecx\in\reals^m$, $n>s$ is  sufficient for the existence of a 
Borel measurable mapping $g\colon\reals^{n\times m}\times\reals^n\to\reals^m$ 
  satisfying    
\begin{align}\label{eq:achCor}
\opP\mleft[g(\matA,\matA\rvecx)\neq\rvecx\mright]=0\quad\text{for $\lebmeasure^{n\times m}$-a.a. $\matA\in\reals^{n\times m}$}.  
\end{align}
\end{cor}
\begin{proof}
Follows  from the general achievability result Theorem \ref{th1} and Lemma \ref{lem:recsup}. 
\end{proof}
Again, Corollary \ref{cor:rec} is nonasymptotic (i.e., it applies for finite $m$) and guarantees zero error probability. 
Next, we present three examples of $s$-rectifiable random vectors aimed at illustrating the relationship between the rectifiability parameter and the stochastic sparsity level defined in \eqref{def:sl}. 
Specifically, in the first example, the  random vector's  
%is supported on the finite union  of $s$-dimensional linear 
%subspaces $\setU^\matI_s$  defined in Example \ref{exa:subspace}. 
%In this case, 
 rectifiability parameter will be shown to agree with its  stochastic sparsity level.  The second example constructs  an  $(r+t-1)$-rectifiable  random vector of stochastic sparsity level $S(\rvecx)=rt$ for general $r,t$. In this case,  
the  stochastic sparsity level $rt$ can be  much larger than  the rectifiability parameter $r+t-1$, and hence Corollary \ref{cor:rec} implies 
that this random vector can be recovered with zero error probability from a number of linear measurements that is much smaller than its stochastic sparsity level.  
%The third example is an $1$-rectifiable random vector on the unit circle with stochastic sparsity level two.  
The third example  constructs a random vector that is uniformly distributed on a submanifold,  
namely, the unit circle. In this case, the random vector's rectifiability parameter equals the dimension of the submanifold, whereas its stochastic sparsity level equals the dimension of the ambient space, i.e., the random vector  is not sparse at all.

\begin{exa}\label{exrecana1a}
Suppose that $\rvecx=(\vece_{\rveckc_1} \mydots \vece_{\rveckc_s})\rvecz\in\reals^m$, where $\rvecz\in\reals^s$ with $\mu_\rvecz\ll\lebmeasure^s$ and $\rveck=\tp{(\rveckc_1\mydots\rveckc_s)}\in\{1,\dots,m\}^s$   satisfies $\rveckc_1<\dots<\rveckc_s$.  
We first show that  %$\opP[\|\rveca\|_0 =r]=\opP[\|\rvecb\|_0 =t]=1$, which implies 
%$\opP[\|\rvecx\|_0 =rt]=1$ and %, in turn,  
 $S(\rvecx)=s$. To this end, let 
\begin{align} \label{eq:setUa}
 \setU=\{\vecx\in\reals^m:\|\vecx\|_0\leq s\}. 
\end{align}  
Since $\opP[\rvecx\in\setU]=1$ by construction, it follows that $S(\rvecx)\leq s$. 
To establish that $S(\rvecx)\geq s$, and hence $S(\rvecx)= s$,  
towards a contradiction, assume that there exists a linear subspace $\setT\subseteq\reals^m$ of dimension $d<s$ such that $\opP[\rvecx\in\setT]>0$.
Since
\begin{align}
0&<\mu_\rvecx(\setT)\\
&=\opP[(\vece_{\rveckc_1} \mydots \vece_{\rveckc_s})\rvecz\in\setT]\\
%&= \sum_{1\leq i_1<\dots<i_s\leq m}\opP[(\vece_{i_1} \mydots \vece_{i_s}) \rvecz \in\setT,\rveck=\tp{(i_1\mydots i_s)}]\\
&\leq  \sum_{1\leq i_1<\dots<i_s\leq m}\opP[(\vece_{i_1} \mydots \vece_{i_s}) \rvecz \in\setT], 
\end{align}
 there must exist a set of indices $\{i_1,\dots,i_s\}\subseteq\{1,\dots,m\}$ with $i_1<\dots < i_s$ such that
\begin{align}\label{eq:geqmeas1}
\opP[\matE\rvecz\in\setT]
&>0,   
\end{align}
where $\matE=(\vece_{i_1}\mydots\vece_{i_s})$. 
Next, consider the linear subspace  
\begin{align}
\widetilde \setT =\{\vecz\in\reals^{s}:\matE\vecz\in\setT\}\subseteq\reals^s, 
\end{align}
which, by \eqref{eq:geqmeas1}, satisfies $\opP[\rvecz\in \widetilde \setT]>0$. 
%By linearity of $\matE$,  $\widetilde\setT$ is a linear subspace of $\reals^{s}$. 
%As $\matE$ is obtained by removing column vectors from the identity matrix, it follows that $\dim(\widetilde\setT)\leq\dim(\setT)=d$. 
As  $d<s$ by assumption and $\dim(\widetilde\setT)\leq\dim(\setT)=d$, there must exist a nonzero vector $\vecb_0\in\reals^{s}$ such that $\tp{\vecb_0}\vecz=0$ for all $\vecz\in\widetilde\setT$. It follows that 
$\opP[\tp{\vecb_0}\rvecz=0]\geq\opP[\rvecz\in \widetilde \setT]>0$, which stands in contradiction to $\mu_{\rvecz}\ll\lebmeasure^{s}$ because  $\lebmeasure^{s}(\{\vecz:\tp{\vecb_0}\vecz=0\})=0$. 
Thus, we indeed have $S(\rvecx)=s$. 
%As $\rvecx$ has at most $s$ nonzero entries, .  
It follows from Properties \ref{exaC1} and  \ref{exasum} in Lemma \ref{LemmaExarec}
that $\setU$ in \eqref{eq:setUa} is countably $s$-rectifiable and, therefore, also countably $(\mathscr{H}^{s},s)$-rectifiable. We will see in Example \ref{exrecana1} that $\rvecx$ is in fact $s$-analytic, which,  by Property \ref{propra3} of Lemma \ref{lem:propanalytic}, implies  $\mu_\rvecx\ll \mathscr{H}^{s}$. 
Finally, by $\opP[\rvecx\in\setU]=1$, we get $\mu_\rvecx\ll \mathscr{H}^{s}|_{\setU}$, which, thanks to the countable  $(\mathscr{H}^{s},s)$-rectifiability of $\setU$ establishes the  $s$-rectifiability of $\rvecx$.
\end{exa}

\begin{exa}\label{exrecana1b}
Let $\rvecx =\rveca\otimes\rvecb$ with $\rveca\in\reals^{k}$ and $\rvecb\in\reals^{l}$. 
Suppose that  $\rveca=(\vece_{\rvecpc_1} \mydots \vece_{\rvecpc_r})\rvecu$ and $\rvecb=(\vece_{\rvecqc_1} \mydots \vece_{\rvecqc_t})\rvecv$, where 
$\rvecu\in\reals^r$ and  $\rvecv\in\reals^t$ with $\mu_{\rvecu}\times \mu_{\rvecv}\ll\lebmeasure^{r+t}$, and  
$\rvecp=\tp{(\rvecpc_1\mydots\rvecpc_r)}\in\{1,\dots,k\}^r$ and 
$\rvecq=\tp{(\rvecqc_1\mydots\rvecqc_t)}\in\{1,\dots,l\}^t$ satisfy  
$\rvecpc_1<\dots<\rvecpc_r$  
and   $\rvecqc_1<\dots<\rvecqc_t$, respectively.   
We first show that  %$\opP[\|\rveca\|_0 =r]=\opP[\|\rvecb\|_0 =t]=1$, which implies 
%$\opP[\|\rvecx\|_0 =rt]=1$ and %, in turn,  
 $S(\rvecx)=rt$. 
%To this end, let 
%\begin{align} \label{eq:setVa}
% \setV=\{\vecx\in\reals^m:\|\vecx\|_0\leq rt\}. 
%\end{align}  
Since $\opP[\|\rvecx\|_0\leq rt]=1$ by construction, it follows that $S(\rvecx)\leq rt$. 
To establish that $S(\rvecx)\geq rt$, and hence $S(\rvecx)= rt$, 
towards a contradiction, assume that there exists a linear subspace $\setT\subseteq\reals^m$ of dimension $d<rt$ such that $\opP[\rvecx\in\setT]>0$. 
Since 
\begin{align}
0&<\mu_\rvecx(\setT)\\
&=\opP[((\vece_{\rvecpc_1} \mydots \vece_{\rvecpc_r})\rvecu)\otimes((\vece_{\rvecqc_1} \mydots \vece_{\rvecqc_t})\rvecv) \in\setT]\\
&=\opP[((\vece_{\rvecpc_1} \mydots \vece_{\rvecpc_r})\otimes(\vece_{\rvecqc_1} \mydots \vece_{\rvecqc_t}))(\rvecu\otimes\rvecv) \in\setT]\label{eq:useHorn}\\
%&=\sum_{
%\substack{
%1\leq i_1<\dots<i_r\leq k\\[0.2em]
%1\leq j_1<\dots<j_t\leq l}
%}
%\opP[((\vece_{i_1}\mydots \vece_{i_r})\otimes(\vece_{j_1}\mydots \vece_{j_t}))(\rvecu\otimes\rvecv)\in\setT, \rvecp=\tp{(i_1\mydots i_r)},\rvecq=\tp{(j_1\mydots j_t)}]\\
&\leq \hspace*{-5truemm}\sum_{
\substack{
1\leq i_1<\dots<i_r\leq k\\[0.2em]
1\leq j_1<\dots<j_t\leq l}
}
\opP[((\vece_{i_1}\mydots \vece_{i_r})\otimes(\vece_{j_1}\mydots \vece_{j_t}))(\rvecu\otimes\rvecv)\in\setT], 
\end{align}
where \eqref{eq:useHorn} relies on  \cite[Lemma 4.2.10]{hojo91}, 
there must exist a set of indices $\{i_1,\dots,i_r\}\subseteq\{1,\dots,k\}$ with $i_1<\dots < i_r$ 
and a set of indices $\{j_1,\dots,j_t\}\subseteq\{1,\dots,l\}$ with $j_1<\dots < j_t$
such that
\begin{align}\label{eq:geqmeas}
\opP[(\matE_1\otimes\matE_2)(\rvecu\otimes\rvecv)\in\setT]
&>0,   
\end{align}
where $\matE_1=(\vece_{i_1}\mydots\vece_{i_r})$ and $\matE_2= (\vece_{j_1}\mydots\vece_{j_t})$. 
Next, consider the linear subspace 
\begin{align}
\widetilde \setT =\{\vecz\in\reals^{rt}:(\matE_1\otimes\matE_2)\vecz\in\setT\}\subseteq\reals^{rt}, 
\end{align}
which, by \eqref{eq:geqmeas}, satisfies $\opP[\rvecu\otimes\rvecv\in \widetilde \setT]>0$. 
As $d<rt$ by assumption and $\dim(\widetilde\setT)\leq\dim(\setT)=d$, there must exist a nonzero vector $\vecb_0\in\reals^{rt}$ such that $\tp{\vecb_0}\vecz=0$ for all $\vecz\in\widetilde\setT$. It follows that   
$\opP[\tp{\vecb_0}(\rvecu\otimes\rvecv)=0]\geq \opP[\rvecu\otimes\rvecv\in \widetilde \setT]>0$. 
As  $\mu_{\rvecu}\times \mu_{\rvecv}\ll\lebmeasure^{r+t}$ by assumption, we also have 
\begin{align}
\lebmeasure^{r+t}\mleft(\mleft\{\tp{(\tp{\vecu}\ \tp{\vecv})}: \vecu\in\reals^{r}, \vecv\in\reals^{t}, \tp{\vecb_0}(\vecu\otimes\vecv)=0\mright\}\mright)>0.
\end{align} 
We now view $\tp{\vecb_0}(\vecu\otimes\vecv)$ as a polynomial in the entries of $\vecu$ and $\vecv$. 
Since a polynomial   vanishes either on a set of Lebesgue measure zero or is identically zero (see Corollary \ref{cor:poly} and Lemma \ref{lem:vanishrealanalytic}), it follows that 
\begin{align}
 \tp{\vecb_0}(\vecu\otimes\vecv)=0\quad\text{for all $\vecu\in\reals^r$ and $\vecv\in\reals^t$,}
\end{align}
which stands in contradiction to  
$\vecb_0\neq\veczero$. Thus, we indeed have $S(\rvecx)=rt$. 
We next construct a countably $(r+t-1)$-rectifiable support set $\setU$ of $\rvecx$. 
To this end, we let 
\begin{align}
\setA&=\{\veca\in\reals^k :\|\veca\|_0\leq r\},\\
\setB&=\{\vecb\in\reals^l : \|\vecb\|_0 \leq t\},
\end{align}
and set 
$\setU=\setA\otimes\setB$. 
Since $\setA$ is a support set of $\rveca$ and $\setB$ is a support set of $\rvecb$, 
 $\setU$ is a support set of $\rvecx=\rveca\otimes\rvecb$.   
Note that  $\setU=(\setA\mysetminus\{\veczero\})\otimes\setB$. 
For $\veca\in\setA\mysetminus\{\veczero\}$, let $\bar a$ denote the first nonzero entry of $\veca$. 
We can now write 
\begin{align}
\veca\otimes\vecb
&=\mleft(\frac{\veca}{\bar a}\mright)\otimes(\bar a\,\vecb)\quad\text{for all}\ \veca\in\setA\mysetminus\{\veczero\},\vecb\in\setB.
\end{align}
This allows us to decompose $\setU$ according to $\setU=\tilde\setA\otimes\setB$, where
\begin{align}
\tilde\setA&=\{\veca\in\setA\mysetminus\{\veczero\}: \|\veca\|_0\leq r, \bar a=1\}.
\end{align}
Now, since $\tilde\setA$ is a finite union of affine subspaces all of dimension  $r-1$,  it is countably $(r-1)$-rectifiable by 
Properties \ref{exaC1} and \ref{exasum} in Lemma \ref{LemmaExarec}. By the same token, 
 $\setB$ as a  finite union of  linear subspaces all of dimension  $t$ is countably $t$-rectifiable.  
Therefore, the set 
\begin{align}
\setC=\big\{\tp{\big(\tp{\veca}\ \tp{\vecb}\big)}: \veca\in \tilde \setA, \vecb\in \setB\big\}\subseteq\reals^{k+l}
\end{align} 
is countably $(r+t-1)$-rectifiable thanks to 
Property \ref{exaProdrec} of Lemma \ref{LemmaExarec}. 
Now, the multivariate mapping $\sigma\colon\reals^{k+l}\to\reals^{kl}$,  $\tp{(\tp{\veca}\ \tp{\vecb})}\mapsto \veca\otimes\vecb$  is   bilinear and as such locally Lipschitz. 
Moreover, since $\setU=\sigma(\setC)$ with $\setC$  countably $(r+t-1)$-rectifiable,  
it follows from Property \ref{recf2}  of Lemma \ref{lem:recsetlip} that $\setU$ is countably $(r+t-1)$-rectifiable  and, therefore, also countably $(\mathscr{H}^{r+t-1},r+t-1)$-rectifiable.
We will see in Example \ref{exrecana} 
 that $\rvecx$ is in fact $(r+t-1)$-analytic, which,  by Property \ref{propra3} of Lemma \ref{lem:propanalytic}, implies  $\mu_\rvecx\ll \mathscr{H}^{r+t-1}$. 
With  $\opP[\rvecx\in\setU]=1$ this yields $\mu_\rvecx\ll \mathscr{H}^{r+t-1}|_{\setU}$ and, in turn, thanks to countable $(\mathscr{H}^{r+t-1},r+t-1)$-rectifiability of $\setU$,  establishes $(r+t-1)$-rectifiability of $\rvecx$.
\end{exa}

%The next example constructs $1$-rectifiable random vector supported on a  manifold,   
%The following example constructs a random vector that is uniformly distributed on a manifold, 
%namely, the unit circle. In this case the random vector is not sparse at all as it's stochastic sparsity level equals the dimension of the ambient space. 

\begin{exa}\label{exacircle1}
Let $\setS^1$\! denote the unit circle in $\reals^2$, $\rveczc\in \reals$ with $\mu_\rveczc\ll\lebmeasure^{1}$, and $g\colon\reals\to\setS^1$, $z\mapsto \tp{(\cos(z)\ \sin(z))}$.  Set $\rvecx=g(\rveczc)$, and note that this implies $\opP[\rvecx\in\setS^1]=1$. 
We first establish that $S(\rvecx)=2$. 
Since $\opP[\rvecx\in\reals^2]=1$, it follows that $S(\rvecx)\leq 2$. 
To establish that $S(\rvecx)\geq 2$, and hence $S(\rvecx)= 2$,  
towards a contradiction, assume that there exists a linear subspace $\setT\subseteq\reals^2$ of dimension one such that $\opP[\rvecx\in\setT]>0$. Set $\setA=\setT\cap\setS^1$, which consists of two antipodal points on $\setS^1$  (see Figure \ref{fig:S1}).  
Now, $0<\opP[\rvecx\in\setT]=\opP[\rvecx\in\setA]=\opP[\rveczc\in g^{-1}(\setA)]$, 
which constitutes  a contradiction to $\mu_\rveczc\ll\lebmeasure^{1}$ because $g^{-1}(\setA)$---as a countable set---must have Lebesgue measure zero. Therefore, $S(\rvecx)=2$.  
Finally, $\rvecx$ is $1$-rectifiable by \cite[Section III.D]{kopirihl15}. 
\end{exa}

Since $s$-rectifiable random vectors 
cannot have positive probability measure on sets of Hausdorff dimension $t<s$, 
it is natural to ask whether taking $n\geq s$ linear measurements is necessary for zero error recovery of $s$-rectifiable random vectors. 
Surprisingly, it turns out that this is not the case.  This will be demonstrated by first constructing 
 a  $2$-rectifiable (and therefore also $(\mathscr{H}^{2},2)$-rectifiable) set $\setG\subseteq \reals^3$ of strictly positive $2$-dimensional Hausdorff measure  with the property that  $\tp{\vece_3}\colon\reals^3\to\reals$, $\tp{(x_1\ x_2\ x_3)}\mapsto x_3$ is one-to-one on $\setG$.  
Then, we show that every $2$-rectifiable random vector  $\rvecx$ satisfying $\mu_\rvecx\ll\mathscr{H}^{2}|_\setG$ can be recovered with zero error  probability from one linear measurement, specifically from  $\rvecyc=\tp{\vece_3}\rvecx$. 
Moreover, all this is possible even though   $\setG$ contains 
the image of a Borel set in $\reals^2$ of positive  Lebesgue measure under a $C^\infty$-embedding.  
% (see Figure \ref{fig:2}  for an illustration). 

The construction of our example is based on the following result.  

%\begin{figure}\centering
%\includegraphics[width=0.45\textwidth]{fig.pdf}
%\caption{\small 
%$\setU$ contains the embedded image of a set $\setA$ of positive Lebesgue measure. \label{fig:2}}
%\end{figure}

\begin{thm}\label{prpcounter}
There exist a compact set $\setA\in\colB(\reals^2)$ with $\lebmeasure^{2}(\setA)=1/4$, and a $C^\infty$-function $\kappa\colon\reals^2\to\reals$  such that $\kappa$ is one-to-one on $\setA$.
\end{thm}
\begin{proof}
See Section \ref{prpcounterproof} for an explicit construction of $\kappa$ and $\setA$. 
\end{proof}

We now proceed to the construction of our example demonstrating that $n \geq s$ is in general not a necessary condition for zero error recovery of $s$-rectifiable random vectors.

\begin{exa}\label{exaG}
Let $\kappa$ and $\setA$ be as constructed in the proof of Theorem  \ref{prpcounter} and 
consider the mapping 
\begin{align}
 h\colon\reals^2&\to\reals^3\\
\vecz& \mapsto\tp{(\tp{\vecz}\ \kappa(\vecz))}. 
\end{align}
We set  $\setG=h(\setA)$ and show the following: 
\begin{enumerate}[label=\roman*)]
%\item $h$ is one-to-one;\label{ER3a}
\item  $h$ is a $C^\infty$-embedding;\label{ER3} 
\item $\setG$ is $2$-rectifiable;\label{ER4}
\item $0<\mathscr{H}^{2}(\setG)<\infty$; \label{ER1}
\item $\tp{\vece_3}\colon\reals^3\to\reals$, $\tp{(x_1\ x_2\ x_3)}\mapsto x_3$ is one-to-one on $\setG$.\label{ER2}
\item For every  $2$-rectifiable random vector  $\rvecx\in\reals^3$  with $\mu_\rvecx\ll\mathscr{H}^{2}|_\setG$,  there exists a Borel measurable mapping $g\colon\reals^{3}\times \reals\to \reals^3$ satisfying $\opP\!\big[g\big(\vece_3,\tp{\vece_3}\rvecx\big)\neq \rvecx\big]=0$.  \label{GG}
\end{enumerate}
It follows immediately that  $h$  is one-to-one. Thus, to establish Property \ref{ER3},  it  suffices  to prove that $h$ is a $C^\infty$-immersion.  
Since $\kappa$ is  $C^\infty$, so is $h$.   Furthermore, 
\begin{align}
Jh(\vecz)
&=\sqrt{\det(\tp{(Dh(\vecz))}Dh(\vecz))}\\
&=\sqrt{\det(\matI_2+\veca(\vecz)\tp{\veca(\vecz)})}\quad \text{for all}\ \vecz\in\reals^2,
\end{align}
where 
\begin{align}
\veca(z)=
\begin{pmatrix}
\frac{\partial \kappa(\vecz)}{\partial z_1}\\[2mm]
\frac{\partial \kappa(\vecz)}{\partial z_2}
\end{pmatrix}.
\end{align}  
Since $\veca(\vecz)\tp{\veca(\vecz)}$ is positive semidefinite, %\cite[Corollary 7.7.4]{hojo13} implies 
$Jh(\vecz)\geq \sqrt{\det(\matI_2)}=1$ for all $\vecz\in\reals^2$, which establishes that $h$ is  an immersion and completes the proof of   \ref{ER3}. 
  
To prove  \ref{ER4}, note that  $h$ is $C^\infty$ and as such  locally Lipschitz. 
As $\setA$ is compact, Lemma \ref{lem:lipcompact}  implies that $h|_\setA$ is Lipschitz. 
The set  $\setG=h(\setA)$ is hence the Lipschitz image of a compact set in $\reals^2$ and as such $2$-rectifiable. 

To establish \ref{ER1}, we first note that 
\begin{align}
\mathscr{H}^{2}(\setG)
&= \mathscr{H}^{2}( h({\setA}))\label{eq:step1G}\\
&\leq L^2 \mathscr{H}^{2}({\setA})\label{eq:step2G}\\
&=L^2\lebmeasure^2({\setA})\label{eq:step3G}\\
&<\infty, 
\end{align}
where the individual steps follow from  Properties of Lemma \ref{lem:prophausdorff}, namely,
\eqref{eq:step2G} from Property \ref{HL}  with $L$ denoting  the Lipschitz constant of $ h|_{{\setA}}$, and 
\eqref{eq:step3G} from  Property \ref{HBorel}. 

To establish 
$\mathscr{H}^{2}(\setG)>0$, consider the linear mapping 
\begin{align}
\pi\colon\reals^3&\to\reals^2\\
\tp{(x_1\ x_2\ x_3)}&\mapsto \tp{(x_1\ x_2)}. 
\end{align}
Clearly, $\pi$ is  Lipschitz  with Lipschitz constant equal to $1$. Therefore, 
\begin{align}
\mathscr{H}^{2}(\setG)
&\geq \mathscr{H}^{2}(\pi(\setG))\label{eq:step1Ga}\\
&=\mathscr{H}^{2}({\setA})\label{eq:step2Ga}\\
&=\lebmeasure^{2}(\setA)\label{eq:step3Ga}\\
&=\tfrac{1}{4},
\end{align}
where 
\eqref{eq:step1Ga} follows from Property \ref{HL} of Lemma \ref{lem:prophausdorff}, 
\eqref{eq:step2Ga} from $\pi(\setG)={\setA}$, 
 and \eqref{eq:step3Ga} is by Property \ref{HBorel} of Lemma \ref{lem:prophausdorff}.

To show \ref{ER2}, let $\vecx_1,\vecx_2\in\setG$ with $\vecx_1\neq\vecx_2$. Thus, $\vecx_1=\tp{(\tp{\vecz_1}\ \kappa(\vecz_1))}$ and $\vecx_2=\tp{(\tp{\vecz_2}\ \kappa(\vecz_2))}$ with $\vecz_1,\vecz_2\in\setA$ and $\vecz_1\neq\vecz_2$. 
As $\kappa$ is one-to-one on $\setA$, 
we conclude that $\tp{\vece_3}$ is one-to-one on $\setG$. 

It remains to establish  \ref{GG}. Since $\mu_\rvecx\ll\mathscr{H}^{2}|_\setG$ by assumption, it follows that 
$\opP[\rvecx\in\setG]=1$.  
We   show  that there exists a Borel measurable mapping $g\colon \reals^{3}\times \reals\to\reals^3$ such  that 
\begin{align}\label{Geq:gprop}
g(\veca,y)
\begin{cases}
\in\big\{\vecv\in\setG:\tp{\veca}\vecv=y\big\}&\text{if}\ \exists\vecv\in\setG:\tp{\veca}\vecv=y\\
=\vece&\text{else}, 
\end{cases}
\end{align}
where $\vece$ is an arbitrary vector not in $\setG$,  used to declare a decoding error.
Since $\opP[\rvecx\in\setG]=1$ and $\tp{\vece_3}$ is one-to-one on $\setG$, this  then implies  that $\opP\!\big[g\big(\vece_3,\tp{\vece_3}\rvecx\big)\neq \rvecx\big]=0$. 
To construct $g$ in  \eqref{Geq:gprop}, 
consider first the mapping 
\begin{align}
f\colon\reals^{3}\times\reals\times\reals^3&\to\reals\label{eq:Gf1a}\\
(\veca,y,\vecu)&\mapsto |y-\tp{\veca}\vecu|. \label{eq:Gf1b}
\end{align}
Since $f$ is continuous, 
Lemma \ref{lem:normal}  implies that $f$ is a normal integrand (see Definition \ref{dfn:normal}) with respect to $\colB(\reals^{3}\times\reals)$.
Let 
\begin{align}
\setT
&=\{(\veca,y)\in\reals^{3}\times\reals:\exists\vecu\in\setG\ \text{with}\  f(\veca,y,\vecu)\leq 0\}\\
&=\{(\veca,y)\in\reals^{3}\times\reals:\exists\vecu\in\setG\ \text{with}\  \tp{\veca}\vecu=y\}. 
\end{align} 
Note that $\setG$  as the Lipschitz image of the compact set $\setA$ is compact (see Lemma \ref{lem:fK}).  
It now follows from Properties \ref{lem:normal3ii} and \ref{lem:normal3iii} of Lemma \ref{lem:normal3} (with 
$\setT=\reals^{3}\times\reals$, $\alpha=0$,  $\setK=\setG$, and  $f$ as in \eqref{eq:Gf1a}--\eqref{eq:Gf1b}, which is a normal integrand with respect to $\colB(\reals^{3}\times\reals)$)  that i)  $\setT\in\colB(\reals^{3}\times\reals)$  and ii) there 
exists a Borel measurable mapping
\begin{align}
p\colon\setT&\to \reals^3\\
\big(\veca,y\big)&\mapsto p\big(\veca,y\big)\in\big\{\vecu\in\setG: \tp{\veca}\vecu=y\big\}. 
\end{align}
This mapping can then be extended to  a mapping $g\colon\reals^{3}\times\reals\to \reals^3$ by setting 
\begin{align}
g|_{\setT}&=p\\
g|_{(\reals^{3}\times\reals)\setminus\setT}&=\vece. 
\end{align}   
Finally, $g$ is  Borel measurable owing  to 
Lemma \ref{lem:fa}  as $p$ is Borel measurable and $\setT\in\colB(\reals^{3}\times\reals)$.  
\end{exa}

\section{Strong Converse}\label{sec:analytic} 
%Although  $s$-rectifiable random vectors 
%cannot have positive probability measure on sets of Hausdorff dimension $t<s$,  
Example \ref{exaG}   in the previous section demonstrates that 
%there exist 
%$s$-rectifiable random vectors 
%that can be recovered with zero error probability from $n<s$ linear measurements. 
%Thus, 
$n\geq s$ is not necessary for zero error recovery of $s$-rectifiable random vectors in general. 
In this section, we introduce the class of $s$-analytic random vectors $\rvecx$, which will be shown to allow for   a 
strong converse in the sense of  $n\geq s$ being necessary for recovery of $\rvecx$ with probability of error smaller than one. 
The adjective ``strong''  refers to the fact that $n< s$ linear measurements are necessarily insufficient even  if we allow a recovery  error probability that is  arbitrarily close to one. 
We prove that  an  $s$-analytic random vector is   $s$-rectifiable 
if and only if 
it admits a support set $\setU$ that is ``not too rich'' (in terms of $\sigma$-finiteness of $\mathscr{H}^{s}|_\setU$), 
show that the $s$-rectifiable random vectors  considered in 
Examples \ref{exrecana1a}--\ref{exacircle1} are all $s$-analytic, 
and discuss examples of $s$-analytic random vectors that fail to be  $s$-rectifiable. 
Random vectors that are  both $s$-analytic and  $s$-rectifiable can be recovered with zero error probability from $n> s$ linear measurements, and $n\geq s$ linear measurements are  necessary for  recovery  with error probability  smaller than one. The border case $n=s$ remains open. 

We now  make our way towards developing the strong converse and the formal definition of $s$-analyticity. 
The following auxiliary result will turn out to be useful.  %without making reference to a measurable decoder $g$.  
%We start with an auxiliary result. 
\begin{lem}\label{lem:convimpliesA}
For $\rvecx\in\reals^m$ and $\matA\in\reals^{n\times m}$, consider the following statements: 
\begin{enumerate}[label=\roman*)]
\item There exists  a  Borel measurable mapping $g\colon \reals^{n\times m}\times \reals^n\to\reals^m$ satisfying 
$\opP[g(\matA,\matA\rvecx)\neq\rvecx]<1$. \label{conv1}
\item There exists a set $\setU\in\colB(\reals^m)$ with $\opP[\rvecx\in\setU]>0$ such that $\matA$ is one-to-one on $\setU$. \label{conv2}
\end{enumerate}
Then,  \ref{conv1} implies \ref{conv2}. 
\end{lem}
\begin{proof} 
See Appendix \ref{lem:convimpliesAproof}.
\end{proof}

We first establish  a strong converse for the class of random vectors considered  in Example \ref{exrecana1a}.  
This will guide us to the crucial  defining property of  $s$-analytic random vectors. 
 
\begin{lem}\label{lem:simpleconverse}
Let $\rvecx=(\vece_{\rveckc_1} \mydots \vece_{\rveckc_s})\rvecz\in\reals^m$, where $\rvecz\in\reals^s$ with $\mu_\rvecz\ll\lebmeasure^s$ and $\rveck=\tp{(\rveckc_1\mydots\rveckc_s)}\in\{1,\dots,m\}^s$  satisfies $\rveckc_1<\dots<\rveckc_s$. 
If there exist a measurement matrix $\matA\in\reals^{n\times m}$ and a Borel measurable mapping $g\colon \reals^{n\times m}\times \reals^n\to\reals^m$  such that  
$\opP[g(\matA,\matA\rvecx)\neq\rvecx]<1$, then  $n\geq s$.
\end{lem}
\begin{proof}
Towards a contradiction, suppose that there exist a measurement matrix $\matA\in\reals^{n\times m}$ and a Borel measurable mapping $g\colon \reals^{n\times m}\times \reals^n\to\reals^m$  so that  
$\opP[g(\matA,\matA\rvecx)\neq\rvecx]<1$ for $n<s$. 
By Lemma \ref{lem:convimpliesA},  there must exist a $\setU\in\colB(\reals^m)$ with $\opP[\rvecx\in\setU]>0$ such that $\matA$ is one-to-one on $\setU$.  
Since  
\begin{align}
\opP[\rvecx\in\setU]
&=\opP[(\vece_{\rveckc_1} \mydots \vece_{\rveckc_s})\rvecz\in\setU]\\
&\leq\sum_{1\leq i_1<\dots<i_s\leq m}\opP[(\vece_{i_1}\mydots \vece_{i_s}) \rvecz \in\setU], 
\end{align}
there must exist a set of indices $\{i_1,\dots,i_s\}\subseteq\{1,\dots,m\}$ with $i_1<\dots < i_s$ such that the rank-$s$ matrix  $\matH:=(\vece_{i_1}\mydots \vece_{i_s})$ 
satisfies $\opP[\matH \rvecz \in\setU]>0$. 
Setting $\setA=\{\vecz\in\reals^s: \matH\vecz\in\setU\}$ yields 
$\mu_\rvecz(\setA)=\opP[\matH \rvecz \in\setU]>0$. Furthermore,  $\setA$  as the inverse image of the Borel set $\setU$ under a linear mapping is Borel. 
Finally, since $\mu_\rvecz\ll\lebmeasure^s$ by assumption, we conclude that $\lebmeasure^s(\setA)>0$.
Summarizing, 
there exist  a set $\setA\in\colB(\reals^s)$ with $\lebmeasure^s(\setA)>0$  and a matrix $\matH\in\reals^{m\times s}$ such that $\setU$ contains the one-to-one image of $\setA$ under   $\matH$. 
We now follow the line of argumentation used in the proof of  the converse part of \cite[Theorem 6]{wuve10}. 
Specifically, as  $\matA$ is one-to-one on $\setU$ and $\matH$ is one-to-one on $\reals^s$, it follows that $\matA\matH$ is one-to-one on $\setA$, i.e.,  
\begin{align}
\ker(\matA\matH)\cap(\setA-\setA)&=\{\veczero\}. \label{simpleexatwo}
\end{align} 
As $\lebmeasure^s(\setA)>0$, the Steinhaus Theorem \cite{st76} implies the existence of an $r>0$ such that $\setB_s(\veczero,r)\subseteq\setA-\setA\subseteq\reals^s$. Since $\dim\ker(\matA\matH)\geq s-n>0$, we conclude that the linear subspace $\ker(\matA\matH)$ must have a nontrivial intersection with $\setA-\setA$, 
which stands in  contradiction to  \eqref{simpleexatwo}. 
%The remainder of the proof now follows closely the line of argumentation used in \cite[Section VI]{wuve10} to get a  strong converse  for discrete-continuous mixture distributions.  
%We first note that, thanks to 
% We can now find $s$ linearly independent vectors $\vecb_1,\dots,\vecb_s\in\reals^s$ with 
%\begin{align}
%\{\vecb_1,\dots,\vecb_s\}\subseteq\setB_s(\veczero,r)\subseteq\setA-\setA. 
%\end{align} 
%Recall that $s>n$  by assumption. 
%As $\dim\ker(\matA\matH)\geq s-n$, there must exist at least  $s-n$ linearly independent vectors $\veca_1,\dots,\veca_{s-n}\in\reals^s$ such that 
%\begin{align}
%\{\veca_1,\dots,\veca_{s-n}\}\subseteq \ker(\matA\matH). 
%\end{align} 
%As $s>n$, the linear subspace spanned by $\{\veca_1,\dots,\veca_{s-n}\}\subseteq\reals^s$ and the linear subspace spanned by $\{\vecb_1,\dots,\vecb_s\}\subseteq\reals^s$ must have a nontrivial intersection, i.e., 
% there must exist a nonzero vector $\vecc\in\reals^s$ such that 
%\begin{align}\label{eq:vecc}
%\vecc=\alpha_1\veca_1+\dots+\alpha_{s-n}\veca_{s-n}=\beta_1\vecb_1+\dots+\beta_{s}\vecb_{s}.
%\end{align} 
%As  \eqref{eq:vecc} can be multiplied by an arbitrary positive constant, we may assume that $\vecc\in\setB_s(\veczero,r)\subseteq\setA-\setA$. 
%This contradicts, however,  $\ker(\matA\matH)\cap(\setA-\setA)=\{\veczero\}$ and thereby finishes the proof.  
\end{proof}

The strong converse just derived hinges critically on the specific structure of the $s$-rectifiable  random vector $\rvecx$ considered. Concretely, 
we used the fact that, for every $\setU\in\colB(\reals^m)$ with $\opP[\rvecx\in\setU]>0$, there exist 
a set $\setA\in\colB(\reals^s)$ with $\lebmeasure^s(\setA)>0$ and 
a  matrix $\matH\in\reals^{m\times s}$   such that $\setU$ contains the one-to-one image of $\setA$ under  $\matH$. 
The following example demonstrates, however, that this property is too strong for our purposes as it fails to hold for random vectors on general submanifolds like, e.g., the  unit circle:  

\begin{figure}
\begin{center}
\begin{tikzpicture}[scale=3]
\draw[arrows=->](-1,0)--(1,0);
\draw[arrows=->](0,-1)--(0,1);
\draw[thick,blue] (-0.9,-0.9)--(0.9,0.9);
\put (90,-2) {$x_1$};
\put (-3,90) {$x_2$};
\put (79,79) {\color{blue}$\setT_\vech$\color{black}};
\put (-50,50) {$\setS_1$};
\draw[thick] (0,0) circle(0.7);
\draw[red,fill=red] (0.5,0.5) circle(0.03);
\draw[red,fill=red] (-0.5,-0.5) circle(0.03);
\end{tikzpicture}
\end{center}
\caption{For every vector $\vech\subseteq\reals^2\mysetminus\{\veczero\}$, the linear subspace $\setT_\vech:=\{\vech z:z\in\reals\}$ intersects the unit circle $\setS_1$ in the two antipodal points $\pm\vech/\|\vech\|_2$.\label{fig:S1}} 
\end{figure}

\begin{exa}\label{exacircle2}
Let $\setS^1\subseteq\reals^2$\! denote the unit circle  and consider $\rvecx\in\reals^2$ supported on $\setS^1$, i.e.,  $\opP[\rvecx\in\setS^1]=1$. 
Towards a contradiction, 
suppose that  there exist a set $\setA\in\colB(\reals)$  with $\lebmeasure^1(\setA)>0$ 
and a  vector $\vech\in\reals^{2}$   such that $\{\vech z:z\in\setA\}\subseteq\setS^1$ and $\vech$ is one-to-one on $\setA$.  %We now show that this is only possible if $\lebmeasure^1(\setA)=0$.   
Since $\vech$ is one-to-one on $\setA$ and $\lebmeasure^1(\setA)>0$, it follows that $\vech\neq\veczero$. 
Noting that  $\{\vech z:z\in\setA\}\subseteq \{\vech/\|\vech\|_2,-\vech/\|\vech\|_2\}$ (see Figure \ref{fig:S1}), $\setA$ necessarily satisfies  $\setA\subseteq \{1/\|\vech\|_2,-1/\|\vech\|_2\}$. Thus,    $\lebmeasure^1(\setA)=0$, which is a contradiction to  $\lebmeasure^1(\setA)>0$. 
\end{exa}

%This   demonstrates that the line of argumentation used  to derive the strong  converse in Lemma \ref{lem:simpleconverse} fails for all random vectors on the unit circle. % in Example \ref{exacircle1}. 
The reason for this  failure is that every  $\vech\in\reals^2$ maps into a $1$-dimensional linear subspace in $\reals^2$,  and  $1$-dimensional linear subspaces in  $\reals^2$  
intersect the unit circle in two antipodal points only. To map  a set $\setA\in\colB(\reals)$ to a set in $\reals^2$ that is not restricted to be a subset of a $1$-dimensional linear subspace, we  have to employ a nonlinear mapping.  
But this puts us into the same dilemma as in  Example \ref{exaG}, where we demonstrated that even the requirement of  every $\setU\in\colB(\reals^m)$ with $\opP[\rvecx\in\setU]>0$  containing the embedded image---under a $C^\infty$-mapping---of a set $\setA\in\colB(\reals^s)$ of positive   Lebesgue measure  is not sufficient to obtain a strong  converse  for general $\rvecx$. 
We therefore need to impose additional constraints on the mapping. 
It  turns out that requiring real analyticity is enough. 
Examples of real analytic mappings include, e.g., multivariate polynomials, the exponential function, or trigonometric mappings.  
This finally leads us to the new concept of $s$-analytic measures and $s$-analytic random vectors. 

\begin{dfn}(Analytic measures)\label{dfn:measureanalytic}
A Borel measure $\mu$ on $\reals^{m}$ is $s$-analytic, with $s\in\{1,\dots,m\}$, if, for each $\setU\in\colB(\reals^{m})$ with $\mu(\setU)>0$, there exist a set $\setA\in\colB(\reals^{s})$ of positive  Lebesgue measure and 
a real analytic mapping (see Definition \ref{def:real}) $h\colon\reals^s\to \reals^m$ of $s$-dimensional Jacobian $Jh\not\equiv 0$   such that $h(\setA)\subseteq\setU$. 
\end{dfn}

Note that the only requirement on the real analytic mappings  in Definition \ref{dfn:measureanalytic} is  that their 
$s$-dimensional Jacobians do not vanish identically. Since  the $s$-dimensional Jacobian of a real analytic mapping is a real analytic function,   it  
vanishes either identically or on a set of Lebesgue measure zero (see Lemma \ref{lem:lebJ}). 
By Lemma \ref{lem:immreal},  this guarantees that, for an analytic measure $\mu$,  every $\setU\in\colB(\reals^{m})$ with $\mu(\setU)>0$  contains the real analytic \emph{embedding}  of a set $\setA\in\colB(\reals^{s})$ of positive  Lebesgue measure. 

%The restriction of $h$ in Definition \ref{dfn:measureanalytic} being real analytic on $\reals^s$ (and not just on an open set $\setU\subseteq\reals^s$) seems to be quite restrictive. However, this is not the case, and this assumption is for convenience only. In fact, if $\tilde h\colon \setB$

We have the following  properties of $s$-analytic measures. 

\begin{lem}\label{lem:propanalytic}
If $\mu$ is an $s$-analytic measure on $\reals^{m}$, then the following  holds:    
\begin{enumerate}[label=\roman*)]
\item  
$\mu$  is $t$-analytic for all   $t\in\{1,\dots,s-1\}$;\label{propra2}
\item 
$\mu\ll\mathscr{H}^{s}$;\label{propra3}
\item  
there exists a set $\setU\subseteq\reals^m$ such that $\mu=\mu|_\setU$ and $\mathscr{H}^{s}|_{\setU}$ is $\sigma$-finite if and only if  
 there exists a countably $(\mathscr{H}^{s},s)$-rectifiable set $\setW\subseteq\reals^m$ such that $\mu=\mu|_\setW$. 
\label{propra4} 
\end{enumerate}
\end{lem}
\begin{proof}
See Appendix \ref{proof:propanalytic}.
\end{proof}

We are now ready to define $s$-analytic random vectors. 

\begin{dfn}(Analytic random vectors)\label{DfnXana}
A random vector $\rvecx\in\reals^m$ is $s$-analytic if   $\mu_\rvecx$ is $s$-analytic.
The corresponding value $s$ is the analyticity parameter.
\end{dfn}

We have the following immediate consequence of Lemma \ref{lem:propanalytic}. 

\begin{cor}\label{cor:anarec}
Let $\rvecx$ be $s$-analytic. Then,  $\rvecx$ is $s$-rectifiable if and only if it admits a support set $\setU$ such that 
$\mathscr{H}^{s}|_{\setU}$ is $\sigma$-finite.
\end{cor}
\begin{proof}
Follows from Properties \ref{propra3} and \ref{propra4} in Lemma \ref{lem:propanalytic} and   Definition \ref{dfnXrec}.
\end{proof}

By Corollary \ref{cor:anarec}, an  $s$-analytic random vector  is $s$-rectifiable if and only if it admits a support set $\setU$ that is  ``not too rich'' (in terms of $\sigma$-finiteness of $\mathscr{H}^{s}|_\setU$). 
%We next construct  $s$-analytic random vectors that fail to be $s$-rectifiable (and, thus, do not admit such  support sets). 
As an example of an $s$-analytic random vector that is not $s$-rectifiable, consider an $(s+1)$-analytic random vector $\rvecx$ with $s>0$. By   
Property \ref{propra2} in Lemma \ref{lem:propanalytic},  this $\rvecx$ is also $s$-analytic, but it 
 cannot be $s$-rectifiable, as shown next.   %, regardless whether it is  $(s+1)$-rectifiable or not. 
Towards a contradiction, suppose that $\rvecx$ is $s$-rectifiable. 
Then, by Lemma \ref{lem:recsup}, $\rvecx$ has a countably $s$-rectifiable support set 
$\setU$, which by Property  \ref{exarec} in 
Lemma \ref{LemmaExarec} 
is  also  countably $(s+1)$-rectifiable.  
As, by assumption, $\rvecx$ is $(s+1)$-analytic, Property  \ref{propra3} in Lemma \ref{lem:propanalytic} implies   $\mu_\rvecx\ll\mathscr{H}^{s+1}$. Thus, $\mu_\rvecx\ll\mathscr{H}^{s+1}|_\setU$ with $\setU$ countably $(s+1)$-rectifiable,  and we  conclude that   $\rvecx$  would also be $(s+1)$-rectifiable, which 
contradicts uniqueness of the rectifiability parameter, as guaranteed by Lemma \ref{lem:uniques}. 

We  just demonstrated that  $s$-analytic random vectors cannot be  $s$-rectifiable if they are also $(s+1)$-analytic. 
The question now arises whether $s$-analytic random vectors that fail to be $(s+1)$-analytic (and, therefore, fail to be  $t$-analytic for all  $t>s$ by Property \ref{propra2} in Lemma \ref{lem:propanalytic}) are necessarily $s$-rectifiable. 
The next example shows that this is not the case. 
%But is it possible to construct an   The following example shows  that this is indeed the case. 

\begin{exa}\label{exareviw}
Let $\setC$ be the middle third Cantor set \cite[pp. xvi--xvii]{fa14} and consider $\setU=\{\tp{(c\ t)}:c\in\setC, t\in [0,1]\}\subseteq\reals^2$. 
Since   $0<\mathscr{H}^{\ln 2/\ln3}(\setC)<\infty$  \cite[Example 4.5]{fa14}, it follows that 
the random vector $\rvecx$  with distribution 
$\mu_\rvecx=\pi\times (\lebmeasure^1|_{[0,1]})$, where 
\begin{align}
\pi=\frac{\mathscr{H}^{\ln 2/\ln3}|_{\setC}}{\mathscr{H}^{\ln 2/\ln3}(\setC)}
\end{align}
is the normalized Hausdorff measure on $\setC$, is well defined.   
We now show that 
\begin{enumerate}[label=\roman*)]
\item $\rvecx$ is $1$-analytic;\label{item.A}
\item $\rvecx$ is not $2$-analytic;\label{item.B}
\item $\rvecx$ is not $1$-rectifiable.\label{item.C}
\end{enumerate} 
To establish \ref{item.A}, consider $\setB\in\colB(\reals^2)$ with $\opP[\rvecx\in\setB]>0$. 
Now,  
\begin{align}
0&<\mu_\rvecx(\setB)\\
&=\int_\setC\lebmeasure^1\big(\big\{t\in[0,1]:\tp{(c\ t)}\in\setB\big\}\big)\,\mathrm d \pi(c)\label{eq:fuba}\\
&\leq\int_\setC\lebmeasure^1\big(\big\{t\in\reals:\tp{(c\ t)}\in\setB\big\}\big)\,\mathrm d \pi(c),\label{eq:monla}
\end{align}
where in \eqref{eq:fuba} we applied Corollary \ref{cor.fubini} (with the finite measure spaces 
$(\reals, \colB(\reals),\pi)$ and  $(\reals, \colB(\reals),\lebmeasure^1|_{[0,1]}$) and 
\eqref{eq:monla} is by  monotonicity of Lebesgue measure. 
Thus,  by Lemma \ref{lem.fzm}, there must exist a $c_0\in\setC$ such that  $\setA:=\big\{t\in\reals:\tp{(c_0\ t)}\in\setB\big\}$ satisfies $\lambda^{1}(\setA)>0$. 
Now, define the mapping  
$h\colon \reals\to\reals^2$, $t\mapsto \tp{(c_0\ t)}$ and note that this mapping is  (trivially) real analytic with $Jh\equiv 1$. Moreover, $h(\setA)\subseteq\setB$ by construction, and $\setA$ is Borel measurable as the inverse image of the Borel set $\setB$ under the real analytic and, therefore, continuous mapping $h$.  
Thus, $\rvecx$ is $1$-analytic. 

We next  show that  $\rvecx$ is not $2$-analytic. Towards a contradiction, suppose that $\rvecx$ is $2$-analytic. Since $\mu_\rvecx(\setU)=1$, by $2$-analyticity of $\rvecx$,
there must exist a set 
$\setD\in\colB(\reals^2)$ with $\lambda^{2}(\setD)>0$ and a real-analytic mapping $g\colon \reals^2\to\reals^2$ of $2$-dimensional Jacobian $Jg\not\equiv 0$ 
such that $g(\setD)\subseteq \setU$. By   Property \ref{immreal2} in Lemma \ref{lem:immreal}, we can assume, w.l.o.g., that $g|_\setD$ is an embedding. It follows that    
\begin{align}
\mathscr{H}^2(g(\setD))
&=\int_\setD J\! g(\vecz)\,\mathrm d \lebmeasure^{2}(\vecz)\label{eq:usearea}\\
&>0,\label{eq:usefnz}
\end{align}
where 
in \eqref{eq:usearea} we applied the area formula 
Corollary \ref{cor:area} upon noting that  $g|_\setD$ is one-to-one as an embedding and locally Lipschitz by real analyticity of  $g$, 
and \eqref{eq:usefnz} is by Lemma \ref{lem.fzm}, $\lambda^{2}(\setD)>0$, and $Jg(\vecz)>0$ for all $\vecz\in\setD$.  
Since $g(\setD)\subseteq\setU$ and $\mathscr{H}^2(g(\setD))>0$,  monotonicity of $\mathscr{H}^2$ yields $\mathscr{H}^2(\setU)>0$. Upon noting that  
 $\mathscr{H}^{1+\ln 2/\ln3}(\setU)<\infty$ \cite[Example 4.3]{fa14}, this results in  a contradiction to   Property \ref{Hjump} in Lemma \ref{lem:prophausdorff}.

Finally, to establish  \ref{item.C}, towards a contradiction, suppose that $\rvecx$ is $1$-rectifiable. Then,  Lemma 
\ref{lem:recsup} implies that $\rvecx$ admits a countably $1$-rectifiable support set.  As every countably $1$-rectifiable  set is the countable union of 
$1$-rectifiable sets, the union bound implies that there must exist a $1$-rectifiable set $\setV$ with $\opP[\rvecx\in\setV]>0$. By Definition \ref{Dfn:recset}, 
there must therefore exist a compact set $\setK\subseteq\reals$ and a Lipschitz mapping $f\colon\setK\to\reals^2$  such that $\setV=f(\setK)$. 
It follows that 
\begin{align}
0
&<\mu_\rvecx(\setV)\label{eq:ccc0}\\
&=\mu_\rvecx(f(\setK))\\
&=\int_\setC\lebmeasure^{1}\big(\big\{t\in [0,1]:\tp{(c\ t)}\in f(\setK)\big\}\big)\,\mathrm d \pi(c),\label{eq:ccc3}\\
&= \int_\setC\lebmeasure^{1}\big(\big\{t\in [0,1]:\tp{(c\ t)}\in f(\setA_c)\big\}\big)\,\mathrm d \pi(c),\label{eq:ccc6}
\end{align} 
where in \eqref{eq:ccc3} we applied Corollary \ref{cor.fubini} (with the finite measure spaces 
$(\reals, \colB(\reals),\pi)$ and  $(\reals, \colB(\reals),\lebmeasure^1|_{[0,1]}$)  
and in \eqref{eq:ccc6} we set, for every $c\in\setC$,  
\begin{align}
\setA_c=f^{-1}\mleft(\big\{\tp{(c\ t)}:t\in[0,1]\big\}\mright)\subseteq\setK.   
\end{align}
Note that the sets $\setA_c\subseteq\setK$ are pairwise disjoint 
as inverse images of pairwise disjoint sets. 
Now, Lemma \ref{lem.fzm} together with \eqref{eq:ccc0}--\eqref{eq:ccc6} implies that there must exist a set $\setF\subseteq\setC$ with $\pi(\setF)>0$ such that 
\begin{align}\label{eq:lebunc}
\lebmeasure^{1}\big(\big\{t\in [0,1]:\tp{(c\ t)}\in f(\setA_c)\big\}\big)>0\quad\text{for all}\ c\in\setF. 
\end{align}
Since 
$\pi=\mathscr{H}^{\ln 2/\ln3}|_{\setC}/\mathscr{H}^{\ln 2/\ln3}(\setC)$  and  $0<\pi(\setF)\leq 1$, the definition of Hausdorff dimension (see Definition \ref{def:HDorf}) implies    $\dim_{\mathrm{H}}(\setF)= \ln 2/\ln3$.    
As every countable set has Hausdorff dimension zero \cite[p. 48]{fa14}, we conclude that $\setF$ must be uncountable. 
Moreover, 
\begin{align}
\lebmeasure^{1}(\setA_c)
&=\mathscr{H}^{1}(\setA_c)\label{eq:aaa1}\\
&\geq \frac{1}{L} \mathscr{H}^{1}(f(\setA_c))\label{eq:aaa2}\\
&\geq\frac{1}{L} \mathscr{H}^{1}\big(\big\{t\in [0,1]:\tp{(c\ t)}\in f(\setA_c)\big\}\big)\label{eq:aaa3a}\\
&=\frac{1}{L} \lebmeasure^{1}\big(\big\{t\in [0,1]:\tp{(c\ t)}\in f(\setA_c)\big\}\big)\label{eq:aaa3}\\
&>0\quad\text{for all}\ c\in\setF,\label{eq:aaa4}
\end{align}
where
\eqref{eq:aaa1} and \eqref{eq:aaa3} follow from Property \ref{HBorel} in Lemma \ref{lem:prophausdorff}, 
\eqref{eq:aaa2} is by Property \ref{HL} in Lemma \ref{lem:prophausdorff} with $L$ the Lipschitz constant of $f$, 
\eqref{eq:aaa3a} is again by Property \ref{HL} in Lemma \ref{lem:prophausdorff} with 
the Lipschitz constant of the projection $\tp{\vece_2}\colon\reals^2\to\reals, \tp{(c\ \, t)}\mapsto t$ equal to one,  
and in \eqref{eq:aaa4} we used \eqref{eq:lebunc}. 
As the sets $\setA_c$ are  pairwise disjoint subsets of positive Lebesgue measure of the compact set $\setK$, it follows that 
\begin{align}
\sup_{\setE\subseteq \setF: \card{\setE}<\infty}\sum_{c\in\setE}\lebmeasure^1(\setA_c)\leq \lebmeasure^1(\setK) <\infty,\label{eq:uncfinal} 
\end{align}
which, by Lemma \ref{lem:sumunc}, contradicts the uncountability of  $\setF$.  
Therefore, $\rvecx$ cannot be $1$-rectifiable.  
\end{exa}

Our strong converse for analytic random vectors will   be based on the following result. 

\begin{thm}\label{thm.converse}
Let $\setA\in\colB(\reals^s)$ be of positive  Lebesgue measure,   $h\colon\reals^s\to \reals^m$, with $s\leq m$,  real analytic of $s$-dimensional Jacobian $Jh\not\equiv 0$, and $f\colon\reals^m\to \reals^n$  real analytic. 
If $f$ is one-to-one on $h(\setA)$, then    $n\geq s$. 
\end{thm}
\begin{proof}See Section \ref{proof.converse}. \end{proof}

With the help of Theorem \ref{thm.converse}, we can now prove the strong converse for $s$-analytic random vectors. 

\begin{cor}\label{cor:analytic}
For $\rvecx\in\reals^m$ $s$-analytic,  $n\geq s$ is  necessary for the existence of a measurement matrix $\matA\in\reals^{n\times m}$ and a Borel measurable mapping  $g\colon\reals^{n\times m}\times\reals^n\to\reals^m$ such that $\opP[g(\matA,\matA\rvecx)\neq \rvecx]<1$. 
\end{cor}

\begin{proof}
Suppose, to the contrary,  that there exist a measurement matrix $\matA\in\reals^{n\times m}$ and a Borel measurable mapping $g\colon \reals^{n\times m}\times \reals^n\to\reals^m$  satisfying 
$\opP[g(\matA,\matA\rvecx)\neq\rvecx]<1$ for $n<s$. 
Then, by Lemma \ref{lem:convimpliesA},  there must exist a $\setU\in\colB(\reals^m)$ with $\opP[\rvecx\in\setU]>0$ such that $\matA$ is one-to-one on $\setU$.
As $\opP[\rvecx\in\setU]>0$, the $s$-analyticity of $\mu_\rvecx$ implies the existence of 
a set $\setA\in\colB(\reals^{s})$ of positive   Lebesgue measure along with
a real analytic mapping $h\colon\reals^s\to \reals^m$ of $s$-dimensional Jacobian $Jh\not\equiv 0$    such that $h(\setA)\subseteq\setU$. 
As $\matA$ is one-to-one on $h(\setA)$ and linear mappings are trivially real-analytic, 
 Theorem  \ref{thm.converse} implies that we  must have $n\geq s$, which contradicts  $n<s$.
\end{proof}

We next show that the $s$-rectifiable random vectors considered in 
Examples \ref{exrecana1a}--\ref{exacircle1} are all $s$-analytic with the analyticity parameter 
equal to the corresponding rectifiability parameter.  
We need the following result, which states that real analytic immersions  preserve analyticity in the following sense.   

\begin{lem}\label{lemancomp}
If $\rvecx\in\reals^m$ is $s$-analytic and $f\colon \reals^m\to\reals^k$, with $m\leq k$, is a real analytic immersion, then  
$f(\rvecx)$ is  $s$-analytic. 
\end{lem}
\begin{proof}
See Appendix \ref{proof:lemancomp}.
\end{proof}

\begin{exa}\label{exrecana1}
We show that $\rvecx$ in Example \ref{exrecana1a} is $s$-analytic. 
To this end, we consider an arbitrary but fixed  $\setU\in\colB(\reals^{m})$ with  $\mu_\rvecx(\setU)>0$ and establish the existence of  a set $\setA\in\colB(\reals^{s})$ of positive   Lebesgue measure
and a real analytic mapping $h\colon\reals^{s}\to \reals^{m}$ of $s$-dimensional Jacobian $Jh\not\equiv 0$   such that $h(\setA)\subseteq\setU$. Since    
\begin{align}
0&<\mu_\rvecx(\setU)\\
&=\opP[(\vece_{\rveckc_1} \mydots \vece_{\rveckc_s})\rvecz\in\setU]\\
%&=\sum_{1\leq i_1<\dots<i_s\leq m}\opP[(\vece_{i_1} \mydots \vece_{i_s}) \rvecz \in\setU,\rveck=\tp{(i_1\mydots i_s)}]\\
&\leq\sum_{1\leq i_1<\dots<i_s\leq m}\opP[(\vece_{i_1} \mydots \vece_{i_s}) \rvecz \in\setU],  
\end{align}
 there must exist a set of indices $\{i_1,\dots,i_s\}\subseteq\{1,\dots,m\}$ with $i_1<\dots < i_s$ such that
\begin{align}\label{eq:geqexa1a}
\opP[u(\rvecz) \in\setU]>0,  
\end{align}
where  $u\colon \reals^s\to\reals^m, \vecz\mapsto (\vece_{i_1} \mydots \vece_{i_s})\vecz$.  
As   $\mu_{\rvecz}\ll\lebmeasure^{s}$ by assumption, $\rvecz$ is  
$s$-analytic  thanks to Lemma \ref{auxlemma0} below. 
The mapping $u$ is 
linear and, therefore, trivially real analytic. 
Furthermore,  
\begin{align}
Ju(\vecz)
&= 
\sqrt{\det\Big(\tp{(\vece_{i_1} \mydots \vece_{i_s})}(\vece_{i_1} \mydots \vece_{i_s})\Big)}\label{eq:useDk}\\
&=
\sqrt{\det\matI_s}\\
&=1\quad \text{for all}\ \vecz\in\reals^{s}, 
\end{align}
where \eqref{eq:useDk} follows from $Du(\vecz)=(\vece_{i_1} \mydots \vece_{i_s})$ for all $\vecz\in\reals^{s}$, 
which proves that $u$ is an immersion. We can therefore employ Lemma \ref{lemancomp} and conclude that $u(\rvecz)$ is $s$-analytic.  Hence,  Definition \ref{DfnXana} together with \eqref{eq:geqexa1a}  implies that  there must exist 
a set $\setA\in\colB(\reals^{s})$ of positive  Lebesgue measure 
and a real analytic mapping $h\colon\reals^{s}\to \reals^{m}$ of $s$-dimensional Jacobian $Jh\not\equiv 0$   such that $h(\setA)\subseteq\setU$.
\end{exa}

\begin{lem}\label{auxlemma0}
If  $\rvecx\in\reals^m$ with $\mu_{\rvecx}\ll\lebmeasure^{m}$,  
then $\rvecx$ is $m$-analytic.  
\end{lem}
\begin{proof}
We have to show that, for each $\setU\in\colB(\reals^{m})$ with $\mu_\rvecx(\setU)>0$, we can find 
a set $\setA\in\colB(\reals^{m})$ of positive  Lebesgue measure and 
a real analytic mapping $h\colon\reals^m\to \reals^m$ of $m$-dimensional Jacobian $Jh\not\equiv 0$   such that $h(\setA)\subseteq\setU$. For given such   $\setU\in\colB(\reals^{m})$, simply take $\setA=\setU$ and 
 $h$  the identity mapping on $\reals^m$.
\end{proof}

\begin{exa}\label{exrecana}
We show that $\rvecx =\rveca\otimes\rvecb\in\reals^{kl}$ as in Example \ref{exrecana1b} is $(r+t-1)$-analytic. 
To this end, let   $\setU\in\colB(\reals^{kl})$ with  $\mu_\rvecx(\setU)>0$ be arbitrary but fixed.  
We have to  establish that there exist a set $\setA\in\colB(\reals^{r+t-1})$ of positive  Lebesgue measure
and a real analytic mapping $h\colon\reals^{r+t-1}\to \reals^{kl}$ of $(r+t-1)$-dimensional Jacobian $Jh\not\equiv 0$    such that $h(\setA)\subseteq\setU$. 
Since 
\begin{align}
0&<\mu_\rvecx(\setU)\\
&=\opP[((\vece_{\rvecpc_1} \mydots \vece_{\rvecpc_r})\rvecu)\otimes((\vece_{\rvecqc_1} \mydots \vece_{\rvecqc_t})\rvecv) \in\setU]\\
&=\opP[((\vece_{\rvecpc_1} \mydots \vece_{\rvecpc_r})\otimes(\vece_{\rvecqc_1} \mydots \vece_{\rvecqc_t}))(\rvecu\otimes\rvecv) \in\setU]\label{eq:useHornb} \\
&\leq\hspace*{-5truemm}\sum_{
\substack{
1\leq i_1<\dots<i_r\leq k\\[0.2em]
1\leq j_1<\dots<j_t\leq l}
}
\opP[((\vece_{i_1}\mydots \vece_{i_r})\otimes(\vece_{j_1}\mydots \vece_{j_t}))(\rvecu\otimes\rvecv)\in\setU], 
\end{align}
where \eqref{eq:useHornb} relies on  \cite[Lemma 4.2.10]{hojo91}, 
there must exist a set of indices $\{i_1,\dots,i_r\}\subseteq\{1,\dots,m\}$ with $i_1<\dots < i_r$ 
and a set of indices $\{j_1,\dots,j_t\}\subseteq\{1,\dots,m\}$ with $j_1<\dots < j_t$
such that
\begin{align}\label{eq:geqexa1}
\opP[v(\rvecu\otimes\rvecv)\in\setU]
&>0, 
\end{align}
where
\begin{align}
v\colon \reals^{rt}&\to\reals^{kl}\\
\vecw&\mapsto ((\vece_{i_1}\mydots \vece_{i_r})\otimes (\vece_{j_1}\mydots \vece_{j_t})) \vecw.  
\end{align}
Since $\mu_{\rvecu}\times \mu_{\rvecv}\ll\lebmeasure^{r+t}$ by assumption, 
it follows from Lemma 
\ref{auxlemma1} below that $\rvecu\otimes\rvecv$  is 
$(r+t-1)$-analytic. 
The mapping $v$ is linear and, therefore, trivially real analytic. 
Furthermore, 
\begin{align}
Jv(\vecw)
&= 
\sqrt{\det\Big(\tp{(\matE_1\otimes\matE_2)}(\matE_1\otimes\matE_2)\Big)}\label{eq:usedk2}\\
&=\sqrt{\det\Big((\tp{\matE_1}\matE_1)\otimes(\tp{\matE_2}\matE_2)\Big)}\label{eq:useHornc}\\
&=
\sqrt{\det(\matI_r\otimes\matI_t)}\\
&=1\quad \text{for all}\ \vecw\in\reals^{rt}, 
\end{align}
where \eqref{eq:usedk2} follows from $Dv(\vecw)=\matE_1\otimes\matE_2$ for all $\vecw\in\reals^{rt}$  with $\matE_1=(\vece_{i_1}\mydots\vece_{i_r})$ and $\matE_2= (\vece_{j_1}\mydots\vece_{j_t})$,   
and  \eqref{eq:useHornc} relies on  \cite[Equation (4.2.4)]{hojo91} and \cite[Lemma 4.2.10]{hojo91},
which proves that $v$ is an immersion. 
We can therefore employ Lemma \ref{lemancomp} and conclude that $v(\rvecu\otimes\rvecv)$ is $(r+t-1)$-analytic. Hence, 
Definition \ref{DfnXana} together with  \eqref{eq:geqexa1} implies that  
there must exist 
a set $\setA\in\colB(\reals^{r+t-1})$ of positive  Lebesgue measure and a 
real analytic mapping $h\colon\reals^{r+t-1}\to \reals^{kl}$ of $(r+t-1)$-dimensional Jacobian $Jh\not\equiv 0$  such that $h(\setA)\subseteq\setU$.
\end{exa}

\begin{lem}\label{auxlemma1}
If $\rveca\in\reals^{k}$ and $\rvecb\in\reals^{l}$ with $\mu_{\rveca}\times \mu_{\rvecb}\ll\lebmeasure^{k+l}$,  
then $\rveca\otimes\rvecb$ is $(k+l-1)$-analytic.  
\end{lem}
\begin{proof}See Appendix \ref{auxlemma1proof}.\end{proof}

\begin{exa}\label{exa:circle2}
Let $\rvecx$, $\rveczc$, and $h$ be  as in Example \ref{exacircle1}. We first note that $\sin$ and $\cos$ are real analytic. In fact, 
it follows from the ratio test \cite[Theorem 3.34]{ru89} that the power series 
\begin{align}
\sin(z)&=\sum_{n=0}^\infty\frac{(-1)^nz^{2n+1}}{(2n+1)!}\\
\cos(z)&=\sum_{n=0}^\infty\frac{(-1)^nz^{2n}}{(2n)!}
\end{align}
are absolutely convergent for all $z\in\reals$. Thus, $\sin$ and $\cos$ can both be represented by  convergent power series at $0$ with infinite convergence radius. Lemma \ref{lem:analyticrad}  therefore implies that 
 $\sin$ and $\cos$ are  both real analytic. As each  component of $h$ is real analytic, so is $h$. 
Furthermore, $Jh(z)=\sqrt{\sin^2(z)+\cos^2(z)}=1$ for all $z\in\reals$, which implies that $h$ is a real analytic immersion. Since $\rveczc$ is $1$-analytic by 
Lemma \ref{auxlemma0}  and $\rvecx=h(\rveczc)$, Lemma \ref{lemancomp} implies that $\rvecx$ is  $1$-analytic.
\end{exa}

%Property \ref{Mequiv} in Lemma \ref{lem:Mpropmindim}
\section{Proof of Theorem \ref{th1} (Achievability)}\label{th1.proof} 
Suppose that $K(\rvecx)< n$.   
It then follows from \eqref{def:dc} that  $\rvecx$  must admit a support set $\setU\subseteq \reals^m$  with $\underline{\dim}_\mathrm{MB}(\setU)< n$. 
We first construct a new support set $\setV\subseteq\reals^m$ for $\rvecx$ as a countable union of compact sets satisfying  $\underline{\dim}_\mathrm{MB}(\setV)< n$.
Based on this  support set $\setV$ we  %, which  as  a countable union of compact sets, 
then  prove the existence of a measurable decoder $g$ satisfying $\opP[g(\matA,\matA\rvecx)\neq\rvecx]=0$.   
The construction of $\setV$ starts by noting that,  thanks to Property \ref{Mequiv} in Lemma \ref{lem:Mpropmindim}, $\underline{\dim}_\mathrm{MB}(\setU)< n$ 
implies the existence of a covering $\{\setU_i\}_{i\in\naturals}$ of $\setU$ by nonempty  compact sets $\setU_i$ satisfying
\begin{align}\label{eq:Uisup}
 \sup_{i\in\naturals} \underline{\dim}_\mathrm{B}(\setU_i)<n.
\end{align}
For this covering, we set 
\begin{align}
\setV&=\bigcup_{i\in\naturals}\setU_i, \label{eq:defV} 
\end{align} 
and note that  
\begin{align}
\underline{\dim}_\mathrm{MB}(\setV)
&=\inf\mleft\{\sup_{i\in\naturals} \underline{\dim}_\mathrm{B}(\setV_i) : \setV\subseteq \bigcup_{i\in\naturals}\setV_i\mright\}\label{eq:Vn}\\
&\leq \sup_{i\in\naturals} \underline{\dim}_\mathrm{B}(\setU_i)\label{eq:Vn2}\\
&<n, \label{eq:NNN}
\end{align}
where \eqref{eq:Vn} follows from Property \ref{Mequiv} in Lemma \ref{lem:Mpropmindim} with the infimum taken over all 
coverings $\{\setV_i\}_{i\in\naturals}$ of $\setV$ by  nonempty  compact sets $\setV_i$, 
\eqref{eq:Vn2} is by \eqref{eq:defV}, 
and in \eqref{eq:NNN} we used \eqref{eq:Uisup}. 
Since $\dim_\mathrm{MB}(\reals^m)=m$ by Property \ref{Mpropmanifold} of Lemma \ref{lem:Mpropmindim}, and $\underline{\dim}_\mathrm{MB}(\setV)<n\leq m$, we must have $\setV\subsetneq\reals^m$.  
Now, $\setV$ is a support set  because it contains the support set $\setU$ as a subset. 
Furthermore, since $\opP[\rvecx\in\setV]=1$ and $\setV\subsetneq\reals^m$, there must exist an $\vece\in\reals^m\mysetminus\setV$ such that $\opP[\rvecx=\vece]=0$. 
This $\vece$ will be used to declare a decoding error. 
We will show in Section \ref{sec.existence} 
that there exists a Borel measurable mapping  
$g\colon \reals^{n\times m}\times \reals^n\to\reals^m$ satisfying  
\begin{align}\label{eq:meassec}
g(\matA,\vecy)
\begin{cases}
\in\{\vecv\in\setV:\matA\vecv=\vecy\}&\text{if}\ \exists\vecv\in\setV:\matA\vecv=\vecy\\
=\vece&\text{else}. 
\end{cases}
\end{align}
The mapping $g$ is guaranteed to deliver a $\vecv\in\setV$ that is consistent with $(\matA,\vecy)$ (in the sense of $\matA\vecv=\vecy$) if at least one such consistent  
$\vecv\in\setV$ exists, otherwise an error is declared by delivering the ``error symbol'' $\vece$. 
Next, for each  $\matA\in\reals^{n\times m}$, let  $p_\text{e}(\matA)$ denote the probability of error defined as 
\begin{align}
p_\text{e}(\matA)=\opP[g(\matA,\matA\rvecx)\neq\rvecx].     
\end{align}
It remains to show that $p_\text{e}(\matA)=0$ for $\lebmeasure^{n\times m}$-a.a. $\matA$.  
Now, 
\begin{align}
&p_\text{e}(\matA)\\
&=\opP\mleft[g(\matA,\matA\rvecx)\neq\rvecx,\rvecx\in\setV\mright]+\opP\mleft[g(\matA,\matA\rvecx)\neq\rvecx,\rvecx\notin\setV\mright]\label{eq:errorbound1a0}
\\
&=\opP\mleft[g(\matA,\matA\rvecx)\neq\rvecx,\rvecx\in\setV\mright]\label{eq:errorbound1a}\\
&=\opP\mleft[(\matA,\rvecx)\in\setA\mright]\quad \text{for all}\ \matA\in\reals^{n\times m}, \label{eq:errorbound1b}
\end{align}
where \eqref{eq:errorbound1a} follows from $\opP[\rvecx\in\setV]=1$ 
and in \eqref{eq:errorbound1b}  we set 
\begin{align}\label{eq:setAdef}
\setA
&=\{(\matA,\vecx) \in \reals^{n\times m}\times \setV : g(\matA,\matA\vecx)\neq\vecx \}.   
\end{align} 
Since $\setA\in\colB(\reals^{n\times m})\otimes\colB(\reals^{m})$ by  Lemma \ref{lem.measfixed} below (with $\setX=\reals^m$, $\setY=\reals^{n\times m}$, $f(\vecx,\matA)=g(\matA,\matA\vecx)$, and $\setV\in\colB(\reals^m)$ as a countable union of compact sets),  we can apply Corollary \ref{cor.fubini} (with the $\sigma$-finite measure spaces 
$(\reals^{m}, \colB(\reals^{m}),\mu_\rvecx)$ and  $(\reals^{n\times m}, \colB(\reals^{n\times m}),\lebmeasure^{n\times m})$) to 
$\setA$ and get  
\begin{align}
&\int_{\reals^{n\times m}} p_\text{e}(\matA)\,\mathrm d\lebmeasure^{n\times m}(\matA)\label{eq:stepfubini0}\\
&=\int_{\reals^m} \lebmeasure^{n\times m}(\{\matA : (\matA,\vecx)\in\setA\})\,\mathrm d\mu_\rvecx(\vecx)\label{eq:stepfubini}.
\end{align}
Next, note that for $\vecy=\matA\vecx$ with $\vecx\in\setV$, the vector $g(\matA,\vecy)$ can  differ from $\vecx$ only if there is a $\vecv\in\setV\,\mysetminus\{\vecx\}$ that is consistent with $\vecy$, i.e., if $\vecy=\matA\vecv$ for some $\vecv\in\setV\mysetminus\{\vecx\}$. Thus,  
\begin{align}\label{eq:rewriteA}
\setA\subseteq \{(\matA,\vecx)\in\reals^{n\times m}\times\setV: \ker(\matA)\cap\setV_\vecx\neq\{\matzero\}\}, 
\end{align}
where, for each $\vecx\in\setV$, we set  
\begin{align}
\setV_\vecx=\{\vecv-\vecx : \vecv\in\setV\}.  
\end{align}
As \eqref{eq:rewriteA} yields 
\begin{align}
&\{\matA \in\reals^{n\times m}:(\matA,\vecx)\in\setA\}\\
&\subseteq \big\{\matA \in\reals^{n\times m}:\ker(\matA)\cap\setV_\vecx\neq\{\matzero\}\big\}, 
\end{align}
 monotonicity of $\lebmeasure^{n\times m}$ implies 
\begin{align}\label{eq:finalth1}
&\lebmeasure^{n\times m}(\{\matA \in\reals^{n\times m}:(\matA,\vecx)\in\setA\})\\
&\leq \lebmeasure^{n\times m}\mleft(\big\{\matA \in\reals^{n\times m}:\ker(\matA)\cap\setV_\vecx\neq\{\matzero\}\big\}\mright)\label{eq:finalth1a} 
\end{align}
for all  $\vecx\in\setV.$ 
The null-space property Proposition \ref{prp:ns},  with $\setU=\setV_\vecx$ and $\underline{\dim}_\mathrm{MB}(\setV_\vecx)=\underline{\dim}_\mathrm{MB}(\setV)< n$ ((lower) modified Minkowski dimension is invariant under translation, as seen by translating covering balls accordingly) now implies that  \eqref{eq:finalth1a} equals zero for all $\vecx\in\setV$. Therefore,   \eqref{eq:finalth1} must  equal zero as well for all $\vecx\in\setV$. 
We conclude that \eqref{eq:stepfubini} must equal zero as the integrand is identically zero (recall that $\setV$ is a support set of $\rvecx$), which, by \eqref{eq:stepfubini0}-\eqref{eq:stepfubini} and Lemma \ref{lem.fzm}, implies that we must have $p_\text{e}(\matA)=0$ for  $\lebmeasure^{n\times m}$-a.a. $\matA$, thereby completing the proof. 
\qed

\subsection{Existence of  Borel Measurable $g$}\label{sec.existence}

Recall that i) $\setV=\bigcup_{i\in\naturals}\setU_i\subsetneq\reals^m$, where  $\setU_i\subseteq\reals^m$ is nonempty and compact for all $i\in\naturals$ and ii) the error symbol  $\vece\in\reals^m\mysetminus\setV$.  
We have to show  that there exists a Borel measurable mapping $g\colon \reals^{n\times m}\times \reals^n\to\reals^m$ such  that 
\begin{align}\label{eq:gprop}
g(\matA,\vecy)
\begin{cases}
\in\{\vecv\in\setV:\matA\vecv=\vecy\}&\text{if}\ \exists\vecv\in\setV:\matA\vecv=\vecy\\
=\vece&\text{else}. 
\end{cases}
\end{align}
To this end,  first consider the mapping 
\begin{align}
f\colon\reals^{n\times m}\times\reals^n\times\reals^m&\to\reals\label{eq:f1a}\\
(\matA,\vecy,\vecv)&\mapsto \|\vecy-\matA\vecv\|_2. \label{eq:f1b}
\end{align}
Since $f$ is continuous, 
Lemma \ref{lem:normal}  implies that $f$ is a normal integrand (see Definition \ref{dfn:normal}) with respect to $\colB(\reals^{n\times m}\times\reals^n)$.
For each $i\in\naturals$, let 
\begin{align}
\setT_i
&=\{(\matA,\vecy)\in\reals^{n\times m}\times\reals^n:\exists\vecu\in\setU_i\ \text{with}\  f(\matA,\vecy,\vecu)\leq 0\}\label{eq:Taui}\\
&=\{(\matA,\vecy)\in\reals^{n\times m}\times\reals^n:\exists\vecu\in\setU_i\ \text{with}\  \matA\vecu=\vecy\}. 
\end{align} 
It now follows from Properties \ref{lem:normal3ii} and \ref{lem:normal3iii} of Lemma \ref{lem:normal3} (with 
$\setT=\reals^{n\times m}\times\reals^n$, $\alpha=0$,  $\setK=\setU_i$, and  $f$ as in \eqref{eq:f1a}--\eqref{eq:f1b}, which is a normal integrand with respect to $\colB(\reals^{n\times m}\times\reals^n)$)  that i)  $\setT_i\in\colB(\reals^{n\times m}\times\reals^n)$ for all $i\in\naturals$ and ii) for every $i\in\naturals$, there 
exists a Borel measurable mapping
\begin{align}
p_{i}\colon\setT_i&\to \reals^m\\
(\matA,\vecy)&\mapsto p_{i}(\matA,\vecy)\in\{\vecu\in\setU_i: \matA\vecu=\vecy\}. 
\end{align}
For each $i\in\naturals$, the mapping $p_i$ can be extended to  a mapping $g_i\colon\reals^{n\times m}\times\reals^n\to \reals^m$ by setting 
\begin{align}
g_i|_{\setT_i}&=p_i\label{eq:gi1}\\
g_i|_{(\reals^{n\times m}\times\reals^n)\setminus\setT_i}&=\vece, \label{eq:gi2}
\end{align}
which is Borel measurable thanks to 
Lemma \ref{lem:fa}  as $p_{i}$ is Borel measurable and $\setT_i\in\colB(\reals^{n\times m}\times\reals^n)$.   
Based on this sequence $\{g_i\}_{i\in\naturals}$ of  Borel measurable mappings $g_i$, we now construct a Borel measurable mapping satisfying  \eqref{eq:gprop}. 
The idea underlying  this construction is as  follows. 
For a given pair $(\matA,\vecy)$, we first use $g_1$ to try to find a consistent (in  the sense of $\vecy=\matA\vecu$) $\vecu\in\setU_1$.  
If $g_1$ delivers the error symbol $\vece$, we use  $g_2$  to try to find a  consistent $\vecu\in\setU_2$. 
This procedure is continued until a $g_i$ delivers a consistent $\vecu\in\setU_i$. 
If no $g_i$ yields a consistent $\vecu\in\setU_i$, we deliver the error symbol $\vece$ as the final decoder output. 
The formal construction is  as follows. We set  $G_1=g_1$ and, 
for every $i\in\naturals\mysetminus\{1\}$, we define the mapping 
$G_i\colon\reals^{n\times m}\times\reals^n\to \reals^m$ iteratively by setting  
\begin{align}
G_{i}(\matA,\vecy)=
\begin{cases}
G_{i-1}(\matA,\vecy)&\text{if}\ G_{i-1}(\matA,\vecy)\neq\vece\\
g_{i}(\matA,\vecy)&\text{else}.
\end{cases}
\end{align}
Then, $G_1$ ($=g_1$) is Borel measurable, and, for each $i\in\naturals\mysetminus\{1\}$,  the Borel-measurability of $G_{i}$ follows from the Borel-measurability of  $G_{i-1}$ and  $g_{i}$ thanks to 
Lemma \ref{lem:combine} below. Note that by construction 
\begin{align}
G_{i}(\matA,\vecy)
%\begin{cases}
\in \Big\{\vecv\in\bigcup_{j=1}^i\setU_j: \matA\vecv=\vecy\Big\}
%=\vece&\quad \text{else}.
%\end{cases}
\end{align}
if there exists a  $\vecv\in\bigcup_{j=1}^i\setU_j$ such that $ \matA\vecv=\vecy$ and $G_{i}(\matA,\vecy)=\vece$ else. 
Finally, we obtain  $g\colon\reals^{n\times m}\times\reals^n\to \reals^m$ according to 
\begin{align}
g(\matA,\vecy)=\lim_{i\to\infty} G_i(\matA,\vecy), 
\end{align}
which satisfies  \eqref{eq:gprop}  by construction. 
As the pointwise limit of a sequence of Borel measurable mappings, $g$ 
is Borel measurable thanks to Corollary \ref{lem.meascart2}.\qed

\begin{lem}\label{lem.measfixed}
Let $\setX$ and $\setY$  be Euclidean spaces,  consider a Borel measurable mapping  
\begin{align}
f\colon\setX\times\setY \to \setX,   
\end{align}
and let  $\setV\in\colB(\setX)$. Then, 
\begin{align}
&\setA =\{(x,y)\in\setV\times\setY: f(x,y)\neq x\}\\
&\phantom{\setA}\in\colB(\setX\times\setY)=\colB(\setX)\otimes\colB(\setY).
\end{align}
\end{lem}
\begin{proof}
We first note that $\colB(\setX\times\setX)=\colB(\setX)\otimes\colB(\setX)$ and $\colB(\setX\times\setY)=\colB(\setX)\otimes\colB(\setY)$, both thanks to  Lemma \ref{lem:prodborel}. 
Therefore, $\setV\times\setX\in \colB(\setX\times\setX)$. Now,  consider the diagonal 
$\setD=\{(x,x):x\in\setX\}$ and note that $\setD$ as the inverse image of $\{0\}$ under the Borel measurable mapping 
$g\colon\setX\times\setX\to\setX$, $(u,v)\mapsto u\,-\,v$ is in $\colB(\setX\times\setX)$.
Let  $\setC=(\setV\times\setX)\cap((\setX\times\setX)\mysetminus\setD)$. Since $\setD\in \colB(\setX\times\setX)$, it follows that $\setC\in \colB(\setX\times\setX)$.  Define the  mapping  
\begin{align}
F\colon\setX\times\setY&\to \setX\times\setX\\
(x, y) &\mapsto (x,f(x,y)), 
\end{align}
and note that it is Borel measurable thanks to Lemma \ref{lem.meascart} 
(with $f_1\colon \setX\times\setY\to \setX$, $(x,y)\mapsto x$  and $f_2\colon \setX\times\setY\to \setX$, $(x,y)\mapsto f(x,y)$). 
Finally, $\setA\in\colB(\setX\times\setY)$ as $\setA=F^{-1}(\setC)$. 
\end{proof}

\begin{lem}\label{lem:combine}
Let $\setX$ and $\setY$ be topological spaces and $y_0\in\setY$ and suppose that  $f,g\colon\setX\to\setY$ are both Borel measurable.  
Then,   $h\colon\setX\to\setY$,   
\begin{align}
h(x)=
\begin{cases}
f(x)&\text{if}\ f(x)\neq y_0\\
g(x)&\text{else}
\end{cases}
\end{align}
is Borel measurable.
\end{lem}
\begin{proof}
We have to show that $h^{-1}(\setU)\in\colB(\setX)$ for all $\setU\in\colB(\setY)$. To this end, consider an arbitrary but fixed  
 $\setU\in\colB(\setY)$. Now, $\{y_0\}\in\colB(\setY)$ implies  $\setU\mysetminus \{y_0\}\in \colB(\setY)$. We  write $h^{-1}(\setU)=\setA\cup\setB$ with 
\begin{align}
\setA&=\{x\in\setX:h(x)\in\setU, f(x)\neq y_0\}\\
\setB&=\{x\in\setX:h(x)\in\setU, f(x)= y_0\}
\end{align}
and show that $\setA$ and $\setB$ are both in $\colB(\setX)$, which in turn implies $h^{-1}(\setU)\in\colB(\setX)$.  
Since 
\begin{align}
\setA
&=\{x\in\setX:f(x)\in\setU, f(x)\neq y_0\}\\
&=f^{-1}(\setU\mysetminus\{y_0\}), 
\end{align}
$\setU\mysetminus \{y_0\}\in \colB(\setY)$, and $f$ is Borel measurable by assumption, it follows that  $\setA\in \colB(\setX)$. 
 Finally, as 
\begin{align}
\setB
&=\{x\in\setX:g(x)\in\setU, f(x)= y_0\}\\
&=f^{-1}(\{y_0\})\cap g^{-1}(\setU), 
\end{align}
$\{y_0\}\in\colB(\setY)$, $\setU\in\colB(\setY)$, and $f$ and $g$ are both  Borel measurable by assumption,  it follows that $\setB\in\colB(\setX)$. 
Thus, $h^{-1}(\setU)=\setA\cup\setB\in\colB(\setX)$. Since $\setU$ was arbitrary, we conclude that  $h$ is Borel measurable.
\end{proof}

\section{Proof of Theorem  \ref{prpcounter}}\label{prpcounterproof}

\emph{Construction of $\setA$}.  
Consider the sequence  $\{a_k\}_{k\in\naturals}$, where $a_k=1/2+1/2^k$, and note that 
\begin{align}
\lim_{k\to\infty} a_k &= \tfrac{1}{2}\label{eq:conseq}\\
 a_k&> a_{k+1}\quad \text{for all}\ k\in\naturals. \label{eq:monseq}
\end{align}
Let $\setQ_1=[0,1]^2$ be the unit square of side length one. We define 
\begin{align}
\setQ_2=\setQ_{2,1}\cup\setQ_{2,2}\cup\setQ_{2,3}\cup\setQ_{2,4},
\end{align}
where every square $\setQ_{2,i}\subseteq\setQ_1$ has side length $a_2/2$ with 
$\tp{(0\ 0)}\in\setQ_{2,1}$, 
$\tp{(1\ 0)}\in\setQ_{2,2}$,
$\tp{(0\ 1)}\in\setQ_{2,3}$, and 
$\tp{(1\ 1)}\in\setQ_{2,4}$. 
It follows from \eqref{eq:monseq} that the squares $\setQ_{2,1},\dots, \setQ_{2,4}$ are pairwise disjoint and $\setQ_2\subsetneq\setQ_1$. 
To define $\setQ_3$, we follow the same procedure and split up every set $\setQ_{2,i}$ into the disjoint union of four squares with side length $a_3/4$. 
The sets $\setQ_1$, $\setQ_2$, and $\setQ_3$ are depicted in Figure \ref{fig:Q2}.
\begin{figure}
\begin{center}
\begin{tikzpicture}[scale=4]
\draw (0,0) rectangle (1,1);
\fill[color=gray] (0,0) rectangle (3/8,3/8);
\fill[color=gray] (5/8,0) rectangle (1,3/8);
\fill[color=gray] (0,5/8) rectangle (3/8,1);
\fill[color=gray] (5/8,5/8) rectangle (1,1);
\fill[color=gray] (0,0) rectangle (3/8,3/8);

\fill[color=black,pattern=north east lines,draw] (0,0) rectangle (5/32,5/32);
\fill[color=black,pattern=north east lines,draw] (0+7/32,0) rectangle (12/32,5/32);
\fill[color=black,pattern=north east lines,draw] (0,7/32) rectangle (5/32,12/32);
\fill[color=black,pattern=north east lines,draw] (7/32,7/32) rectangle (12/32,12/32);

\fill[color=black,pattern=north east lines,draw] (5/8,0) rectangle (5/32+5/8,5/32);
\fill[color=black,pattern=north east lines,draw] (7/32+5/8,0) rectangle (12/32+5/8,5/32);
\fill[color=black,pattern=north east lines,draw] (5/8,7/32) rectangle (5/32+5/8,12/32);
\fill[color=black,pattern=north east lines,draw] (7/32+5/8,7/32) rectangle (12/32+5/8,12/32);

\fill[color=black,pattern=north east lines,draw] (0,5/8) rectangle (5/32,5/32+5/8);
\fill[color=black,pattern=north east lines,draw] (7/32,5/8) rectangle (12/32,5/32+5/8);
\fill[color=black,pattern=north east lines,draw] (0,7/32+5/8) rectangle (5/32,12/32+5/8);
\fill[color=black,pattern=north east lines,draw] (7/32,7/32+5/8) rectangle (12/32,12/32+5/8);

\fill[color=black,pattern=north east lines,draw] (5/8,5/8) rectangle (5/32+5/8,5/32+5/8);
\fill[color=black,pattern=north east lines,draw] (7/32+5/8,5/8) rectangle (12/32+5/8,5/32+5/8);
\fill[color=black,pattern=north east lines,draw] (5/8,7/32+5/8) rectangle (5/32+5/8,12/32+5/8);
\fill[color=black,pattern=north east lines,draw] (7/32+5/8,7/32+5/8) rectangle (12/32+5/8,12/32+5/8);

%\draw (1,0)--(1,-0.03);
%\draw (0,1)--(-0.03,1);

%\draw[arrows=->](0,0)--(0,1.2);
%\draw[arrows=->](0,0)--(1.2,0);
%\put (140,-2) {$z_1$};
%\put (-3,140) {$z_2$};
%\put (115,-10) {$1$};
%\put (-10,110) {$1$};

\put (11, 47) {$Q_{2,1}$};
\put (85, 47) {$Q_{2,2}$};
\put (11, 118) {$Q_{2,3}$};
\put (85, 118) {$Q_{2,4}$};
\end{tikzpicture}
\end{center}
\caption{The set $Q_2$ consists of the four grey squares. The set $Q_3$ consists of the sixteen shaded black squares.\label{fig:Q2}} 
\end{figure}
We iterate this construction and obtain a sequence $\{\setQ_k\}_{k\in\naturals}$, where  $\setQ_k$ is the disjoint union of $4^{k-1}$ squares 
$\setQ_{k,i}, i=1,\dots, 4^{k-1}$, of side length $a_k/2^{k-1}$ and 
\begin{align}\label{eq:monsets}
\setQ_{k+1}\subsetneq \setQ_k\quad \text{for all}\ k\in\naturals.
\end{align} 
Next, we set 
\begin{align}\label{eq:defA}
\setA=\bigcap_{k\in\naturals}\setQ_k,
\end{align} 
which, as the intersection of closed sets, is closed. Since $\setA$ is also bounded it must be compact by the Heine-Borel theorem \cite[Theorem 2.41]{ru89}. 
Finally,  
\begin{align}
\lebmeasure^{2}(\setA)
&=\lim_{k\to\infty}\lebmeasure^{2}(\setQ_k)\label{eq:applem15}\\
&=\lim_{k\to\infty}\frac{4^{k-1}a_k^2}{(2^{k-1})^2}\\
&=\lim_{k\to\infty}a_k^2\\
&=\tfrac{1}{4},\label{eq:1/4}
\end{align}
where \eqref{eq:applem15} follows from Property \ref{lem.capmeasureiii} of Lemma \ref{lem.capmeasure}.

\emph{Construction of $\kappa$}. We now construct a $C^\infty$-function $\kappa\colon\reals^2\to\reals$ that is one-to-one on $\setA$ as defined in \eqref{eq:defA}. This will be accomplished by building compactly supported $C^\infty$-functions $\varphi_{k,i}\colon\reals^2\to[0,1]$, $i=1,\dots,4^{k-1}$, $k\in\naturals$, such that    
\begin{align}\label{eq:fi}
\varphi_{k,i}(\vecz)=
\begin{cases}
1&\text{if}\ \vecz\in\setQ_{k,i},\\
0&\text{if}\ \vecz\in\setQ_{k,j},\  j\in\{1,\dots,4^{k-1}\}\mysetminus\{i\}.
\end{cases}
\end{align}
The construction of  these functions is effected by  Lemma \ref{lem:square} below with  $\varphi_{k,i}(\vecz)=\psi_{\delta_{k,i},a_{k,i},\vecw_{k,i}}(\vecz)$,  where 
$\vecw_{k,i}$ denotes the center of $\setQ_{k,i}$, $a_{k,i}$ equals half  the side-length of $\setQ_{k,i}$, and  $\delta_{k,i}$ is chosen sufficiently small for \eqref{eq:fi} to hold (recall that the squares $\setQ_{k,i}$ are  closed and disjoint).   
 Next, we define  the $C^\infty$-functions 
\begin{align}\label{eq:phik}
\varphi_k=\sum_{i=1}^{4^{k-1}}\frac{i}{4^{k-1}} \varphi_{k,i}
\end{align}
and note that 
\begin{align}
\varphi_k(\vecz)&=\frac{i}{4^{k-1}}\label{eq:sameQkj} 
\end{align}
for all $\vecz\in \setQ_{k,i}$, $i=1,\dots, 4^{k-1}$, and $k\in\naturals$,  and 
\begin{align}
|\varphi_k(\vecz)-\varphi_k(\vecw)|&=\frac{|i-j|}{4^{k-1}}\label{eq:nsameQkja}\\
&\geq \frac{1}{4^{k-1}}\label{eq:nsameQkj} 
\end{align}
for all  $\vecz\in \setQ_{k,i}$,  $\vecw\in\setQ_{k,j}$,  $1 \leq i < j\leq 4^{k-1}$, and $k\in\naturals$. 
For $l\in\naturals$ and $a,b\in\naturals_0$, consider now the $C^\infty$-function    
\begin{align}
s^{(a,b)}_{l}(\vecz)
&:=\frac{\partial^{\,a+b}}{\partial z_1^a \partial z_2^b}\sum_{k=1}^l \frac{1}{8^{2^k}(M_{k}+1)}\varphi_k(\vecz)\\  
&=\sum_{k=1}^l \frac{1}{8^{2^k}(M_{k}+1)}\frac{\partial^{\,a+b}\varphi_k(\vecz)}{\partial z_1^a \partial z_2^b}  
\end{align} 
 with 
\begin{align} 
M_{k}&=\max_{1\leq i\leq   k} \,\max_{1\leq j < k} d(i,j),\label{Mk}\\
d(i,j)&=\sup_{\vecz\in\reals^2}\max\Bigg\{\mleft|\frac{\partial^j\varphi_i(\vecz)}{\partial z_1^a \partial z_2^b}\mright| : a,b\in\naturals_0,\ a+b=j \Bigg\}.
\end{align} 

We now show that this particular  choice for the constants $M_{k}$ guarantees, 
for each  $a,b\in\naturals_0$, that the 
sequence $\{s^{(a,b)}_{l}\}_{l\in\naturals}$  of $C^\infty$-functions converges uniformly and denote the corresponding limiting functions by $\kappa^{(a,b)}$.  
Corollary \ref{cor.exchangesumder} then  implies 
\begin{align}
\frac{\partial\kappa^{(a,b)}(\vecz)}{\partial z_1}=\kappa^{(a+1,b)}(\vecz)\\
\frac{\partial\kappa^{(a,b)}(\vecz)}{\partial z_2}=\kappa^{(a,b+1)}(\vecz)
\end{align}
for all $a,b\in\naturals_0$ and $\vecz\in\reals^2$, and   $\kappa^{(0,0)}$ must therefore be  $C^\infty$. We set $\kappa=\kappa^{(0,0)}$.
It remains to prove uniform convergence of the sequences $\{s^{(a,b)}_{l}\}_{l\in\naturals}$  of $C^\infty$-functions for all   $a,b\in\naturals_0$. 
To this end, let $a,b\in\naturals_0$ be arbitrary but fixed and note that
 \begin{align}\label{eq:boundMk}
\frac{1}{8^{2^k}(M_{k}+1)}\mleft|\frac{\partial^{\,a+b}\varphi_k(\vecz)}{\partial z_1^a \partial z_2^b}\mright| 
&\leq \frac{d(k,a+b)}{8^{2^k}(M_{k}+1)}\\
&< \frac{1}{8^{2^k}}
\end{align}
for all $k>a+b$ and $\vecz\in\reals^2$.    
Furthermore, by the sum formula for the geometric series,  
\begin{align}
\sum_{k={a+b+1}}^\infty \frac{1}{8^{2^k}}
&< \sum_{k\in\naturals}\frac{1}{8^{k}}= \frac{1}{7}. \label{eq:boundMk2}
\end{align} 
We can now conclude from \eqref{eq:boundMk} and \eqref{eq:boundMk2} 
that the sequence  $\{t^{(a,b)}_{l}\}_{l\in\naturals}$ of $C^\infty$-functions 
\begin{align}
t^{(a,b)}_{l}(\vecz) 
&:=\sum_{k=a+b+1}^{a+b+l} \frac{1}{8^{2^k}(M_{k}+1)}\frac{\partial^{\,a+b}\varphi_k(\vecz)}{\partial z_1^a \partial z_2^b}  
\end{align}  
satisfies the assumptions of Theorem \ref{thm.M} and therefore converges uniformly to a function, which we denote by $\rho^{(a,b)}$. 
As 
\begin{align}
s^{(a,b)}_{l}(\vecz)
&=\sum_{k=1}^{a+b}\frac{1}{8^{2^k}(M_{k}+1)}\frac{\partial^{\,a+b}\varphi_k(\vecz)}{\partial z_1^a \partial z_2^b} +  t^{(a,b)}_{l-a-b}(\vecz)
\end{align}
for all $l>a+b$, we conclude that  $\{s^{(a,b)}_{l}\}_{l\in\naturals}$ must converge uniformly to  
\begin{align}
\kappa^{(a,b)}(\vecz):=\sum_{k=1}^{a+b}\frac{1}{8^{2^k}(M_{k}+1)}\frac{\partial^{\,a+b}\varphi_k(\vecz)}{\partial z_1^a \partial z_2^b}+ \rho^{(a,b)}(\vecz).  
\end{align}
Since $a$ and $b$ are arbitrary, this implies that $\{s^{(a,b)}_{l}\}_{l\in\naturals}$  converges uniformly for all $a,b\in\naturals_0$, thereby concluding the proof of $\kappa$ being $C^\infty$.

It remains to show that $\kappa$ is one-to-one on $\setA$. To this end, consider arbitrary but fixed $\vecz_0$ and $\vecw_0$ in $\setA$ with $\vecz_0\neq\vecw_0$.  We have to show that $\kappa(\vecz_0)\neq \kappa(\vecw_0)$. 
Note that by construction of $\setA$ (see \eqref{eq:defA}), there exists a $k_0\in\naturals$ such that 
\begin{enumerate}[label=\roman*)]
\item\label{k0i}
for every $k\geq k_0$, there exist $i_k$ and $j_k$ in $\{1,\dots,4^{k-1}\}$ with $i_k\neq j_k$ such that $\vecz_0\in\setQ_{k,i_k}$ and $\vecw_0\in\setQ_{k,j_k}$, and
\item\label{k0ii}
for every $k < k_0$, there exists an  $i_k$ such that  $\vecz_0$ and $\vecw_0$ are both in $\setQ_{k,i_k}$. 
\end{enumerate}
We therefore have 
\begin{align}
&|\kappa(\vecz_0)-\kappa(\vecw_0)|\\
&=\mleft|\sum_{k\in\naturals}\frac{\varphi_k(\vecz_0)-\varphi_k(\vecw_0)}{8^{2^k}(M_k+1)}\mright|\label{eq:step1g}\\
&=\mleft|\sum_{k=k_0}^\infty\frac{\varphi_k(\vecz_0)-\varphi_k(\vecw_0)}{8^{2^k}(M_k+1)}\mright|\label{eq:step2g}\\
&\geq
\frac{|\varphi_{k_0}(\vecz_0)-\varphi_{k_0}(\vecw_0)|}{8^{2^{k_0}}(M_{k_0}+1)}-
\mleft|\sum_{k=k_0+1}^\infty\frac{\varphi_k(\vecz_0)-\varphi_k(\vecw_0)}{8^{2^k}(M_k+1)}\mright|\label{eq:step3g}\\
&\geq
\frac{|\varphi_{k_0}(\vecz_0)-\varphi_{k_0}(\vecw_0)|}{8^{2^{k_0}}(M_{k_0}+1)}-
\sum_{k=k_0+1}^\infty\frac{|\varphi_k(\vecz_0)-\varphi_k(\vecw_0)|}{8^{2^k}(M_k+1)}\label{eq:step4g}\\
&\geq
\frac{1}{M_{k_0}+1}
\mleft(
\frac{|\varphi_{k_0}(\vecz_0)-\varphi_{k_0}(\vecw_0)|}{8^{2^{k_0}}}-\mright.\nonumber\\
&\phantom{\frac{1}{M_{k_0}+1}(\ \ \ \ }
\mleft.\sum_{k=k_0+1}^\infty\frac{|\varphi_k(\vecz_0)-\varphi_k(\vecw_0)|}{8^{2^k}}
\mright),\label{eq:step5g}
\end{align} 
where  \eqref{eq:step1g} follows from the uniform  convergence of $\{s^{(0,0)}_{l}\}_{l\in\naturals}$ to $\kappa$, in \eqref{eq:step2g} we used \eqref{eq:sameQkj} and \ref{k0ii} above, \eqref{eq:step3g} is by  the reverse triangle inequality,  and \eqref{eq:step5g} follows from $M_k\leq M_{k+1}$ for all $ k\in\naturals$ (see \eqref{Mk}). 
Moreover, \eqref{eq:nsameQkja}--\eqref{eq:nsameQkj} imply 
$|\varphi_{k_0}(\vecz_0)-\varphi_{k_0}(\vecw_0)|\geq 1/4^{k_0-1}$ and  \eqref{eq:sameQkj} yields 
\begin{align}
|\varphi_k(\vecz_0)-\varphi_k(\vecw_0)|
&\leq |\varphi_k(\vecz_0)|+|\varphi_k(\vecw_0)|\\
&=\frac{i_k+j_k}{4^{k-1}}\\
&\leq 2\quad \text{for all}\ k\in \naturals.
\end{align}
We can therefore further lower-bound $|\kappa(\vecz_0)-\kappa(\vecw_0)|$ according to 
\begin{align}
&|\kappa(\vecz_0)-\kappa(\vecw_0)|\\
&\geq
\frac{1}{M_{k_0}+1}
\Bigg(
\frac{1}{4^{k_0-1}8^{2^{k_0}}}-
\sum_{k=k_0+1}^\infty\frac{2}{8^{2^k}}
\Bigg)\\
&=
\frac{1}{M_{k_0}+1}
\Bigg(
\frac{1}{4^{k_0-1}8^{2^{k_0}}}-
\sum_{k\in\naturals_0}\frac{2}{8^{2^{k+k_0+1}}}\Bigg)\\
&\geq
\frac{1}{M_{k_0}+1}
\Bigg(
\frac{1}{4^{k_0-1}8^{2^{k_0}}}-
\frac{2}{8^{2^{k_0+1}}}\sum_{k\in\naturals_0}\frac{1}{8^{k}}
\Bigg)\\
%&=
%\frac{1}{(M_{k_0}+1)8^{2^{k_0}}}
%\Bigg(
%\frac{1}{4^{k_0-1}}-
%\frac{2}{8^{2^{k_0}}}\frac{8}{7}
%\Bigg)\\
&>
\frac{4}{(M_{k_0}+1)8^{2^{k_0}}}
\Bigg(
\frac{1}{4^{k_0}}-
\frac{1}{8^{2^{k_0}}}
\Bigg)\label{eq:195}\\
%&=
%\frac{1}{(M_{k_0}+1)8^{2^{k_0}}}\frac{8^{2^{k_0}}-4^{k_0}}{4^{k_0-1}8^{2^{k_0}}}\\
&>0.
\end{align} 
Since $\vecz_0$ and $\vecw_0$ are arbitrary, this establishes that $\kappa$ is indeed one-to-one, thereby finishing the proof. 
\qed

\begin{lem}\label{lem:square}
For  $a>0$, $\delta>0$, and $\vecw=\tp{(\vecwc_1\ \vecwc_2)}\in\reals^2$, consider the mapping 
\begin{align}
\psi_{\delta,a,\vecw}\colon\reals^2&\to[0,1]\\
\vecz&\mapsto \rho_{\delta,a}(z_1-w_1)\rho_{\delta,a}(z_2-w_2), 
\end{align}
where
\begin{align}
\rho_{\delta,a}\colon \reals&\to[0,1]\\
t&\mapsto \frac{f(a+\delta-|t|)}{f(a+\delta-|t|)+f(|t|-a)} 
\end{align}
with $f(t)=e^{-1/t}\ind{\reals_+}(t)$. 
Then, $\psi_{\delta,a,\vecw}$ is $C^\infty$ and satisfies
\begin{align}
\psi_{\delta,a,\vecw}(\vecz)
&=
\begin{cases}
1& \text{if}\ \max\{|\veczc_1-\vecwc_1|,|\veczc_2-\vecwc_2|\} \leq a,\\
0&\text{if}\ \min\{|\veczc_1-\vecwc_1|,|\veczc_2-\vecwc_2|\} \geq a+\delta.
\end{cases}
\end{align}
\end{lem}
\begin{proof}
It follows from \cite[Lemma 2.22]{le13} with $H=\rho_{\delta,a}$, $r_1=a$, and $r_2=a+\delta$ that $\rho_{\delta,a}$ is $C^\infty$ and satisfies 
\begin{align}
\rho_{\delta,a}(t)=
\begin{cases}
1&\text{if}\ |t|\in [0,a],\\
0&\text{if}\ |t|\in [a+\delta,\infty).
\end{cases}
\end{align}
The claim now follows  from $\psi_{\delta,a,\vecw}=\rho_{\delta,a}(z_1-w_1)\rho_{\delta,a}(z_2-w_2)$ and the properties of $\rho_{\delta,a}$.
\end{proof}

\section{Proof of Theorem \ref{thm.converse} (Strong Converse)}\label{proof.converse}

Towards a contradiction, suppose  that the statement is false. 
That is, we can find an $s\in\naturals$ such that there exist   
\begin{enumerate}[label=\roman*)]
\item  an $\setA\in\colB(\reals^{s})$  with $\lebmeasure^{s}(\setA)>0$, \label{falsei}
\item a real analytic mapping $h\colon\reals^{s}\to \reals^{m}$, with $s\leq m$, of $s$-dimensional Jacobian $Jh\not\equiv 0$, and \label{falseii}
\item an $n\in\naturals$ with $n<s$  and a real analytic mapping $f\colon\reals^{m}\to \reals^{n}$ that is one-to-one on $h(\setA)$. \label{falseiii}
\end{enumerate}
Let $s_0$ be the smallest  $s\in\naturals$ such that  \ref{falsei}--\ref{falseiii} hold. 
The proof will be effected by  showing that this implies  validity of \ref{falsei}--\ref{falseiii}  for $s_0-1$ and $n-1$, which 
contradicts the assumption that $s_0$ is the smallest natural number  for \ref{falsei}--\ref{falseiii}  to hold.  
The reader might wonder what happens to this argument in the case where $n=1$. 
%which in particular is forced by $s_0=2$. 
In fact we establish below that, if \ref{falsei}--\ref{falseiii} is satisfied, 
 then necessarily $n\geq 2$, see the claims \ref{2falsei}--\ref{2falseiii} right after \eqref{eq:Ci0}. % \pageref{2falsei}.

Let $\setA$, $h$, $n$, and $f$ satisfy  \ref{falsei}--\ref{falseiii} for  $s_0$.
%W.l.o.g., we can assume that $\setA$ is bounded. 
We start by noting that 
\begin{align}
m\geq s_0>n\geq 1\label{eq:dimensions}
\end{align}
by  \ref{falseii} and \ref{falseiii}.  
Next, we write 
\begin{align}\label{eq:fdecomp}
f(\vecx)=\tp{(f_1(\vecx)\dots f_n(\vecx))} 
\end{align}
and set $\psi_i=f_i\circ h$, $i=1,\dots,n$, and $\psi=f\circ h$, which,  by Corollary \ref{lem:analyticcomp}, are all real analytic as compositions of real analytic mappings. We now show that there must exist an $i_0\in\{1,\dots,n\}$ and a set $\setA_{i_0}\subseteq\setA$ 
such that 
\begin{align}\label{eq:lebai0}
\lebmeasure^{s_0}(\setA_{i_0})>0,
\end{align}
 $J\psi_{i_0}(\vecz)> 0$, and $Jh(\vecz) >0$ for all $\vecz\in\setA_{i_0}$. To this end, we first decompose 
\begin{align}
\setA= \setA_0\cup\bigcup_{i=1}^n\setA_i, 
\end{align}
where
\begin{align}
\setA_i&=\{\vecz\in\setA: D\psi_i(\vecz)\neq \veczero, Jh(\vecz) >0\}\\
       &=\{\vecz\in\setA: J\psi_i(\vecz)> 0, Jh(\vecz) >0\} \label{eq:setAi}
\end{align}
for $i=1,\dots,n$ and
\begin{align}       
\setA_0&=\{\vecz\in\setA: D\psi(\vecz)=\veczero\}\cup\{\vecz\in\setA: Jh(\vecz)=0\}. 
\end{align}
 By Lemma \ref{lem.detnonzero} below (with $s=s_0$),    $D\psi(\vecz)\neq \veczero$ for $\lebmeasure^{s_0}$-a.a. $\vecz\in\setA$. Furthermore,  
 $Jh(\vecz)\neq 0$ for  $\lebmeasure^{s_0}$-a.a. $\vecz\in\setA$ because of $Jh\not\equiv 0$ and Lemma \ref{lem:lebJ}. Thus,  
$\lebmeasure^{s_0}(\setA_0)=0$ by the  countable subadditivity of Lebesgue measure. Since $\lebmeasure^{s_0}(\setA)>0$ by assumption,  and $\lebmeasure^{s_0}(\setA_0)=0$, it follows, again by the countable subadditivity of Lebesgue measure, that  there must exist an $i_0$ such that \eqref{eq:lebai0} holds. 

Now, for  each $y\in\reals$, let 
\begin{align}\label{eq:My}
\setM_{y}=\psi_{i_0}^{-1}(\{y\}).
\end{align} 
We show in Section \ref{Proofstepone} below that there exist a $y_0\in\reals$ and a $\vecz_0\in\setA_{i_0}\cap\setM_{y_0}$ such that  
\begin{align}
\mathscr{H}^{s_0-1}\big(\setB_{s_0}(\vecz_0,r)\cap \setA_{i_0}\cap \setM_{y_0}\big)&>0\quad \text{for all}\ r>0,\label{posH1}\\
J\psi_{i_0}(\vecz)&>0\quad\text{for all}\  \vecz\in\setM_{y_0},\label{regular1}\\
Jh(\vecz_0)&>0.\label{eq:Jhz0}
\end{align}
It now follows from  \eqref{regular1}, real analyticity of $\psi_{i_0}$, $\setM_{y_0}\neq\emptyset$, and  Lemma \ref{lem:summf}  that $\setM_{y_0}$ is an $(s_0-1)$-dimensional real analytic submanifold of $\reals^{s_0}$. 
Therefore, by  Lemma \ref{lem:man}, 
there exist a real analytic embedding  $\zeta\colon\reals^{s_0-1}\to\reals^{s_0}$ and an $\eta>0$ such that 
\begin{align}
\zeta(\veczero)&=\vecz_0,\label{eq:M1a}\\ 
\setB_{s_0}(\vecz_0,\eta)\cap\setM_{y_0}&\subseteq\zeta\big(\reals^{s_0-1}\big),\label{eq:M2a}
\end{align}
and $\zeta\big(\reals^{s_0-1}\big)$ is relatively open in $\setM_{y_0}$, i.e., there exists an open set $\setV\subseteq \reals^{s_0}$ with $\zeta\big(\reals^{s_0-1}\big)=\setV\cap\setM_{y_0}$. 
Combining  \eqref{posH1} and \eqref{eq:M2a} yields 
\begin{align}
\mathscr{H}^{s_0-1}\big(\zeta\big(\reals^{s_0-1}\big)\cap \setA_{i_0}\big)&>0.\label{eq:HposC}
\end{align}
Next, let  
\begin{align}\label{eq:Ci0}
\setC_{i_0}=\zeta^{-1}\big(\zeta\big(\reals^{s_0-1}\big)\cap \setA_{i_0}\big).  
\end{align}
We now show that 
\begin{enumerate}[label=\alph*)]
\item   $\setC_{i_0}\in\colB(\reals^{s_0-1})$  with $\lebmeasure^{s_0-1}(\setC_{i_0})>0$, \label{2falsei}
\item $\tilde h=h\circ \zeta\colon\reals^{s_0-1}\to\reals^m$ is real analytic and   of $(s_0-1)$-dimensional Jacobian $J\tilde h\not\equiv 0$, and \label{2falseii}
\item $s_0-1>n-1>0$ and the real analytic mapping 
\begin{align}
\tilde f\colon \reals^m&\to\reals^{n-1}\\
\vecx&\mapsto \tp{(f_1(\vecx)\dots f_{i_0-1}(\vecx)\, f_{i_0+1}(\vecx)\dots f_n(\vecx))}
\end{align}
is one-to-one on $\tilde h(\setC_{i_0})$,  \label{2falseiii}
\end{enumerate}
which finally yields the desired contradiction to the statement of $s_0$ being the smallest natural number such that \ref{falsei}--\ref{falseiii}   at the beginning of the proof are satisfied. 

\emph{Proof of \ref{2falsei}}.  
We first establish  that $\setC_{i_0}\in\colB(\reals^{s_0-1})$. Since $\zeta\big(\reals^{s_0-1}\big)$ is relatively open in $\setM_{y_0}$ and, therefore, a Borel set in $\reals^{s_0}$, it follows  
from \eqref{eq:Ci0} that  $\setC_{i_0}$ is the inverse image of a finite intersection of Borel sets under the real analytic  embedding $\zeta$ and, therefore, also a Borel set. 
Next, we show that  $\lebmeasure^{s_0-1}(\setC_{i_0})>0$. 
Since  
\begin{align}
\int_{\setC_{i_0}} J\zeta(\vecw)\,\mathrm d \lebmeasure^{s_0-1}(\vecw)
&=\mathscr{H}^{s_0-1}(\zeta(\setC_{i_0}))\label{eq:step1th2final}\\
&=\mathscr{H}^{s_0-1}\big(\zeta(\reals^{s_0-1})\cap \setA_{i_0}\big)\label{eq:step4th2final}\\
&>0,\label{eq:step5th2final}
\end{align}
where \eqref{eq:step1th2final} follows from the area formula Corollary \ref{cor:area} upon noting that $\zeta$ is one-to-one as an embedding and locally Lipschitz by real analyticity, 
\eqref{eq:step4th2final} is by \eqref{eq:Ci0}, 
and  in \eqref{eq:step5th2final} we applied  \eqref{eq:HposC}.
Using Lemma \ref{lem.fzm}, we conclude from  \eqref{eq:step1th2final}--\eqref{eq:step5th2final} that $\lebmeasure^{s_0-1}(\setC_{i_0})>0$.

\emph{Proof of \ref{2falseii}}. By Corollary \ref{lem:analyticcomp}, 
$\tilde h$ is real analytic  as the composition of real analytic mappings.  
It remains to   show that $J\tilde h\not\equiv 0$. To this end, we establish  $J\tilde h(\veczero)>0$.  
First note that the chain rule yields
\begin{align}
D\tilde h(\veczero)
&= (Dh)(\zeta(\veczero))D\zeta(\veczero)\\
&= Dh(\vecz_0)D\zeta(\veczero),
\end{align}
where the second equality is by \eqref{eq:M1a}. Since $\zeta\colon\reals^{s_0-1}\to \reals^{s_0}$  is an embedding, it follows that 
\begin{align}
\rank(D\zeta(\veczero))=s_0-1.\label{eq:D1}
\end{align}
Moreover, as $h\colon \reals^{s_0}\to\reals^m$ with $m\geq s_0$, \eqref{eq:Jhz0} implies 
\begin{align}
\rank(Dh(\vecz_0))=s_0.\label{eq:D2}
\end{align}
Applying  Lemma \ref{lem:sylvester}  to  $D\tilde h(\veczero)= Dh(\vecz_0)D\zeta(\veczero)$ therefore yields 
$\rank(D\tilde h(\veczero))\geq s_0-1$, which   in turn implies $J\tilde h(\veczero)>0$ because $D\tilde h(\veczero)\in\reals^{m\times (s_0-1)}$. 

\emph{Proof of \ref{2falseiii}}.
$s_0-1>n-1$  simply follows from $s_0>n$ (see \eqref{eq:dimensions}). 
To prove that $\tilde f$ is one-to-one on $\tilde h(\setC_{i_0})$, we first  show that $f_{i_0}$ is constant on $\tilde h(\setC_{i_0})$.   
In fact, since $\zeta(\setC_{i_0})\subseteq \setM_{y_0}$  and $\setM_{y_0}=(f_{i_0}\circ h)^{-1}(\{y_0\})$, it follows that 
\begin{align}\label{eq:fi0const}
f_{i_0}(\tilde h(\vecw))&=y_0\quad\text{for all}\ \vecw\in  \setC_{i_0}. 
\end{align}
As $f$ is one-to-one on $h(\setA)$ and  $f_{i_0}$ is constant on $\tilde h(\setC_{i_0})$ with 
$\tilde h(\setC_{i_0})=(h\circ \zeta)(\setC_{i_0})\subseteq h(\setA)$ by \eqref{eq:Ci0},  we conclude that 
 $\tilde f$ must also be  one-to-one on $\tilde h(\setC_{i_0})$.  
It remains to   show that we must have $n>1$, which obviously implies   $n-1>0$. 
Suppose, to the contrary,  that $n=1$. 
Since $\tilde h$ is real analytic with $J\tilde h\not\equiv 0$ and $\lebmeasure^{s_0-1}(\setC_{i_0})>0$, 
it follows from  Property \ref{immreal2} of Lemma \ref{lem:immreal}   that $\tilde h$ is one-to-one on a subset of $\setC_{i_0}$ of positive  Lebesgue measure. Thus, the set 
$\tilde h(\setC_{i_0})$ is uncountable. 
Now, since by \eqref{eq:fi0const} $f_{i_0}$ is constant on $\tilde h(\setC_{i_0})$, and  
a  constant function cannot be one-to-one on a set of cardinality larger than one, 
the uncountability of $\tilde h(\setC_{i_0})$ implies that $f_{i_0}$ cannot be one-to-one on 
$\tilde h(\setC_{i_0})$.  
But $f=f_{i_0}$ for $n=1$ (with $i_0=1$) and $f$ is  one-to-one on  $\tilde h(\setC_{i_0})\subseteq h(\setA)$ by assumption, which results in a contradiction. Therefore, we necessarily have  $n>1$ and we can    conclude that \ref{2falsei}--\ref{2falseiii} must hold, which  finalizes the proof of the theorem.

\subsection{Proof of \eqref{posH1}--\eqref{eq:Jhz0}}\label{Proofstepone}
We start by noting that  
\begin{align}
&\int_{\reals}\mathscr{H}^{s_0-1}(\setA_{i_0}\cap \setM_{y}) \,\mathrm d\lebmeasure^{1}(y)\label{eq:sstep0}\\
&=\int_{\setA_{i_0}} J\psi_{i_0}(\vecz)\,\mathrm d\lebmeasure^{s_0}(\vecz)\label{eq:sstep1}\\
&=\int_{\reals^{s_0}}\ind{\setA_{i_0}}(\vecz) J\psi_{i_0}(\vecz)\,\mathrm d\lebmeasure^{s_0}(\vecz)\\
&>0,\label{eq:coarea2a}
\end{align}
where \eqref{eq:sstep1} is by the coarea formula Corollary \ref{cor:coarea} 
upon noting that  $\psi_{i_0}$ is locally Lipschitz by real analyticity,  and 
\eqref{eq:coarea2a} follows from $\ind{\setA_{i_0}}(\vecz) J\psi_{i_0}(\vecz)>0$ for all $\vecz\in\setA_{i_0}$,    
 $\lebmeasure^{s_0}(\setA_{i_0})>0$ (see \eqref{eq:lebai0}),  and Lemma \ref{lem.fzm}. 
Also, by   Lemma \ref{lem.fzm}  and \eqref{eq:sstep0}--\eqref{eq:coarea2a}, there must exist a set $\setD\subseteq\reals$ with $\lebmeasure^{1}(\setD)>0$ such that 
\begin{align}
\mathscr{H}^{s_0-1}\big(\setA_{i_0}\cap \setM_{y}\big)>0\quad\text{for all}\  y\in\setD\label{posH}.
\end{align}
Furthermore, 
\begin{align}
&\lebmeasure^{1}\big(\{y\in\reals:\exists \vecz\in\setM_y\ \text{with}\ J\psi_{i_0}(\vecz)=0\}\big)\label{eq:Myapp0}\\
&=\lebmeasure^{1}\big(\psi_{i_0}\big(\{\vecz\in\reals^{s_0}:J\psi_{i_0}(\vecz)=0\}\big)\big)\label{eq:Myapp}\\
&=0,\label{eq:sardapp}
\end{align}
where  \eqref{eq:Myapp} is by 
\eqref{eq:My}
and 
\eqref{eq:sardapp} follows from Property  \ref{thm:sard2} of 
 Theorem \ref{thm:sard} with  $\psi_{i_0}$ being 
$C^\infty$ (recall that $\psi_{i_0}=f_{i_0}\circ h$ is real analytic as a composition of real analytic mappings and, therefore, $C^\infty$ by Lemma \ref{lem:analyticder}).
Now, \eqref{eq:Myapp0}--\eqref{eq:sardapp} 
together with  $\lebmeasure^{1}(\setD)>0$ implies the  existence of a $y_0\in\setD$ such that 
\eqref{regular1} holds. For this $y_0$, we must have $\mathscr{H}^{s_0-1}(\setA_{i_0}\cap \setM_{y_0})>0$ by \eqref{posH}.  
Therefore,  Lemma \ref{lem:loef}  implies that there must exist a $\vecz_0\in\setA_{i_0}\cap \setM_{y_0}$ such that \eqref{posH1} holds. 
As this  $\vecz_0\in\setA_{i_0}\cap \setM_{y_0}\subseteq \setA_{i_0}$, \eqref{eq:Jhz0} finally follows from    \eqref{eq:setAi}.\qed

\begin{lem}\label{lem.detnonzero} 
Let $\setA\in\colB(\reals^s)$ be of positive  Lebesgue measure, $h\colon\reals^s\to \reals^m$, with $s\leq m$,  real analytic of $s$-dimensional Jacobian $Jh\not\equiv 0$, and $f\colon\reals^m\to \reals^n$   real analytic. If $f$ is one-to-one on $h(\setA)$, then
 $D (f\circ h)(\vecz)\neq\veczero$ for  $\lebmeasure^s$-a.a. $\vecz\in\setA$.
\end{lem}
\begin{proof}
Towards a contradiction, suppose that $f$ is one-to-one on $h(\setA)$ and   $D (f\circ h) \equiv \veczero$ on a set $\setB\subseteq\setA$ with $\lebmeasure^{s}(\setB)>0$. 
Since $f\circ h$ is real analytic by Corollary \ref{lem:analyticcomp}, $D (f\circ h)$ is real analytic as a consequence of Lemma \ref{lem:analyticder}. It therefore follows from Lemma 
 \ref{lem:vanishrealanalytic} that $D (f\circ h)$ is identically zero. Hence, $f\circ h$ must be constant. In particular, $f\circ h$ must be constant on $\setB$ and, therefore, $f$ must be constant on $h(\setB)$. 
Since $f$ is one-to-one and constant on $h(\setB)$, it follows that $h$ must be constant on $\setB$. 
Thus, $Jh(\vecz)= 0$ for all  $\vecz\in\setB$, and Lemma \ref{lem:vanishrealanalytic}  implies  $Jh\equiv 0$.
This is a contradiction to $Jh\not\equiv 0$.
\end{proof}

\appendices

\section{Proof of Lemma \ref{LemmaExarec}}\label{ProofLemmaExarec}

\emph{Proof of \ref{exarec}}. 
Suppose that  $\setU$ is $s$-rectifiable. Then,  $\setU=\varphi(\setA)$ with $\setA\subseteq\reals^s$ compact and $\varphi$ Lipschitz. For $t\in\naturals$ with $t>s$,  take
\begin{align}
\setB=
\mleft\{
\begin{pmatrix}
\vecz\\
\veczero
\end{pmatrix} 
: \vecz\in\setA
\mright\}\subseteq\reals^{t}
\end{align}
and define  
\begin{align}
\psi\colon \setB&\to\reals^m\\ 
\begin{pmatrix}
\vecz\\
\veczero
\end{pmatrix}
&\mapsto\varphi(\vecz).
\end{align}
Then, 
$\setU=\psi(\setB)$  is $t$-rectifiable as a consequence of $\setB$ compact and $\psi$  Lipschitz. 

\emph{Proof of \ref{exar1}}. 
First note that we can cover $\reals^s$ according to
\begin{align}
\reals^s=\bigcup_{k\in\naturals} \setA_k,  
\end{align}
with $\setA_k\subseteq \reals^s$  compact for all  $k\in\naturals$ (take, for example, $\setA_k=\overline{\setB_s(\veczero,k)}$). 
Setting, for every $i,k\in\naturals$, $\varphi_{i,k}=\varphi_i|_{\setA_k}$ (locally Lipschitz mappings are Lipschitz on compact sets by Lemma \ref{lem:lipcompact}), 
we can write 
\begin{align}
\setV=\bigcup_{i,k\in\naturals}\varphi_{i,k}(\setA_k),  
\end{align}
which implies that $\setV$ is countably $s$-rectifiable. 

\emph{Proof of \ref{exaProdrec}}. 
Suppose that $\setU$ is countably $s$-rectifiable and $\setV$ is countably $t$-rectifiable. Then,  we 
can write 
\begin{align}
\setU&=\bigcup_{i\in\naturals}\varphi_{i}(\setA_i)\\
\setV&=\bigcup_{j\in\naturals}\psi_{j}(\setB_j),
\end{align}
where the $\setA_i\subseteq\reals^s$ and  the $\setB_j\subseteq\reals^t$  are compact and  the $\varphi_i\colon \reals^s\to\reals^m$ and the  $\psi_j\colon \reals^t\to\reals^n$ are Lipschitz.  
Thus, 
\begin{align}
\setW
&=\bigcup_{i,j\in\naturals}\theta_{i,j}(\setC_{i,j}), 
\end{align}
where we defined  
\begin{align}
\theta_{i,j}\colon \reals^{s+t}&\to \reals^{m+n}\\
\tp{(\veca\ \vecb)}&\mapsto \tp{(\varphi_i(\veca)\ \psi_j(\vecb))}
\end{align}
and set 
\begin{align}
\setC_{i,j}=\{\tp{(\veca\ \vecb)}: \veca\in \setA_i, \vecb\in\setB_j\}. 
\end{align}
Now, the $\theta_{i,j}$ are Lipschitz because the  $\varphi_i$ and the $\psi_j$ are Lipschitz for all $i,j\in\naturals$, 
and the $\setC_{i,j}$ are compact by Tychonoff's theorem \cite[Theorem 4.42]{fo99} thanks to  $\setA_i$ and $\setB_j$ compact for all $i,j\in\naturals$.  
Therefore,  $\setW$ is countably $(s+t)$-rectifiable.

\emph{Proof of \ref{exaC1}}. 
Let $\setM$ be an $s$-dimensional $C^1$-submanifold of $\reals^m$. By \cite[Definition 5.3.1]{krpa08}, we can write 
\begin{align}
\setM=\bigcup_{\vecx\in\setM}\varphi_\vecx(\setU_\vecx), 
\end{align}
where, for every $\vecx\in\setM$, $\setU_\vecx\subseteq\reals^s$ is open, and $\varphi_\vecx\colon\setU_\vecx\to \reals^m$ is $C^1$ and satisfies 
$\vecx\in\varphi_\vecx(\setU_\vecx)$ and 
$\varphi_\vecx(\setU_\vecx)=\setV_\vecx\cap\setM$ with $\setV_\vecx\subseteq\reals^m$ open. 
Now, there must exist a countable set $\{\vecx_i: i\in\naturals\}\subseteq \setM$
such that 
\begin{align}
\setM=\bigcup_{i\in\naturals}\setV_{\vecx_i}\cap\setM.  
\end{align}
For if such a countable set does not exist, the open set 
\begin{align}
\setV=\bigcup_{\vecx\in\setM}\setV_{\vecx}
\end{align}
would not admit a countable subcover, which would contradict that $\setV$   as an open set in $\reals^m$ is 
 Lindel\"of \cite[Definition 5.6.19, Proposition 5.6.22]{so14}.  
With the  countable set $\{\vecx_i: i\in\naturals\}\subseteq \setM$ we can now write  
\begin{align}\label{eq:MUi}
\setM=\bigcup_{i\in\naturals}\varphi_{i}(\setU_{i}),   
\end{align}
where we set $\varphi_i=\varphi_{\vecx_i}$ and  $\setU_{i}=\setU_{\vecx_i}$. 
Now fix $i\in\naturals$ arbitrarily. 
As $\setU_{i}$ is open, for every $\vecu_i
\in\setU_{i}$,  there exists an $r_{\vecu_i}>0$ such that 
\begin{align}\label{eq:Uiopen}
\setB^s\mleft(\vecu_{i},r_{\vecu_i}\mright)\subseteq \setU_i.
\end{align}
We can thus write  
\begin{align}
\setU_i= \bigcup_{\vecu_i\in\setU_i}\setB^s\mleft(\vecu_{i},\frac{r_{\vecu_i}}{2}\mright).
\label{eq:chainineq}
\end{align} 
Since $\setU_i$   as an open set in $\reals^m$ is 
 Lindel\"of \cite[Definition 5.6.19, Proposition 5.6.22]{so14}, 
there exists a countable set $\{\vecu_{i,j}: j\in\naturals\}\subseteq \setU_i$ such that 
\begin{align}
\setU_i = \bigcup_{j\in\naturals}\setB^s\mleft(\vecu_{i,j},\frac{r_{i,j}}{2}\mright),
\end{align} 
where we set $r_{i,j}=r_{\vecu_{i,j}}$ for all $j\in\naturals$. 
Using  $\setB^s\mleft(\vecu_{i,j},r_{i,j}\mright)
\subseteq \setU_i$ for all $j\in\naturals$ (see \eqref{eq:Uiopen}), it follows that 
\begin{align}
\setU_i = \bigcup_{j\in\naturals}\overline{\setB^s\mleft(\vecu_{i,j},\frac{r_{i,j}}{2}\mright)}. \label{eq:Uibar}
\end{align} 
Since  $i\in\naturals$ was arbitrary, using   \eqref{eq:MUi}  and \eqref{eq:Uibar},  we get  
\begin{align}
\setM=\bigcup_{i,j\in\naturals}\varphi_{i}\mleft(\overline{\setB^s\mleft(\vecu_{i,j},\frac{r_{i,j}}{2}\mright)}\mright).    
\end{align}
Finally, all the  $\varphi_i$  as  $C^1$ mappings are locally Lipschitz and, therefore, Lipschitz on the compact sets $\overline{\setB^s\mleft(\vecu_{i,j},\frac{r_{i,j}}{2}\mright)}$ for all $j\in\naturals$, 
which establishes countable rectifiability of $\setM$.  
%thereby proving that is countably $s$-rectifiable. 

\emph{Proof of \ref{exasum}}.
Since every $s_i$-rectifiable set, with $s_i\leq s$, is $s$-rectifiable by \ref{exarec}, it follows that 
  $\setA$ is  a countable union of $s$-rectifiable sets and, therefore,  countably $s$-rectifiable. 
\qed

\section{Proof of Lemma \ref{lem:recsetlip}}\label{lem:recsetlipproof}
\emph{Proof of \ref{recf1}}. Suppose that  $\setU$ is $s$-rectifiable. Then,  there exist a compact set $\setA\subseteq\reals^s$ and a Lipschitz mapping $\varphi\colon \setA\to\reals^m$ such that $\setU=\varphi(\setA)$. 
As $f\circ\varphi\colon \setA\to \reals^n$ is Lipschitz owing to  Lemmata \ref{lem:lipcompact} and \ref{lem:lip},  $f(\setU)=(f\circ\varphi)(\setA)$  is $s$-rectifiable.

\emph{Proof of \ref{recf2}}. Suppose that  $\setU$ is countably $s$-rectifiable.  We can write $\setU=\bigcup_{i\in\naturals}\setU_i$, where $\setU_i$ is $s$-rectifiable for all  $i\in\naturals$. 
As $f(\setU)=\bigcup_{i\in\naturals}f(\setU_i)$ and $f(\setU_i)$ is $s$-rectifiable  for all $i\in\naturals$ by \ref{recf1}, it follows that $f(\setU)$ is countably $s$-rectifiable.

\emph{Proof of \ref{recf3}}. 
Suppose that $\setU\in\colB(\reals^m)$ is countably $(\mathscr{H}^{s},s)$-rectifiable. 
We have to show that $f(\setU)=\setA\cup\setB$, where $\setA\in\colB(\reals^n)$ is countably $(\mathscr{H}^{s},s)$-rectifiable and $\mathscr{H}^{s}(\setB)=0$. 
As $\setU\in\colB(\reals^m)$ is  countably $(\mathscr{H}^{s},s)$-rectifiable, \cite[Lemma 15.5]{ma99} implies that 
 $\mathscr{H}^{s}|_{\setU}$ is $\sigma$-finite. We can therefore employ Lemma \ref{lem:inner}  to decompose 
\begin{align}
\setU=\setV_0\cup\bigcup_{j\in\naturals}\setV_j, 
\end{align}
where $\mathscr{H}^{s}(\setV_0)=\mathscr{H}^{s}|_{\setU}(\setV_0)=0$ and $\setV_j\subseteq\reals^m$ is compact for all $j\in\naturals$. This decomposition allows us to write $f(\setU)=\setA\cup\setB$, where
\begin{align}
\setA&=\bigcup_{j\in\naturals}f(\setV_j)\\
\setB&=f(\setV_0).
\end{align}
Now, thanks to Lemma \ref{lem:fK}, the $f(\setV_j)$ are  compact. 
Therefore, $\setA$ 
as a countable union of compact sets is Borel. Furthermore,  $\mathscr{H}^{s}(\setB)=0$ because of $\mathscr{H}^{s}(\setV_0)=0$ and Lemma  \ref{lem:fH0} below.  
It remains to show that  $\setA$  is countably $(\mathscr{H}^{s},s)$-rectifiable. Since $\setU$ is  countably $(\mathscr{H}^{s},s)$-rectifiable, there exists a countably $s$-rectifiable set $\setU_1$ such that $\mathscr{H}^{s}(\setU\mysetminus\setU_1)=0$. Furthermore, as  $\setU_1$ is countably $s$-rectifiable, $f(\setU_1)$ is countably $s$-rectifiable by \ref{recf2}. Finally, 
\begin{align}
\mathscr{H}^{s}(\setA\mysetminus f(\setU_1))
&\leq \mathscr{H}^{s}(f(\setU)\mysetminus f(\setU_1))\label{eq:Arec1}\\
&\leq \mathscr{H}^{s}(f(\setU\mysetminus \setU_1))\label{eq:Arec2}\\
&=0,\label{eq:Arec3}
\end{align}
where  \eqref{eq:Arec1} follows from the monotonicity of $\mathscr{H}^{s}$ and $f(\setU)=\setA\cup\setB$,  
\eqref{eq:Arec2}  is by  monotonicity of $\mathscr{H}^{s}$ and $f(\setU)\mysetminus f(\setU_1)\subseteq f(\setU\mysetminus\setU_1)$, 
 and \eqref{eq:Arec3} follows from  Lemma \ref{lem:fH0} below. This proves that 
$\setA$ is countably $(\mathscr{H}^{s},s)$-rectifiable, thereby concluding the proof. \qed

\begin{lem}\label{lem:fH0}
If $\setU\subseteq\reals^m$ with $\mathscr{H}^{s}(\setU)=0$ and $f\colon\reals^m\to\reals^n$ is locally Lipschitz, then $\mathscr{H}^{s}(f(\setU))=0$.
\end{lem}
\begin{proof}
We employ the following chain of arguments:  
\begin{align}
\mathscr{H}^{s}(f(\setU))
& =\mathscr{H}^{s}\Big(f\Big(\setU\cap \bigcup_{l\in\naturals}\setB_m(\veczero,l)\Big)\Big)\label{eq:step1fU0}\\
& =\mathscr{H}^{s}\Big(\bigcup_{l\in\naturals}f(\setB_m(\veczero,l)\cap\setU)\Big)\label{eq:step2fU0}\\
&\leq   \sum_{l\in\naturals}\mathscr{H}^{s}(f(\setB_m(\veczero,l)\cap\setU))\label{eq:step3fU0}\\
&\leq   \sum_{l\in\naturals}L_l^s\mathscr{H}^{s}(\setB_m(\veczero,l)\cap\setU)\label{eq:step4fU0}\\
&\leq   \sum_{l\in\naturals}L_l^s\mathscr{H}^{s}(\setU)\label{eq:step5fU0}\\
&=0,\label{eq:step7fU0}
\end{align} 
where \eqref{eq:step3fU0} follows from the countable subadditivity of $\mathscr{H}^{s}$, 
in \eqref{eq:step4fU0} we applied Property  \ref{HL} of Lemma \ref{lem:prophausdorff}, where, for every $l\in\naturals$,  $L_l$ denotes  the Lipschitz constant of $f|_{\overline{\setB_m(\veczero,l)}}$, and in \eqref{eq:step5fU0} we used the monotonicity of $\mathscr{H}^{s}$.
\end{proof}

\section{Proof of Lemma \ref{lem:uniques}}\label{proof:uniques}
Towards a contradiction, suppose that there exist $r,t\in\naturals$ with $r\neq t$ such that $\rvecx$ is $r$-rectifiable and $t$-rectifiable. 
We can assume, w.l.o.g.,  that $r< t$. Now,   
Lemma \ref{lem:recsup} implies the existence of 
\begin{enumerate}[label=\roman*)]
\item a countably $r$-rectifiable set $\setU$ satisfying $\opP[\rvecx\in\setU]=1$ and 
\item a countably $t$-rectifiable set $\setV$ satisfying $\opP[\rvecx\in\setV]=1$. 
\end{enumerate}
With  Definition \ref{Dfn:recset}, we can  conclude that there exist 

\begin{enumerate}[label=\roman*)]
\item
 compact  sets $\setA_i\subseteq\reals^r$ and  Lipschitz  mappings $\varphi_i\colon\setA_i\to \reals^{m}$, $i\in\naturals$, such that 
\begin{align}
\setU=\bigcup_{i\in\naturals}\varphi_i(\setA_i), 
\end{align}
and 
\item 
compact sets $\setB_j\subseteq\reals^t$ and  Lipschitz  mappings $\psi_j\colon\setB_j\to \reals^{m}$, $j\in\naturals$, such that 
\begin{align}
\setV=\bigcup_{j\in\naturals}\psi_j(\setB_j).
\end{align}
\end{enumerate}
The union bound now  yields 
\begin{align}
\opP[\rvecx\in\setU\cap\setV]
&=\opP\!\Big[\rvecx\in\bigcup_{i\in\naturals}\bigcup_{j\in\naturals}\varphi_i(\setA_i)\cap\psi_j(\setB_j)\Big]\label{eq:unionuniquea}\\
&\leq \sum_{i\in\naturals}\sum_{j\in\naturals}\opP\mleft[\rvecx\in\varphi_i(\setA_i)\cap\psi_j(\setB_j)\mright].\label{eq:unionunique}
\end{align}
Since $\opP[\rvecx\in\setU\cap\setV]=1$, which follows from $\opP[\rvecx\in\setU]=1$ and $\opP[\rvecx\in\setV]=1$, 
 \eqref{eq:unionuniquea}--\eqref{eq:unionunique} guarantee the existence of   an  $i_0$ and a $j_0$, both in $\naturals$, such that $\mu_\rvecx(\varphi_{i_0}(\setA_{i_0})\cap\psi_{j_0}(\setB_{j_0}))>0$.  
As $\mu_\rvecx\ll \mathscr{H}^{t}|_\setV$ by the $t$-rectifiability of $\rvecx$ and $\mathscr{H}^{t}|_\setV\ll\mathscr{H}^{t}$ by the monotonicity of $\mathscr{H}^{t}$, it follows that  $\mathscr{H}^{t}(\varphi_{i_0}(\setA_{i_0})\cap\psi_{j_0}(\setB_{j_0}))>0$. Property \ref{Hjump} of Lemma \ref{lem:prophausdorff}  therefore implies  (recall that $r<t$ by assumption) $\mathscr{H}^{r}(\varphi_{i_0}(\setA_{i_0})\cap\psi_{j_0}(\setB_{j_0}))=\infty$. 
The monotonicity of $\mathscr{H}^{r}$  then yields 
\begin{align}
\mathscr{H}^{r}(\varphi_{i_0}(\setA_{i_0}))
&\geq \mathscr{H}^{r}(\varphi_{i_0}(\setA_{i_0})\cap\psi_{j_0}(\setB_{j_0}))\\
&=\infty.\label{eq:continf} 
\end{align}
But we also have 
\begin{align}
\mathscr{H}^{r}(\varphi_{i_0}(\setA_{i_0}))
&\leq L^r\mathscr{H}^{r}(\setA_{i_0})\label{eq:Hstep1}\\
&= L^r\lebmeasure^{r}(\setA_{i_0})\label{eq:Hstep2}\\
&<\infty,\label{eq:Hstep3}   
\end{align}
where \eqref{eq:Hstep1} follows from  Property  \ref{HL} of Lemma \ref{lem:prophausdorff} with $L$ the Lipschitz constant of $\varphi_{i_0}$, 
\eqref{eq:Hstep2} is by Property \ref{HBorel} of Lemma \ref{lem:prophausdorff}  and the fact that $\setA_{i_0}$ as a compact set is in $\colB(\reals^r)$, 
and finally  \eqref{eq:Hstep3} is a consequence, again,  of $\setA_{i_0}$  being compact. This contradicts \eqref{eq:continf}  and thereby concludes the proof. 
\qed

\section{Proof of Lemma \ref{lem:convimpliesA}}\label{lem:convimpliesAproof}
Suppose that there exists  a Borel measurable mapping $g\colon \reals^{n\times m}\times \reals^n\to\reals^m$  such that  
$\opP[g(\matA,\matA\rvecx)\neq\rvecx]<1$ and set $\setU=\{\vecx\in\reals^m:g(\matA,\matA\vecx)=\vecx\}$. 
We first show that $\setU\in\colB(\reals^m)$.  
To this end, consider the mapping 
$h_{\matA}\colon \reals^m\to\reals^m$, $\vecx\mapsto g(\matA,\matA\vecx)$, 
which  as the composition of  two Borel measurable mappings, namely,  $\vecx\mapsto(\matA,\matA\vecx)$ and $g$, is Borel measurable. 
Application of Lemma \ref{lem.meascart} (with $\setX=\setY=\reals^m$, $f_1$ the identity mapping on $\reals^m$, and $f_2=h_{\matA}$) therefore establishes that 
\begin{align}
F\colon \reals^m&\to\reals^m\times \reals^m\\
\vecx&\mapsto (\vecx,h_{\matA}(\vecx))
\end{align}
is Borel measurable. 
Now,  consider the diagonal $\setD=\{(\vecx,\vecx):\vecx\in \reals^m\}$ and note that $\setD$ 
as the inverse image of $\{\veczero\}$ under the Borel measurable mapping 
$d\colon\reals^m\times\reals^m\to\reals^m$, $(\vecu,\vecv)\mapsto \vecu-\vecv$ is in $\colB(\reals^m\times\reals^m)$.    
Since $\setU=F^{-1}(\setD)$, we conclude that $\setU\in\colB(\reals^m)$. 
Now,  $\opP[\rvecx\in\setU]>0$ as  $\opP[\rvecx\notin\setU]=\opP[g(\matA,\matA\rvecx)\neq\rvecx]<1$ by assumption. It remains to show that $\matA$ is one-to-one on $\setU$. 
To this end, consider arbitrary but fixed $\vecu,\vecv\in\setU$ and suppose that $\matA\vecu=\matA\vecv$. This implies  $g(\matA,\matA\vecu)=g(\matA,\matA\vecv)$.  As $\vecu,\vecv\in\setU$ by assumption, this is only possible if $\vecu=\vecv$. Thus, $\matA$ is one-to-one on $\setU$, which concludes the proof.  
\qed

\section{Proof of Lemma \ref{lem:propanalytic}}\label{proof:propanalytic}
\emph{Proof of \ref{propra2}}.
The proof is by induction. Suppose that $\mu$ is $s$-analytic for $s\in\naturals\mysetminus\{1\}$. We have to show  that this implies $(s-1)$-analyticity of $\mu$.  
Consider $\setC\in\colB(\reals^{m})$ with $\mu(\setC)>0$. 
Then, by Definition \ref{dfn:measureanalytic}, there exist a  set $\setA\in\colB(\reals^s)$ of positive  Lebesgue measure and a real analytic mapping $h\colon\reals^s\to \reals^m$  of $s$-dimensional Jacobian $Jh\not\equiv 0$   such that $h(\setA)\subseteq\setC$.  
By Property \ref{immreal2} of Lemma  \ref{lem:immreal}, we can assume, w.l.o.g.,  
that $h|_\setA$ is an embedding. 
%$Jh(\vecz) > 0$ for all $\vecz\in\setA$. 
For each $z\in\reals$, let 
\begin{align}\label{eq:Az}
\setA_z=\{\vecv\in\reals^{s-1} : \tp{(\vecv\ z)}\in\setA\}.
\end{align} 
We now use  Fubini's Theorem to  show that there exists a $z_0\in\reals$ such that  $\lebmeasure^{s-1}(\setA_{z_0})>0$. 
Concretely, 
\begin{align}
0&<\int_{\reals^{s}} \ind{\setA}(\vecz) Jh(\vecz) \,\mathrm d \lebmeasure^{s}(\vecz)\label{eq:xx}\\
&=\int_{\reals}\mleft(\int_{\reals^{s-1}}  \ind{\setA}\big(\tp{(\vecv\ z)}\big) Jh\big(\tp{(\vecv\ z)}\big)\,\mathrm d\lebmeasure^{s-1}(\vecv)\mright)\,\mathrm d \lebmeasure^{1}(z)\label{eq:yy}\\
&=\int_{\reals}\mleft(\int_{\reals^{s-1}}\ind{\setA_z}(\vecv)Jh\big(\tp{(\vecv\ z)}\big)\,\mathrm d\lebmeasure^{s-1}(\vecv)\mright)\,\mathrm d \lebmeasure^{1}(z),
\end{align}
where \eqref{eq:xx} follows from Lemma \ref{lem.fzm}  with $\lebmeasure^{s}(\setA)>0$ and $Jh(\vecz)> 0$ for all $\vecz\in\setA$,  and in \eqref{eq:yy} we applied 
Theorem \ref{thm.fubini}. 
We conclude that there must exist a $z_0\in\reals$ such that 
\begin{align}
\int_{\reals^{s-1}}\ind{\setA_{z_0}}(\vecv)Jh\big(\tp{(\vecv\ z_0)}\big)\,\mathrm d\lebmeasure^{s-1}(\vecv)>0. 
\end{align}
Again using Lemma \ref{lem.fzm}   establishes that $\lebmeasure^{s-1}(\setA_{z_0})>0$. 
It remains to show that there exists a real analytic mapping $\tilde h\colon\reals^{s-1}\to \reals^m$  with $\tilde h(\setA_{z_0})\subseteq\setC$ and of $(s-1)$-dimensional Jacobian $J\tilde h\not\equiv 0$. 
To this end, consider  the mapping
\begin{align}
\varphi\colon\reals^{s-1}&\to \reals^s\\
\vecv&\mapsto \tp{(\vecv\ z_0)}  
\end{align}
and set $\tilde h=h\circ\varphi\colon\reals^{s-1}\to \reals^m$. 
Now, $\varphi$ is real analytic thanks to Corollary \ref{cor:poly}. Thus, 
 $\tilde h$ as a composition of  real analytic mappings is real analytic by Corollary \ref{lem:analyticcomp}. 
Furthermore, since $\varphi(\setA_{z_0})\subseteq\setA$ by \eqref{eq:Az}, it follows that   $\tilde h(\setA_{z_0})\subseteq \setC$.  
It remains to establish that $J\tilde h\not\equiv 0$.  To this end, we first note that the chain rule implies  $D\tilde h(\vecv)=Dh(\vecv,z_0)D\varphi(\vecv)$. 
Therefore, 
\begin{align}
\rank(D\tilde h(\vecv))
&=\rank(Dh(\vecv,z_0)D\varphi(\vecv))\\
&=\rank\!\big(Dh(\vecv,z_0)\tp{(\matI_{s-1}\ \veczero)}\big)\label{eq:stepIsm1}\\
&\geq s-1\quad \text{for all}\ \vecv\in\setA_{z_0},\label{eq:stepIsm2}
\end{align}
where  
\eqref{eq:stepIsm2} is by  
Lemma \ref{lem:sylvester} upon noting that   $Dh(\vecv,z_0)\in\reals^{m\times s}$, with $\rank(Dh(\vecv,z_0))=s$ (recall that  $\tp{(\vecv\ z_0)}\in\setA$ for all $\vecv\in\setA_{z_0}$  by \eqref{eq:Az} and that the $s$-dimensional Jacobian satisfies $Jh(\vecz) > 0$ for all $\vecz\in\setA$),  
and $\tp{(\matI_{s-1}\ \veczero)}\in\reals^{s\times (s-1)}$.   
Since  $\rank(D\tilde h(\vecv))\geq s-1$ and  $D\tilde h(\vecv)\in\reals^{m\times (s-1)}$, we conclude that  $J\tilde h(\vecv)>0$ for all $\vecv\in\setA_{z_0}$. Thus, $J\tilde h\not\equiv 0$, which  finalizes the proof.  

\emph{Proof of \ref{propra3}}.
Suppose that $\mu$ is $s$-analytic and  consider $\setC\in\colB(\reals^{m})$ with $\mu(\setC)>0$. We have to show that $\mathscr{H}^{s}(\setC)>0$. 
By Definition \ref{dfn:measureanalytic}, there exist a set $\setA\in\colB(\reals^s)$ of positive  Lebesgue measure 
and a real analytic mapping $h\colon\reals^s\to \reals^m$  of $s$-dimensional Jacobian $Jh\not\equiv 0$   such that $h(\setA)\subseteq\setC$.  
By Property \ref{immreal2} of Lemma  \ref{lem:immreal}, we can assume, w.l.o.g., that $h|_\setA$ is an embedding. 
We now use  the area formula Corollary  \ref{thm:area} 
 to conclude that $\mathscr{H}^{s}(\setC)>0$. 
Concretely, 
\begin{align}
0&< \int_\setA Jh(\vecz)\,\mathrm d\lebmeasure^{s}(\vecz)\label{eq:step1rich}\\
&=\mathscr{H}^{s}\big(h(\setA)\big)\label{eq:step3rich}\\
&\leq \mathscr{H}^{s}(\setC),\label{eq:step4rich}
\end{align}
where \eqref{eq:step1rich} is by Lemma \ref{lem.fzm}, $\lebmeasure^s(\setA)>0$,  and $Jh(\vecz)> 0$ for all $\vecz\in\setA$,   
 in \eqref{eq:step3rich} we applied the area formula Corollary \ref{thm:area} upon noting that $h|_\setA$ is one-to-one  as an embedding and locally Lipschitz by real analyticity of  $h$, 
 and \eqref{eq:step4rich} is by   monotonicity of $\mathscr{H}^{s}$ together with  $h(\setA)\subseteq\setC$.

\emph{Proof of \ref{propra4}}.
,,$\Rightarrow$'': Suppose that $\mu$ is $s$-analytic and there exists a set $\setU\subseteq\reals^m$ such that $\mu=\mu|_\setU$ and $\mathscr{H}^{s}|_{\setU}$ is $\sigma$-finite. 
Since $\mathscr{H}^{s}$ is Borel regular (see Definition \ref{Def.meas}), we may assume, w.l.o.g., that  $\setU\in\colB(\reals^m)$.
By $\sigma$-finiteness of $\mathscr{H}^{s}|_{\setU}$, there exist sets $\setV_i\in\colB(\reals^m)$, $i\in\naturals$, such that 
\begin{align}
\setU=\bigcup_{i\in\naturals} \setU\cap \setV_i
\end{align}
and $\mathscr{H}^{s}(\setU\cap\setV_i)<\infty$ for all $i\in\naturals$. 
For every $i\in\naturals$,  
since  $\setU\cap \setV_i$ as the intersection of two Borel sets is in $\colB(\reals^m)$ and   $\mathscr{H}^{s}(\setU\cap\setV_i)<\infty$, we can write \cite[p. 83]{amfupa00}
$\setU\cap \setV_i=\setW_i\cup\widetilde{\setW}_i$, where  
$\setW_i$ is countably $(\mathscr{H}^{s},s)$-rectifiable and 
$\widetilde{\setW}_i$ is purely $\mathscr{H}^{s}$-unrectifiable, i.e.,   
$\mathscr{H}^{s}(\widetilde{\setW}_i\cap\setE)=0$ for all countably $(\mathscr{H}^{s},s)$-rectifiable sets $\setE\subseteq\reals^m$. 
This allows us to decompose $\setU$ according to $\setU=\setW\cup\widetilde{\setW}$ with 
$\setW=\bigcup_{i\in\naturals}\setW_i$ and $\widetilde{\setW}=\bigcup_{i\in\naturals}\widetilde{\setW}_i$. 
Since all the $\setW_i$ are countably $(\mathscr{H}^{s},s)$-rectifiable, so is $\setW$; and since all the $\widetilde{\setW}_i$ are purely $\mathscr{H}^{s}$-unrectifiable, so is $\widetilde{\setW}$. 

As $\mu=\mu|_\setU$ by assumption, and $\setU=\setW\cup\widetilde{\setW}$, it remains to show 
that $\mu(\widetilde{\setW})=0$ to conclude that $\mu=\mu|_\setW$ for the countably $(\mathscr{H}^{s},s)$-rectifiable set $\setW$. 
Towards a contradiction, suppose that  $\mu(\widetilde{\setW})>0$. Analyticity of $\mu$ then implies 
that there exist 
a set $\setA\in\colB(\reals^{s})$ of positive  Lebesgue measure
and 
a real analytic mapping $h\colon\reals^s\to \reals^m$ of $s$-dimensional Jacobian $Jh\not\equiv 0$  such that $h(\setA)\subseteq\widetilde{\setW}$. By  countable subadditivity of $\lambda^s$, we can assume, w.l.o.g., that $\setA$ is bounded;  
and by Property \ref{immreal2} in Lemma \ref{lem:immreal}, we can assume, w.l.o.g., that $h|_\setA$ is an embedding. 
It follows that  
\begin{align}
0&< \int_\setA Jh(\vecz)\,\mathrm d\lebmeasure^{s}(\vecz)\label{eq:step1richa}\\
&=\mathscr{H}^{s}\big(h(\setA)\big),\label{eq:step3richa}
%&\leq \mathscr{H}^{s}(\setC),\label{eq:step4richa}
\end{align}
where \eqref{eq:step1richa} is by Lemma \ref{lem.fzm}, $\lebmeasure^s(\setA)>0$,  and $Jh(\vecz)> 0$ for all $\vecz\in\setA$, and   
 in \eqref{eq:step3richa} we applied the area formula Corollary \ref{thm:area} upon noting that $h|_\setA$ is one-to-one as an embedding and locally Lipschitz by real analyticity of  $h$.  
Moreover, as $\mathscr{H}^{s}$ is Borel regular (see Property \ref{Def.br} in Definition \ref{Def.meas}), there must exist a set 
$\setC\in\colB(\reals^m)$ with $h(\setA)\subseteq \setC$ and $\mathscr{H}^{s}(\setC\,\mysetminus\, h(\setA))=0$. It follows that  $\setC$ is countably $(\mathscr{H}^{s},s)$-rectifiable 
as  $\mathscr{H}^{s}(\setC\mysetminus h(\overline{\setA}))=0$ and $h$ is Lipschitz on the compact set $\overline{\setA}$ by Lemma \ref{lem:lipcompact}. 
But this implies 
\begin{align}
\mathscr{H}^{s}(\widetilde{\setW}\cap\setC)
&\geq \mathscr{H}^{s}(h(\setA)\cap\setC)\label{eq:CC1}\\
&=\mathscr{H}^{s}(h(\setA))\label{eq:CC2}\\
&>0, 
\end{align}  
which is not possible as $\widetilde{\setW}$  is purely $\mathscr{H}^{s}$-unrectifiable, thereby concluding the proof.\\ 
,,$\Leftarrow$'': Suppose that there exists a countably  $(\mathscr{H}^{s},s)$-rectifiable set $\setW\subseteq\reals^m$ such that $\mu=\mu|_\setW$. Since $\mathscr{H}^{s}$ is Borel regular (see Definition \ref{Def.meas}), we may assume, w.l.o.g., that  $\setW\in\colB(\reals^m)$.
As this $\setW$ is  countably $(\mathscr{H}^{s},s)$-rectifiable, \cite[Lemma 15.5]{ma99} implies that $\mathscr{H}^{s}|_{\setW}$ is $\sigma$-finite.  
\qed

\section{Proof of Lemma \ref{lemancomp}}\label{proof:lemancomp}
Suppose that $\rvecx\in\reals^m$ is $s$-analytic and  $\rvecu=f(\rvecx)$, where $f\colon\reals^m\to\reals^k$ is a real analytic immersion,  and  consider $\setC\in\colB(\reals^{k})$ with  $\mu_\rvecu(\setC)>0$. 
We have to show that there exist 
a set $\setA\in\colB(\reals^s)$ of positive  Lebesgue measure 
and a real analytic mapping $g\colon\reals^{s}\to \reals^{k}$ of $s$-dimensional Jacobian $Jg\not\equiv 0$    such that $g(\setA)\subseteq\setC$.  
Set $\setD=f^{-1}(\setC)\in\colB(\reals^{m})$. 
Since $\mu_\rvecx(\setD)=\mu_\rvecu(\setC)>0$ and $\rvecx$ is $s$-analytic, there exist 
a set $\setA\in\colB(\reals^s)$ of positive  Lebesgue measure and 
a real analytic mapping $h\colon\reals^{s}\to \reals^{m}$ of $s$-dimensional Jacobian $Jh\not\equiv 0$   such that $h(\setA)\subseteq\setD$. 
We set $g=f\circ h$. 
Now,  $g(\setA)=f(h(\setA))\subseteq f(\setD)\subseteq\setC$. 
Furthermore, $g$  as  the composition of real analytic mappings is real analytic by Corollary \ref{lem:analyticcomp}.  
It remains to show that $Jg\not\equiv 0$. 
To this end, we first note that the chain rule implies 
$Dg(\vecz)=(Df)(h(\vecz))Dh(\vecz)$. 
Since $Jh\not\equiv 0$, there exists a $\vecz_0\in\reals^s$ such that $Jh(\vecz_0)\neq 0$. Thus
\begin{align}
\rank(Dh(\vecz_0))=s. \label{eq:rank2} 
\end{align} 
Now, as $f$ is an immersion, it follows that $k\geq m$ and $J\!f>0$. Thus,  
\begin{align}
\rank((Df)(h(\vecz_0)))=m.\label{eq:rank3}
\end{align}
Applying Lemma \ref{lem:sylvester} to 
$(Df)(h(\vecz_0))\in\reals^{k\times m}$ and $Dh(\vecz_0)\in\reals^{m\times s}$ and using 
\eqref{eq:rank2} and \eqref{eq:rank3} establishes  that $\rank(Dg(\vecz_0))\geq s$, which  in turn implies  $Jg(\vecz_0)\neq 0$ as $Dg(\vecz_0)\in\reals^{k\times s}$.  
\qed

\section{Proof of Lemma \ref{auxlemma1}}\label{auxlemma1proof}
Suppose that $\rveca\in\reals^k$ and $\rvecb\in\reals^l$ with $\mu_\rveca\times\mu_\rvecb\ll\lambda^{k+l}$, set 
$\rvecx=\rveca\otimes\rvecb$, and consider  $\setC\in\colB(\reals^{kl})$ with $\mu_\rvecx(\setC)>0$. 
We have to show that there exist 
a set $\setA\in\colB(\reals^{k+l-1})$ of positive  Lebesgue measure  and 
a real analytic mapping $h\colon\reals^{k+l-1}\to \reals^{kl}$ of $(k+l-1)$-dimensional Jacobian $Jh\not\equiv 0$   such that $h(\setA)\subseteq\setC$. 
Let   
\begin{align}
\setE=\{\tp{(\tp{\veca}\ \tp{\vecb})}:\veca\in\reals^{k}, \vecb\in\reals^{l}, \veca\otimes\vecb\in\setC\}.\label{eq:defE}
\end{align}
Since $\setC\in\colB(\reals^{kl})$ and $\otimes$ is Borel measurable,   $\setE$ as the inverse image of $\setC$ under $\otimes$ is Borel measurable. Furthermore, as $\rvecx=\rveca\otimes\rvecb$,  it follows that   
\begin{align}
(\mu_\rveca\times\mu_\rvecb)(\setE)
&=\mu_\rvecx(\setC)\\
&>0, 
\end{align}
which implies  $\lebmeasure^{k+l}(\setE)>0$ as  $\mu_\rveca\times\mu_\rvecb\ll\lebmeasure^{k+l}$ by assumption. 
Using Corollary \ref{cor.fubini}, we can write  
\begin{align}
\lebmeasure^{k+l}(\setE)=\int_\reals \lebmeasure^{k+l-1}(\setE_z)\,\mathrm d \lebmeasure^{1}(z),\label{eq:step2exaanalytic}
\end{align}
where, for each $z\in\reals$, 
\begin{align}
\setE_z=\big\{\tp{\big(\tp{\veca}\ \tp{\vecv}\big)}: \veca\in\reals^{k}, \vecv\in\reals^{l-1}, \tp{(\tp{\veca}\ \tp{\vecv} z)}\in\setE\big\}.  
\end{align}
As $\lebmeasure^{k+l}(\setE)>0$ we can conclude that there must exist a $z_0\in\reals\mysetminus\{0\}$  such that $\lebmeasure^{k+l-1}(\setE_{z_0})>0$. We set $\setA=\setE_{z_0}$. 
Next, we define the mapping  
\begin{align}
h\colon\reals^{k+l-1}&\to \reals^{kl}\label{eq:mapha}\\
\tp{(c_1 \dots c_{k+l-1})} & \mapsto \tp{(c_1 \dots c_{k})}\otimes \tp{(c_{k+1} \dots c_{k+l-1}\ z_0)},\label{eq:maph}
\end{align}
which  is real analytic thanks to Corollary \ref{cor:poly},  
and we write 
\begin{align}
\setA
&=\big\{\tp{\big(\tp{\veca}\ \tp{\vecv}\big)}: \veca\in\reals^{k}, \vecv\in\reals^{l-1}, \tp{(\tp{\veca}\ \tp{\vecv} z_0)}\in\setE\big\}\\
&=\big\{\tp{\big(\tp{\veca}\ \tp{\vecv}\big)}: \veca\in\reals^{k}, \vecv\in\reals^{l-1}, \veca\otimes \tp{(\tp{\vecv}\ z_0)}\in\setC\big\}\label{eq:useE}\\
&=h^{-1}(\setC),\label{eq:useh}
\end{align}
where 
\eqref{eq:useE} follows from  \eqref{eq:defE}, and 
\eqref{eq:useh} is by \eqref{eq:mapha}--\eqref{eq:maph}.    
By construction,  $h(\setA)\subseteq\setC$.
Furthermore,  $\setA$ as  the inverse image of  $\setC\in\colB(\reals^{kl})$ under a real analytic and, therefore, Borel measurable mapping is in $\colB(\reals^{k+l-1})$.   
 It remains to show that there exist an $\veca_0\in\reals^k$ and a $\vecv_0\in\reals^{l-1}$ such that 
\begin{align}
&Jh\Big(\tp{\big(\tp{\veca_0}\ \tp{\vecv_0}\big)}\Big)\\
&=\sqrt{\det\Big(\tp{\Big(Dh\Big(\tp{\big(\tp{\veca_0}\ \tp{\vecv_0}\big)}\Big)\Big)}Dh\Big(\tp{\big(\tp{\veca_0}\ \tp{\vecv_0}\big)}\Big)\Big)}\\
&> 0. 
\end{align}
This will be accomplished by showing that there exist an $\veca_0\in\reals^k$ and a $\vecv_0\in\reals^{l-1}$ such that 
\begin{align}\label{eq:rankdiff}
\rank\!\Big(Dh\Big(\tp{\big(\tp{\veca_0}\ \tp{\vecv_0}\big)}\Big)\Big)=k+l-1. 
\end{align}
Now,   
\begin{align}
Dh\Big(\tp{\big(\tp{\veca}\ \tp{\vecv}\big)}\Big)=
\begin{pmatrix}
\vecv   &\veczero&\dots   &\veczero&\veczero&a_1\matI_{l-1}\\
z_0     &0       &\dots   &0       &0       &\veczero\\
\veczero&\vecv   &\dots   &\veczero&\veczero&a_2\matI_{l-1}\\
0       &z_0     &\dots   &0       &0       &\veczero\\
\vdots  &\vdots  &\vdots  &\vdots  &\vdots  &\vdots\\
\veczero&\veczero&\dots   &\vecv   &\veczero&a_{k-1}\matI_{l-1}\\
0       &0       &\dots   &z_0     &0       &\veczero\\
\veczero&\veczero&\dots   &\veczero&\vecv   &a_{k}\matI_{l-1}\\
0       &0       &\dots   &0       &z_0     &\veczero
\end{pmatrix}
\end{align}
for general $\veca\in\reals^k$ and $\vecv\in\reals^{l-1}$, and  
\begin{align}
Dh\Big(\tp{\big(\tp{\veca_0}\ \tp{\vecv_0}\big)}\Big)=
\begin{pmatrix}
\veczero&\veczero&\dots   &\veczero&\veczero&\matI_{l-1}\\
z_0     &0       &\dots   &0       &0       &\veczero\\
\veczero&\veczero&\dots   &\veczero&\veczero&\matI_{l-1}\\
0       &z_0     &\dots   &0       &0       &\veczero\\
\vdots  &\vdots  &\vdots  &\vdots  &\vdots  &\vdots\\
\veczero&\veczero&\dots   &\veczero&\veczero&\matI_{l-1}\\
0       &0       &\dots   &z_0     &0       &\veczero\\
\veczero&\veczero&\dots   &\veczero&\veczero&\matI_{l-1}\\
0       &0       &\dots   &0       &z_0     &\veczero
\end{pmatrix}
\end{align}
for the specific choices $\veca_0=\tp{(1\dots 1)} \in\reals^k$ and $\vecv_0=\tp{(0\dots 0)} \in\reals^{l-1}$. 
Since the $(kl)\times(k+l-1)$ matrix  $Dh\big(\tp{(\tp{\veca_0}\ \tp{\vecv_0})}\big)$ has 
the regular $(k+l-1)\times (k+l-1)$ submatrix (recall that $\vecz_0\neq 0$)
\begin{align}
\begin{pmatrix}
\matzero&\matI_{l-1}\\
z_0 \matI_{k}&\matzero
\end{pmatrix},
\end{align}
\eqref{eq:rankdiff} indeed holds, which  concludes the proof. 
\qed 

\section{Tools from (Geometric) Measure Theory}\label{app:geo}

In this appendix, we state some basic definitions and results from measure theory and geometric measure theory  used throughout the paper.  
For an excellent in-depth treatment of geometric measure theory, the interested reader is referred to \cite{krpa08,evga15,fe78,fed69}. 

%We  start with measure theory. 

\subsection{Preliminaries from Measure Theory}\label{app:geoA}
\begin{dfn}\cite[Definition 1.2.1]{krpa08}\label{dfn:measure}
A measure (sometimes called outer measure, see \cite[Remark 1.2.6]{krpa08}) on a nonempty set $\setX$ is a nonnegative function $\mu$ defined on all subsets of $\setX$ with 
the following properties:
\begin{enumerate}[label=\roman*)]
\item $\mu(\emptyset)=0$.
\item Monotonicity: 
\begin{align}\label{eq:mon}
\mu(\setA)&\leq \mu(\setB)\quad \text{for all}\ \setA\subseteq \setB\subseteq\setX.
\end{align}
\item Countable subadditivity:
\begin{align}\label{eq:cs}
\mu\mleft(\bigcup_{i\in\naturals}\setA_i\mright)&\leq \sum_{i\in\naturals}\mu(\setA_i)
\end{align}
for all sequences $\{\setA_i\}_{i\in\naturals}$ of sets $\setA_i\subseteq\setX$. 
\end{enumerate}
A set  $\setA\subseteq\setX$ is $\mu$-measurable if it satisfies 
\begin{align}\label{eq:measset}
\mu(\setE)=\mu(\setE\cap\setA)+\mu(\setE\mysetminus\setA)\quad\text{for all}\  \setE\subseteq\setX. 
\end{align}
\end{dfn}
For a probability measure $\mu_\rvecx$,   countable subadditivity  is equivalent to the union bound 
\begin{align}\label{eq:countsub}
\opP\mleft[\rvecx\in\bigcup_{i\in\naturals}\setA_i\mright]\leq \sum_{i\in\naturals}\opP\mleft[\rvecx\in\setA_i\mright].
\end{align}
The collection of $\mu$-measurable sets forms a $\sigma$-algebra  \cite[Theorem 1.2.4]{krpa08}, which we denote by $\colS_\mu(\setX)$. 
If $\setX$ is endowed with a topology,\footnote{A topology for $\setX$ is a collection of subsets of $\setX$ that contains $\emptyset$ and $\setX$ and is closed under finite intersections and arbitrary unions. The members of a topology are called open sets.} the smallest $\sigma$-algebra containing the open sets is the Borel $\sigma$-algebra $\colB(\setX)$. A measurable space $(\setX,\colS(\setX))$ is a  set $\setX$ equipped with a 
$\sigma$-algebra $\colS(\setX)$. A measure space $(\setX,\colS(\setX),\mu)$ is a set $\setX$ with a measure $\mu$ and a $\sigma$-algebra $\colS(\setX)\subseteq \colS_\mu(\setX)$.   

\begin{lem}\label{lem.capmeasure}\cite[Theorem 1.2.5]{krpa08}
If $(\setX,\colS(\setX),\mu)$ is a measure space and  $\{\setA_n\}_{n\in\naturals}$ is a sequence of  sets $\setA_n\in\colS(\setX)$, then  the following properties hold. 
\begin{enumerate}[label=\roman*)]
\item\label{lem.capmeasurei} If  the sets $\setA_n$, $n\in\naturals$, are pairwise disjoint, then $\mu$ is countably additive on their union. That is, 
\begin{align}
\mu\Big(\bigcup_{n\in\naturals} \setA_n\Big)=\sum_{n\in\naturals} \mu(\setA_n). 
\end{align}\label{eq:additivity}
\vspace*{-3truemm}
\item\label{lem.capmeasureii} If $\setA_{n}\subseteq\setA_{n+1}$ for all $n\in\naturals$, then
\begin{align} 
\mu\Big(\bigcup_{n\in\naturals} \setA_n\Big)=\lim_{n\to\infty} \mu(\setA_n). 
\end{align}
\item\label{lem.capmeasureiii} If  $\setA_{n+1}\subseteq\setA_{n}$ for all $n\in\naturals$ and $\mu(\setA_1)<\infty$, then
\begin{align} 
\mu\Big(\bigcap_{n\in\naturals} \setA_n\Big)=\lim_{n\to\infty} \mu(\setA_n). 
\end{align}
\end{enumerate}
\end{lem}

\begin{dfn}\cite[Definition 1.2.10]{krpa08}\label{Def.meas}
A measure $\mu$ on a nonempty set $\setX$ endowed with a topology is 
\begin{enumerate}[label=\roman*)]
\item Borel if all open sets are $\mu$-measurable; 
\item Borel regular if it is Borel and, for each $\setA\subseteq\setX$, there exists a set $\setB\in\colB(\setX)$ such that $\setA\subseteq\setB$ and 
$\mu(\setA)=\mu(\setB)$. \label{Def.br}
\end{enumerate}
\end{dfn}

\begin{lem}\label{lem:equivrec}
Let  $s\in\naturals$ and consider a measure $\mu$  on $\reals^m$.  For every nonempty set $\setU\subseteq\reals^{m}$, the following statements are equivalent:
\begin{enumerate}[label=\roman*)]
\item\label{defrec1} There exist compact sets $\setA_i\subseteq\reals^s$ and Lipschitz  mappings $\varphi_i\colon\setA_i\to \reals^{m}$, $i\in\naturals$,  such that 
\begin{align}
\mu\mleft(\setU\setminus\bigcup_{i\in\naturals}\varphi_i(\setA_i)\mright)=0.
\end{align}
\item\label{defrec2} There exist bounded sets $\setA_i\subseteq\reals^s$ and Lipschitz  mappings $\varphi_i\colon\setA_i\to \reals^{m}$, $i\in\naturals$,  such that 
\begin{align}
\mu\mleft(\setU\setminus\bigcup_{i\in\naturals}\varphi_i(\setA_i)\mright)=0.
\end{align}
\item\label{defrec3} There exist  Lipschitz  mappings $\varphi_i\colon\reals^s\to \reals^{m}$, $i\in\naturals$,  such that 
\begin{align}
\mu\mleft(\setU\setminus\bigcup_{i\in\naturals}\varphi_i(\reals^s)\mright)=0.
\end{align}
\end{enumerate}
\end{lem}
\begin{proof}
We show that \ref{defrec1}$\,\Rightarrow\,$\ref{defrec2}$\,\Rightarrow\,$\ref{defrec3}$\,\Rightarrow\,$\ref{defrec1}.
 \ref{defrec1}$\,\Rightarrow\,$\ref{defrec2} is trivial since every compact set  $\setA_i\subseteq\reals^s$ is bounded. 
 \ref{defrec2}$\,\Rightarrow\,$\ref{defrec3} follows from the fact that, by  \cite[Theorem 7.2]{ma99}, for every bounded  set $\setA_i\subseteq\reals^s$ and corresponding Lipschitz mapping $\varphi_i\colon\setA_i\to \reals^{m}$, there exists a Lipschitz mapping $\tilde \varphi_i\colon\reals^s\to \reals^{m}$ such that $\tilde \varphi_i|_{\setA_i}= \varphi_i$. Finally, \ref{defrec3}$\,\Rightarrow\,$\ref{defrec1} as 
\begin{align}
\varphi_i(\reals^s)=\bigcup_{j\in\naturals} \varphi_i(\overline{\setB_s(\veczero,j)})\quad\text{for all $i\in\naturals$}
\end{align}
with $\overline{\setB_s(\veczero,j)}$ compact for all $j\in\naturals$.  
\end{proof}

\begin{dfn}(Hausdorff measures)\cite[Definition 2.46]{amfupa00} \label{dfnH}
Let  $d\in [0,\infty)$ and $\setU\subseteq \reals^{m}$. 
The $d$-dimensional Hausdorff measure of $\setU$, denoted by $\mathscr{H}^{d}(\setU)$,  is defined according to 
\begin{align}\label{eq:Hddim}
\mathscr{H}^{d}(\setU) =\lim_{\delta\to 0}\mathscr{H}_\delta^{d}(\setU),
\end{align} 
where 
\begin{align}\label{eq:Hdd}
&\mathscr{H}_\delta^{d}(\setU)\\
&=\frac{\pi^{d/2}}{2^d\Gamma(d/2+1)} \inf\mleft\{\sum_{i\in\naturals}\diam(\setU_i)^d : \diam(\setU_i)<\delta,\mright.\nonumber\\
&\mleft.\phantom{=\frac{\pi^{d/2}}{2^d\Gamma(d/2+1)} \inf \{\ \ \ \ } \setU\subseteq\bigcup_{i\in\naturals}\setU_i\mright\}  
\end{align}
for  all $\delta>0$, and the diameter $\operatorname{diam}(\cdot)$ of an arbitrary set $\setU\subseteq\reals^m$  is defined according to   
\begin{align}\label{eq:diam}
\diam(\setU)=
\begin{cases}
\sup\{\|\vecu-\vecv\|_2 : \vecu,\vecv\in\setU\}&\text{if}\ \setU\neq\emptyset\\
0&\text{else}
\end{cases}
\end{align}
%and %$V(d)$ in the right-hand side of \eqref{eq:Hdd} denoting the volume of the $d$-dimensional unit sphere, i.e., 
with $\Gamma(\cdot)$ denoting the Gamma function. 
\end{dfn} 

\begin{lem}\label{lem:prophausdorff} (Main properties of Hausdorff measures)
\begin{enumerate}[label=\roman*)] 
\item If $a>b\geq0$, then $\mathscr{H}^{a}(\setE)>0$ implies  $\mathscr{H}^{b}(\setE)=\infty$ for all $\setE\subseteq \reals^{m}$ (see Fig. \ref{fig:hd})\label{Hjump}.
\item If $f\colon\reals^{m}\to\reals^{n}$ is  Lipschitz with Lipschitz constant $L$, then 
\begin{align}
\mathscr{H}^{d}(f(\setE))\leq L^d\mathscr{H}^{d}(\setE)\quad \text{for all}\  \setE\subseteq\reals^{m}.
\end{align}\label{HL}
\vspace*{-3truemm}
\item  $\mathscr{H}^{m}(\setE)=\lebmeasure^{m}(\setE)$ for all   $\setE\in\colB(\reals^{m})$.\label{HBorel}
\item $\mathscr{H}^{0}$ is the counting measure. \label{HCounting}
\end{enumerate}
\end{lem}
\begin{proof}
See \cite[Proposition 2.49]{amfupa00} and \cite[Theorem 2.53]{amfupa00} for Properties  \ref{Hjump}--\ref{HBorel}. Property  \ref{HCounting} follows immediately from the definition of the $0$-dimensional Hausdorff measure. 
\end{proof} 

\begin{dfn}(Hausdorff dimension)\cite[Definition 2.51]{amfupa00}\label{def:HDorf}
The Hausdorff dimension of $\setU\subseteq\reals^{m}$, denoted by $\dim_{\mathrm{H}}(\setU)$,  is defined according to 
\begin{align}
\dim_{\mathrm{H}}(\setU)
&:=\sup\{d\geq 0 : \mathscr{H}^{d}(\setU)=\infty\}\\
&=\inf\{d\geq 0 : \mathscr{H}^{d}(\setU)=0\}, 
\end{align}
i.e., $\dim_{\mathrm{H}}(\setU)$  is the value of $d$ for which the sharp transition  from $\infty$ to $0$ in Figure \ref{fig:hd} occurs.
Depending on the set $\setU$, for $d=\dim_{\mathrm{H}}(\setU)$, $\mathscr{H}^{d}(\setU)$ can take on any value in $[0,\infty]$. 
\end{dfn}

\begin{figure}
\begin{center}
\begin{tikzpicture}[scale=1.5]
    \draw [thick, <->] (0,2) node (yaxis) [above] {$\mathscr{H}^{d}(\setU)$}
        |- (4,0) node (xaxis) [right] {$d$};
    \draw [ultra thick] (0,1.6) coordinate (a_1) -- (1.8,1.6) coordinate (a_2);
    \draw [ultra thick] (1.8,0) coordinate (b_1) -- (3,0) coordinate (b_2);
    \node[left] at (a_1) {$\infty$};
    \node[left] at (0,0) {$0$};
    \node[below] at (b_1) {$\dim_{\mathrm{H}}(\setU)$};
    \node[below] at (b_2) {$m$};
\end{tikzpicture}
\caption{(\!\!\cite[Figure 3.3]{fa14}) 
Graph of  $\mathscr{H}^{d}(\setU)$ as a function of $d\in[0,m]$ for a set $\setU\subseteq\reals^{m}$. 
\label{fig:hd}} 
\end{center}
\vspace*{-3truemm}
\end{figure}

As a consequence of Carath{\'e}odory's criterion \cite[Theorem 1.2.13]{krpa08}, Lebesgue and Hausdorff measures are Borel regular. 

\begin{dfn}(Measurable mapping)\cite[Chapter 2]{ba95}\label{dfn:measurable}
\begin{enumerate}[label=\roman*)]
\item
Let $(\setX,\colS(\setX))$ and $(\setY,\colS(\setY))$ be  measurable spaces. 
A mapping $f\colon\setD\to\setY$, with $\setD\in\colS(\setX)$, is $(\colS(\setX),\colS(\setY))$-measurable if  $f^{-1}(\setA)\in\colS(\setX)$ for all $\setA\in\colS(\setY)$. 
\item \label{Borelfa}
Let $(\setX,\colS(\setX))$ be a measurable space and $\setY$ endowed with a topology.
A mapping $f\colon\setD\to\setY$, with $\setD\in\colS(\setX)$, is $\colS(\setX)$-measurable if
$f^{-1}(\setA)\in\colS(\setX)$  for all  $\setA\in\colB(\setY)$. 
\item\label{Borelf}
Let $\setX$ and $\setY$ both be endowed with a topology. 
A mapping $f\colon\setD\to\setY$, with $\setD\in\colB(\setX)$, is  Borel measurable if $f^{-1}(\setA)\in\colB(\setX)$ for all  $\setA\in\colB(\setY)$. 
\end{enumerate}
\end{dfn}

\begin{lem}\label{lem.fzm}\cite[Corollary 4.10]{ba95}
Let $(\setX,\colS(\setX),\mu)$ be a measure space and consider a nonnegative measurable function $f\colon\setX\to\reals$. Then, $f(x)=0$ $\mu$-almost everywhere, i.e., $\mu\{x\in\setX:f(x)\neq 0\}=0$,  
if and only if 
\begin{align}
\int_\setX f(x)\,\mathrm d \mu(x)=0.
\end{align}
\end{lem}

\begin{dfn}\cite[Definition 1.3.25]{krpa08}\label{dfn:sigmafinite}
The measure space $(\setX,\colS(\setX),\mu)$  is $\sigma$-finite if $\setX=\bigcup_{i\in\naturals}\setA_i$, with $\setA_i\in\colS(\setX)$ and $\mu(\setA_i)<\infty$ for all $i\in\naturals$.  The measure space   is finite if $\mu(\setX)<\infty$. A Borel measure $\mu$ on a topological space $\setX$ is \emph{$\sigma$-finite} (respectively \emph{finite}) if  $(\setX,\colB(\setX),\mu)$ is a $\sigma$-finite (respectively finite) measure space. 
\end{dfn}
For example, all probability measures are  finite and the  Lebesgue measure is  $\sigma$-finite. 
However,  the $s$-dimensional Hausdorff measure with $s<m$ 
is not $\sigma$-finite.

\begin{thm}(Fubini's theorem)\label{thm.fubini} \cite[Theorem 10.9]{ba95}
Let $(\setX,\colS(\setX),\mu)$ and $(\setY,\colS(\setY),\nu)$ be $\sigma$-finite measure spaces  and suppose that $f\colon \setX\times\setY\to\reals$ is nonnegative and $(\colS(\setX)\otimes\colS(\setY))$-measurable. Then,   
\begin{align}
\varphi(x)=\int_{\setY} f(x,y) \,\mathrm d\nu(y)\quad \text{is $\colS(\setX)$-measurable},\\
\psi(y)=\int_{\setX}f(x,y) \,\mathrm d\mu(x)\quad \text{is $\colS(\setY)$-measurable},
\end{align}
and 
\begin{align}\label{eq:fubfub}
\int_\setX \varphi(x) \,\mathrm d\mu(x)
&=\int_{\setX\times\setY} f(x,y)\,\mathrm d(\mu\times\nu)(x\times y)\\
&=\int_\setY \psi(y) \,\mathrm d\nu(y).
\end{align}  
\end{thm}

\begin{cor}\label{cor.fubini}
Let $(\setX,\colS(\setX),\mu)$ and $(\setY,\colS(\setY),\nu)$ be $\sigma$-finite measure spaces  and suppose that $\setA\in\colS(\setX)\otimes\colS(\setY)$. Then,   
\begin{align}
&\int_\setX \nu(\{y: (x, y)\in \setA\})\,\mathrm d\mu(x)\\
&=\int_{\setX\times\setY} \ind{\setA}(x,y)\,\mathrm d(\mu\times\nu)(x\times y)\\
&=\int_\setY \mu(\{x: (x,y)\in \setA\})\,\mathrm d\nu(y).
\end{align} 
\end{cor}
\begin{proof}
Follows from Theorem \ref{thm.fubini} with $f(x,y)=\ind{\setA}(x,y)$, noting that 
\begin{align}
\int_\setY \ind{\setA}(x,y) \,\mathrm d\nu(y)&=\nu(\{y: (x,y)\in \setA\}),\\
\int_\setX \ind{\setA}(x,y) \,\mathrm d\mu(x)&=\mu(\{x: (x,y)\in \setA\}).
\end{align}
\end{proof}

\begin{lem}\cite[Exercise 1.7.19]{ta11}\label{lem:prodborel}
If $\setX$ and $\setY$ are Euclidean spaces, then $\colB(\setX\times\setY)=\colB(\setX)\otimes\colB(\setY)$. 
\end{lem}

\begin{lem}\cite[Corollary 2.10]{ba95}\label{lem.measlimit}
Let $(\setX,\colS(\setX))$  be a measurable space and suppose that $\{f_i\}_{i\in\naturals}$ is a sequence of $\colS(\setX)$-measurable functions $f_i\colon\setX\to\reals$ converging pointwise to  $f\colon\setX\to\reals$. Then, $f$ is $\colS(\setX)$-measurable.  
\end{lem}

\begin{lem}\label{lem.meascart}
Let  $\setX$ and $\setY$  be  Euclidean spaces and suppose that 
$f_i\colon\setX\to\setY$, $i=1,\dots,n$, are Borel measurable mappings.  Then,  
\begin{align}
f\colon\setX&\to \tilde\setY=\underbrace{\setY\times\dots\times \setY}_{n\, \text{times}}\\
x&\mapsto (f_1(x),\dots, f_n(x))
\end{align}
is Borel measurable.
\end{lem}
\begin{proof}
Lemma \ref{lem:prodborel} implies $\colB(\tilde\setY)=\colB(\setY)\otimes\dots\otimes\colB(\setY)$. Thus,  $\colB(\tilde\setY)$ is generated by  sets $\setU_1\times\dots\times \setU_n$, where $\setU_i\in\colB(\setY)$ for $i=1,\dots,n$. 
It is therefore sufficient %(see Part \ref{Borelf} in Definition \ref{dfn:measurable})
 to show that $f^{-1}(\setU_1\times\dots\times \setU_n)\in \colB(\setX)$ for all $\setU_i\in\colB(\setX)$,  $i=1,\dots,n$. But 
$f^{-1}(\setU_1\times\dots\times \setU_n)=f_1^{-1}(\setU_1)\cap\dots\cap f_n^{-1}(\setU_n)\in \colB(\setX)$ for all $\setU_i\in\colB(\setY)$, $i=1,\dots,n$, as all the $f_i$ are Borel measurable and every $\sigma$-algebra is closed under finite or countably infinite intersections.  
\end{proof}

\begin{cor}\label{lem.meascart2}
Let $\setX$ be a Euclidean space and 
suppose that $\{f_i\}_{i\in\naturals}$ is a sequence of Borel measurable mappings $f_i\colon\setX\to\reals^n$ converging pointwise to  $f\colon\setX\to\reals^n$.  Then, $f$ is Borel measurable. 
\end{cor}
\begin{proof}
We can write $f(x)=\tp{(f^{(1)}(x),\dots,f^{(n)}(x))}$, where $f^{(j)}\colon \setX\to\reals$ is Borel measurable for $j=1,\dots,n$,  
and $f_i(x)=\tp{\big(f^{(1)}_i(x),\dots,f^{(n)}_i(x)\big)}$, where $f^{(j)}_i\colon \setX\to\reals$ is Borel measurable for $j=1,\dots,n$ and all $i\in\naturals$. 
Then,  for  each  $j=1,\dots,n$, the sequence $(f_i^{(j)})_{i\in\naturals}$ of  Borel measurable  functions  converges pointwise to  $f^{(j)}$,  which is Borel measurable thanks to Lemma \ref{lem.measlimit}. Finally, Lemma \ref{lem.meascart} implies that $f$ is Borel measurable as its individual components 
$f^{(1)},\dots,f^{(n)}$  are  Borel measurable. 
\end{proof}

\begin{lem}\label{lem:fa}
Let $\setX$ and $\setY$ be topological spaces, $\setC\in\colB(\setX)$,   $f\colon\setC\to\setY$ Borel measurable, and $y_0\in\setY\mysetminus f(\setC)$. Consider $h\colon \setX\to \setY$ with 
$h|_{\setC}=f$ and $h|_{\setX\setminus\setC}=y_0$. 
Then, $h$ is Borel measurable. 
\end{lem}
\begin{proof}
We have to show that $h^{-1}(\setA)\in\colB(\setX)$ for all $\setA\in\colB(\setY)$. Consider an arbitrary but fixed $\setA\in \colB(\setY)$ and suppose first that   $y_0\in \setA$. 
Now, $\setA\in\colB(\setY)$ and $\{y_0\}\in\colB(\setY)$ imply $\setA\mysetminus \{y_0\}\in \colB(\setY)$. 
Therefore, 
\begin{align}
h^{-1}(\setA)
&=h^{-1}(\setA\mysetminus \{y_0\})\cup h^{-1}(\{y_0\})\\
&=f^{-1}(\setA\mysetminus \{y_0\})\cup (\setX\mysetminus\setC)\in\colB(\setX)
\end{align}
because $f$ is  Borel measurable, $\setA\mysetminus \{y_0\}\in\colB(\setY)$, and $\setX\mysetminus\setC\in\colB(\setX)$. If $y_0\notin \setA$, then  
\begin{align}
h^{-1}(\setA)
&=f^{-1}(\setA)\in\colB(\setX)
\end{align}
as  $f$ is Borel measurable and  $\setA\in\colB(\setY)$. 
\end{proof}

\begin{lem}\label{lem:loef}
Let $\mu$ be a measure on $\reals^k$ and consider $\setA\subseteq\reals^k$ with $\mu(\setA)>0$. 
Then,  there exists a $\vecz_0\in\setA$ such that
\begin{align} 
\mu\big(\setB_{k}(\vecz_0,r)\cap \setA\big)>0\quad \text{for all}\ r>0. \label{Armupos}
\end{align}
\end{lem}
\begin{proof}
Suppose, to the contrary, that such a $\vecz_0$ does not exist. 
Then, for every $\vecz\in\setA$, there must exist  an $r_\vecz>0$ such that $\mu\big(\setB_{k}(\vecz,r_\vecz)\cap \setA\big)=0$.  
With these $r_\vecz$,  we write   
\begin{align}\label{eq:sentloef}
\setA=\bigcup_{\vecz\in\setA} (\setB_{k}(\vecz,r_\vecz)\cap \setA). 
\end{align}
It now follows from the Lindel\"of property of $\reals^{k}$ \cite[Definition 5.6.19]{so14}\footnote{Recall that $\reals^k$ with the Euclidean distance metric is a separable metric space, i.e., $\reals^k$ includes a countable dense subset, and it is therefore a Lindel\"of space by \cite[Proposition 5.6.22]{so14}.} that there must exist a countable subset $\{\vecz_i:i\in\naturals\}\subseteq\setA$ such that 
\begin{align}
\setA=\bigcup_{i\in\naturals} (\setB_{k}(\vecz_i,r_{\vecz_i})\cap \setA).  
\end{align}
Since $\mu\big(\setB_{k}(\vecz_i,r_{\vecz_i})\cap \setA\big)=0$ for all $i\in\naturals$, the  countable subadditivity of $\mu$  implies 
\begin{align}
\mu(\setA)
&\leq \sum_{i\in\naturals}\mu\big(\setB_{k}(\vecz_i,r_{\vecz_i})\cap \setA\big)\\
&=0,
\end{align}
which contradicts the assumption $\mu(\setA)>0$. 
\end{proof}

\begin{lem}\label{lem:inner}
Let $(\setX,\colB(\setX),\mu)$ be a $\sigma$-finite measure space   
and $\setB\in\colB(\setX)$. Then, 
\begin{align}
\setB=\setN\cup\setA,\label{eq:compunionN}
\end{align}
where $\mu(\setN)=0$ and $\setA=\bigcup_{i\in\naturals}\setA_i$ with $\setA_i\subseteq\reals^m$ compact for all  $i\in\naturals$. 
\end{lem}
\begin{proof}
By \cite[Proposition 1.43]{amfupa00}, we can find, for each $i\in\naturals$,  a compact set $\setK_i$ such that $\setK_i\subseteq\setB$ and $\mu(\setB\mysetminus\setK_i)\leq 1/i$. For $j\in\naturals$, let $\setA_j=\bigcup_{\,i=1}^{\,j}\setK_i$. It follows that  $\{\setA_j\}_{j\in\naturals}$ is an increasing (in terms of $\subseteq$) sequence of compact sets $\setA_j\subseteq\setB$ satisfying $\mu(\setB)-\mu(\setA_j)=\mu(\setB\mysetminus\setA_j)\leq 1/j$ for all $j\in\naturals$. We set $\setA=\bigcup_{j\in\naturals}\setA_j\subseteq\setB$. Thus, 
$\mu(\setA)\leq\mu(\setB)$ by monotonicity  of $\mu$. Now,
\begin{align}
\mu(\setA)
&=\lim_{j\to\infty}\mu(\setA_j)\label{eq:lemappcapm}\\
&=\mu(\setB)-\lim_{j\to\infty}\mu(\setB\mysetminus\setA_j)\\
&\geq \mu(\setB)-\lim_{j\to\infty}\frac{1}{j}\\
&=\mu(\setB),
\end{align}
where \eqref{eq:lemappcapm} follows from Property \ref{lem.capmeasureii} in Lemma \ref{lem.capmeasure}.  
We conclude that $\mu(\setA)=\mu(\setB)$, which, together with $\setA\subseteq\setB$,  yields \eqref{eq:compunionN}.
\end{proof}

\begin{lem}\label{lem:sumunc} 
Let $x_i> 0$ for all $i\in\setI$ and set 
\begin{align}
M=\sup_{\setJ\subseteq \setI\,:\, \card{\setJ}<\infty}\sum_{i\in\setJ}x_i.  
\end{align}
Suppose that $M<\infty$. Then, $\setI$ is finite or countably  infinite. 
\end{lem}
\begin{proof}

For every $k\in\naturals_0$, set $I_k=\{i\in\setI:1/(k+1)\leq x_i < 1/k\}$ and   
\begin{align}
M_k= \sup_{\setJ\subseteq \setI_k: \card{\setJ}<\infty}\sum_{i\in\setJ}x_i. 
\end{align}
Since $M<\infty$ and $M_k\leq M$, we must have 
$M_k<\infty$  for all  $k\in\naturals_0$.
But for every  $k\in\naturals_0$, by the definition of $\setI_k$,  $M_k$ can only be finite if 
$\card{\setI_k}<\infty$.  
Thus, $\card{\setI_k}<\infty$ for all $k\in\naturals_0$. 
Since  $\setI=\bigcup_{k\in\naturals}\setI_k$, we conclude that $\setI$ is finite or countably infinite.  
\end{proof}

\subsection{Properties of Locally Lipschitz and Differentiable Mappings}

\begin{dfn}(Locally Lipschitz mapping)\cite[Definition 3.118]{gapa16} 
\begin{enumerate}[label=\roman*)]
\item
A mapping $f\colon\setU\to \reals^{l}$, where $\setU\subseteq\reals^{k}$, is Lipschitz if there exists a constant $L \geq 0$ such that 
\begin{align}\label{eq:lip}
\|f(\vecu)-f(\vecv)\|_2\leq L\|\vecu-\vecv\|_2\quad\text{for all}\ \vecu,\vecv\in\setU.
\end{align}
The smallest constant $L$ for \eqref{eq:lip} to  hold is  the Lipschitz constant of $f$.
\item
A mapping $f\colon\reals^{k}\to \reals^{l}$ is  locally Lipschitz if  every $\vecx\in\reals^k$  admits an open neighborhood  $\setU_\vecx\subseteq\reals^k$ containing $\vecx$ such that $f|_{\setU_\vecx}$ is Lipschitz.  
%\item A mapping $f\colon\reals^{k}\to \reals^{l}$ is locally bi-Lipschitz if it  is invertible and both $f$ and $f^{-1}$ are locally Lipschitz.
\end{enumerate}
\end{dfn}

The following result establishes a necessary and sufficient condition for a mapping to be locally Lipschitz. Since we could not find a proof for this statement in  the literature, we present one here  for completeness.  

\begin{lem}\label{lem:lipcompact}
The mapping $f\colon\reals^{k}\to \reals^{l}$ is  locally Lipschitz if and only if  $f|_\setK$ is Lipschitz for all compact sets $\setK\subseteq \reals^{k}$.
\end{lem}
\begin{proof}
,,$\Rightarrow$'': Suppose that $f\colon\reals^{k}\to \reals^{l}$ is  locally Lipschitz  and consider a compact set $\setK\subseteq\reals^k$. We have to show that $f|_\setK$ is Lipschitz.   For every $\vecx\in\setK$, by the local Lipschitz property of $f$,   there exists an  open neighborhood   $\setU_\vecx$ containing $\vecx$ such that $f|_{\setU_\vecx}$ is Lipschitz. 
Since $\setK\subseteq\bigcup_{\vecx\in\setK}\setU_\vecx$ is a cover of the compact set $\setK$ by open sets,  there must exist $\setU_{\vecx_i}$, denoted by $\setU_i$, $i=1,\dots,n$, such that $\setK\subseteq\bigcup_{i=1}^n\setU_i$. For  $i=1,\dots,n$, let $L_i$ denote the Lipschitz constant of $f|_{\setU_i}$. 
Now, as shown below,   there exists a $\delta >0$  such that, for every $\vecx,\vecy\in\setK$ with $\|\vecx-\vecy\|_2<\delta$, we can find  a $\setU_i$ in the finite subcover of $\setU$ with $\vecx,\vecy\in\setU_i$. With this $\delta$  we let 
\begin{align}\label{eq:upplip}
L=\max\mleft\{L_1,\dots,L_n,\frac{2\Delta}{\delta}\mright\}, 
\end{align}
where  $\Delta:=\max_{\vecx\in\setK}f(\vecx)$. Note that the local Lipschitz property of $f$ implies its continuity and, therefore, $f$ attains its maximum on the  compact set $\setK$. Consider an arbitrary  pair $\vecx,\vecy\in\setK$.
If $\|\vecx-\vecy\|_2<\delta$, then this pair must be in the same $\setU_i$, which implies  $\|f(\vecx)-f(\vecy)\|\leq L_i\|\vecx-\vecy\|\leq L\|\vecx-\vecy\|$. 
If $\|\vecx-\vecy\|_2\geq\delta$, then $\|f(\vecx)-f(\vecy)\|\leq 2\Delta=2\Delta\delta/\delta\leq (2\Delta/\delta)\|\vecx-\vecy\|_2\leq L \|\vecx-\vecy\|_2$. 
Thus, $f|_\setK$ is Lipschitz with Lipschitz constant at most $L$.
It remains to establish the existence of a $\delta$ with the desired properties. 
Towards a contradiction, suppose that such a $\delta$ does not exist. This implies that, for every $m\in\naturals$, we can find a pair $\vecx_m,\vecy_m\in\setK$ such that $\|\vecx_m-\vecy_m\|_2<1/m$, but there is no $\setU_i$ containing this pair. 
Since $\setK\times\setK$ is compact by Tychonoff's Theorem \cite[Theorem 4.42]{fo99},   there exists 
a convergent subsequence $\{\vecx_{m_j},\vecy_{m_j}\}_{j\in\naturals}$ of the sequence $\{\vecx_{m},\vecy_{m}\}_{m\in\naturals}$ \cite[Theorem 2.41]{ru89}. Let $(\bar\vecx,\bar\vecy)$ denote the limit point of this subsequence. Note that   $\bar\vecx$ and  $\bar \vecy$
cannot be in the same $\setU_i$, for if they were there would be an $M\in\naturals$ such that $(\vecx_{m_j},\vecy_{m_j})\in\setU_i\times\setU_i$  for all $j\geq M$, which is not possible as  $\vecx_m$ and $\vecy_m$ are in different sets $\setU_i$ for all $m\in\naturals$. 
But 
\begin{align}
\bar\vecx-\bar\vecy
&=\lim_{j\to\infty}(\vecx_{m_j}-\vecy_{m_j})\label{eq:limleqa}\\
&=\veczero,\label{eq:limleq}
\end{align}
where the second equality follows from $\lim_{j\to\infty}\|\vecx_{m_j}-\vecy_{m_j}\|_2\leq\lim_{j\to\infty} 1/m_j=0$ and the continuity of $\|\cdot\|_2$.
Thus,  $\bar \vecx=\bar\vecy$, which is not possible as there is no $\setU_i$ containing $\bar \vecx$ and $\bar \vecy$.

,,$\Leftarrow$'': Suppose that $f|_\setK$ is Lipschitz for all compact sets $\setK\subseteq \reals^k$. It is sufficient to show that 
$f|_{\setB_k(\vecx,1)}$ is Lipschitz for every  
$\vecx\in\reals^k$. As $\overline{\setB_k(\vecx,1)}$ is compact, $f|_{\overline{\setB_k(\vecx,1)}}$ is Lipschitz. The  Lipschitz property of $f|_{\setB_k(\vecx,1)}$ therefore follows  immediately from $\setB_k(\vecx,1)\subseteq\overline{\setB_k(\vecx,1)}$.  
\end{proof}

%\begin{cor}\label{cor:lipcompact}
%Let $f\colon\setD\to\reals^n$ with domain $\setD\subseteq\reals^m$ be locally Lipschitz and consider a bounded set $\setU\subseteq\setD$.  Then, $f|_\setU$ is Lipschitz. 
%\end{cor}
%\begin{proof}
%\end{proof}

\begin{lem}\label{lem:fK}
If $f\colon\reals^m\to\reals^n$ is locally Lipschitz and $\setK\subseteq\reals^m$ is compact, then $f(\setK)$ is compact. 
\end{lem}
\begin{proof}
Follows from the continuity of $f$ and \cite[Theorem 2-7.2]{ma95}.
\end{proof}

We will need the following composition property of locally Lipschitz mappings.   

\begin{lem}\label{lem:lip}
Suppose that  $f\colon\reals^k\to\reals^m$  and 
$g\colon\reals^m\to\reals^n$ are both  locally Lipschitz. Then, $h=g\circ f\colon\reals^k\to \reals^n$ is locally Lipschitz. 
\end{lem}
\begin{proof}
%Suppose that $f\colon\reals^k\to\reals^m$  and $g\colon\reals^m\to\reals^n$ are both locally Lipschitz. 
To prove that  $h=g\circ f$ is locally Lipschitz,   
by Lemma \ref{lem:lipcompact}, it is sufficient to show that $h|_\setK$ is Lipschitz for all compact sets $\setK\subseteq\reals^k$. 
Let $\setK\subseteq\reals^k$ be an arbitrary but fixed compact set and note that $\setQ=f(\setK)$ is compact owing  to Lemma \ref{lem:fK}. Thus, by Lemma 
\ref{lem:lipcompact}, $f|_\setK$ and $g|_\setQ$ are both Lipschitz with Lipschitz constants, say,  $L$ and $M$,  respectively. 
It follows that 
\begin{align}
\|h(\vecu)-h(\vecv)\|_2
&\leq M\|f(\vecu)-f(\vecv)\|_2\\
&\leq LM\|\vecu-\vecv\|_2
\end{align}
for all $\vecu$ and $\vecv$ in $\setK$, which implies that $h|_\setK$ is Lipschitz. As $\setK$ was arbitrary, we can conclude that  $h$ is locally Lipschitz. 
\end{proof}

\begin{dfn}(Differentiable mapping)\cite[Definition 1.2]{wa71}\label{dfn:diffmap}
Let $\setU$ be an open set in $\reals^m$.  
A function  $f\colon\setU\to \reals$ is 
\begin{enumerate}[label=\roman*)]
\item $C^0$ if it is continuous. 
\item\label{itemdiff2}
$C^r$ with $r\in\naturals$ if, for all possible choices of $r_1,\dots,r_m\in\naturals_0$ with $\sum_{i=1}^mr_i=r$, the partial derivatives  
\begin{align}
\frac{\partial^r}{\partial x_1^{r_1}\dots\,\partial x_m^{r_m}}f(\vecx)
\end{align} 
exist and are continuous  on $\setU$.
\item 
$C^\infty$ if it is $C^r$ for all  $r\in\naturals$.
\end{enumerate}
For $r\in\naturals_0\cup\{\infty\}$, a mapping  $f\colon\setU\to \reals^n$, $\vecx\mapsto \tp{(f_1(\vecx)\dots f_n(\vecx))}$ is $C^r$ if every component $f_i$, $i=1,\dots,n$, is $C^r$. 

\end{dfn}
It follows from the mean value theorem \cite[Theorem 3.4]{ed73} %applied  component-wise 
that $C^1$ mappings  are locally Lipschitz. 
Conversely, by Rademacher's Theorem \cite[Theorem. 5.1.11]{krpa08}, every locally Lipschitz mapping $f\colon\reals^m\to \reals^n$ has  
an $\colL(\reals^m)$-measurable (but not necessarily continuous) differential $Df$, which is defined $\lebmeasure^m$-almost everywhere. 

\begin{thm}(Sard's theorem)\label{thm:sard}\cite[Theorems 4.1 and 7.2]{sa42}
Let $f\colon\reals^m\to\reals^n$, $\vecx\mapsto\tp{(f_1(\vecx)\dots f_n(\vecx))}$ be $C^r$ and 
set $\setA=\{\vecx:J\! f(\vecx)=0\}$ with $J\! f(\vecx)$ as in \eqref{eq:Jacobiandef}. The following statements hold. 
\begin{enumerate}[label=\roman*)]
\item \label{thm:sard1}
If $m\leq n$, then  $\mathscr{H}^{m}(f(\setA))=0$.
\item \label{thm:sard2}
If $m>n$ and $r\geq m-n+1$, then $\lebmeasure^n(f(\setA))=0$.  
\end{enumerate} 
\end{thm}

\subsection{Area and Coarea Formula}

Next, we state two fundamental results from geometric measure theory that are used  frequently in the paper, namely,  the area and the coarea formula for locally Lipschitz mappings.  

\begin{thm}(Area formula)\label{thm:area}\cite[Theorem 5.1.1]{krpa08}
If $\setA\in\colB(\reals^m)$, $f\colon\reals^m\to\reals^n$  is  Lipschitz,  and $g\colon\setA\to \reals$ is nonnegative and Lebesgue measurable, and $m\leq n$, then  
\begin{align}
\int_\setA J\! f(\vecx)\,\mathrm d \lebmeasure^{m}(\vecx)
&=\int_{\reals^n} \operatorname{card}\mleft(\setA\cap f^{-1}(\{\vecy\})\mright) \,\mathrm d\mathscr{H}^{m}(\vecy).
\end{align}
\end{thm}

\begin{cor}(Area formula)\label{cor:area}
If $\setA\in\colB(\reals^m)$, $f\colon\reals^m\to\reals^n$  is locally Lipschitz and one-to-one on $\setA$,  and $m\leq n$, then  
\begin{align}
\int_\setA J\! f(\vecx)\,\mathrm d \lebmeasure^{m}(\vecx)
&=\mathscr{H}^{m}(f(\setA)).
\end{align}
\end{cor}
\begin{proof}
%For $i\in\naturals$, define $g_i \colon \reals^m \to \{0,1\}$ according to  $g_i=\ind{\setB_m(\veczero,i)}$. It follows that 
We have 
\begin{align}
&\int_\setA J\! f(\vecx)\,\mathrm d \lebmeasure^{m}(\vecx)\label{eq:area0}\\
&=\lim_{i\to\infty}\int_\setA \ind{\setB_m(\veczero,i)}(\vecx)J\! f(\vecx)\,\mathrm d \lebmeasure^{m}(\vecx)\label{eq:area1}\\
&=\lim_{i\to\infty}\int_\setA J (f|_{\setB_m(\veczero,i)})(\vecx)\,\mathrm d \lebmeasure^{m}(\vecx)\label{eq:area2}\\
&=\lim_{i\to\infty}\int_{\reals^n} \operatorname{card}\mleft(\setA\cap f|_{\setB_m(\veczero,i)}^{-1}(\{\vecy\})\mright) \,\mathrm d\mathscr{H}^{m}(\vecy)\label{eq:area3}\\
&=\lim_{i\to\infty}\mathscr{H}^{m}\mleft(f\big(\setA\cap\setB_m(\veczero,i)\big)\mright)  \label{eq:area4}\\
&=\mathscr{H}^{m}\mleft(\bigcup_{i\in\naturals}f\big(\setA\cap\setB_m(\veczero,i)\big)\mright)  \label{eq:area5}\\
&=\mathscr{H}^{m}(f(\setA)), \label{eq:uselebmes}
\end{align}
where 
\eqref{eq:area1} is by the Lebesgue Monotone Convergence Theorem \cite[Theorem 4.6]{ba95} upon noting that $(\ind{\setB_m(\veczero,i)})_{i\in\naturals}$ is an increasing sequence of nonnegative Lebesgue measurable functions converging pointwise to the constant function $1$, 
in \eqref{eq:area3} we applied Theorem \ref{thm:area} to  $f|_{\setB_m(\veczero,i)}$, which is Lipschitz by Lemma \ref{lem:lipcompact} as $\setB_m(\veczero,i)\subseteq\overline{\setB_m(\veczero,i)}$ and $\overline{\setB_m(\veczero,i)}$ is  compact  for all $i\in\naturals$, 
in   \eqref{eq:area4} we used that $f$ is one-to-one, by assumption, which implies 
\begin{align}
\operatorname{card}\mleft(\setA\cap f|_{\setB_m(\veczero,i)}^{-1}(\{\vecy\})\mright) 
=
\begin{cases}
1&\text{if}\ \vecy\in f\big(\setA\cap\setB_m(\veczero,i)\big)\\
0&\text{else},
\end{cases}
\end{align}
and \eqref{eq:area5} is by Property \ref{lem.capmeasureii} in Lemma \ref{lem.capmeasure}.
\end{proof}

If $m=n$, by  Property \ref{HBorel} in Lemma \ref{lem:prophausdorff}, $\mathscr{H}^{m}$ can be replaced by $\lebmeasure^{m}$
in Theorem \ref{thm:area} and in Corollary  \ref{cor:area}. 

\begin{thm}(Coarea formula)\label{thm:coarea}\cite[Theorem 5.2.1]{krpa08}
If $f\colon\reals^m\to\reals^n$ is  Lipschitz, $\setA\in\colB(\reals^m)$, and $m\geq n$, then  
\begin{align}
\int_\setA  J\! f(\vecx)    \,\mathrm d \lebmeasure^{m}(\vecx)
&=\int_{\reals^n} \mathscr{H}^{m-n}\mleft(\setA\cap f^{-1}(\{\vecy\})\mright) \,\mathrm d\lebmeasure^{n}(\vecy).
\end{align}
\end{thm}

\begin{cor}(Coarea formula)\label{cor:coarea}
If $f\colon\reals^m\to\reals^n$ is locally Lipschitz, $\setA\in\colB(\reals^m)$, and $m\geq n$, then  
\begin{align}
\int_\setA  J\! f(\vecx)    \,\mathrm d \lebmeasure^{m}(\vecx)
&=\int_{\reals^n} \mathscr{H}^{m-n}\mleft(\setA\cap f^{-1}(\{\vecy\})\mright) \,\mathrm d\lebmeasure^{n}(\vecy).
\end{align}
\end{cor}
\begin{proof}
We have 
\begin{align}
&\int_\setA J\! f(\vecx)\,\mathrm d \lebmeasure^{m}(\vecx)\\
&=\lim_{i\to\infty}\int_\setA J (f|_{\setB_m(\veczero,i)})(\vecx)\,\mathrm d \lebmeasure^{m}(\vecx)\label{eq:coarea1}\\
&=\lim_{i\to\infty} \int_{\reals^n} \mathscr{H}^{m-n}\mleft(\setA\cap f|_{\setB_m(\veczero,i)}^{-1}(\{\vecy\})\mright) \,\mathrm d\lebmeasure^{n}(\vecy)\label{eq:coarea2}\\
%&= \int_{\reals^n}\lim_{i\to\infty} \mathscr{H}^{m-n}\mleft(\setA\cap (f|_{\setB_m(\veczero,i)})^{-1}(\{\vecy\})\mright) \,\mathrm d\lebmeasure^{n}(\vecy)\label{eq:coarea3}\\
&= \int_{\reals^n} \mathscr{H}^{m-n}\mleft(\bigcup_{i\in\naturals}\setA\cap f|_{\setB_m(\veczero,i)}^{-1}(\{\vecy\})\mright) \,\mathrm d\lebmeasure^{n}(\vecy)\label{eq:coarea4}\\
&=\int_{\reals^n} \mathscr{H}^{m-n}\mleft(\setA\cap f^{-1}(\{\vecy\})\mright) \,\mathrm d\lebmeasure^{n}(\vecy), 
\end{align}
where 
\eqref{eq:coarea1} follows from \eqref{eq:area0}--\eqref{eq:area2}, 
in \eqref{eq:coarea2} we applied Theorem \ref{thm:coarea} to $f|_{\setB_m(\veczero,i)}$, which is Lipschitz by Lemma \ref{lem:lipcompact} as $\setB_m(\veczero,i)\subseteq\overline{\setB_m(\veczero,i)}$ and $\overline{\setB_m(\veczero,i)}$ is  compact  for all $i\in\naturals$, and 
\eqref{eq:coarea4} is by the Lebesgue Monotone Convergence Theorem \cite[Theorem 4.6]{ba95} upon noting that, for every $i\in\naturals$, the  function 
\begin{align}
g_i\colon \reals^n&\to \reals\label{eq:gi1a}\\
\vecy&\mapsto \mathscr{H}^{m-n}\mleft(\setA\cap f|_{\setB_m(\veczero,i)}^{-1}(\{\vecy\})\mright)\label{eq:gi2a}
\end{align}
is Lebesgue measurable \cite[Lemma 5.2.5]{krpa08} and, therefore,  $(g_i)_{i\in\naturals}$ is a sequence of nonnegative 
increasing  Lebesgue measurable functions,  
%Lebesgue measurable by \cite{} 
 with 
\begin{align}
\lim_{i\to\infty} g_i(\vecy)= \mathscr{H}^{m-n}\mleft(\bigcup_{i\in\naturals}\setA\cap f|_{\setB_m(\veczero,i)}^{-1}(\{\vecy\})\mright)
\end{align}  
for all $\vecy\in\reals^n$ by Property \ref{lem.capmeasureii} in Lemma \ref{lem.capmeasure}.  
\end{proof}

\subsection{Properties of  Modified Minkowski Dimension}\label{sec:dimMB}
In this section, we state some properties of  modified Minkowski dimension (see Definitions \ref{dfndim} and \ref{dfndimlocal}). 

\begin{lem}\label{lem:Mpropmindim} Main properties of modified Minkowski dimension:
\begin{enumerate}[label=\roman*)]
\item 
We have
\label{Mequiv}
\begin{align} \label{eq:stablower}
\underline{\dim}_\mathrm{MB}(\setU)=\inf\mleft\{\sup_{i\in\naturals} \underline{\dim}_\mathrm{B}(\setU_i) : \setU\subseteq \bigcup_{i\in\naturals}\setU_i\mright\}
\end{align} 
and
\begin{align} \label{eq:stabupper}
\overline{\dim}_\mathrm{MB}(\setU)=\inf\mleft\{\sup_{i\in\naturals} \overline{\dim}_\mathrm{B}(\setU_i) : \setU\subseteq \bigcup_{i\in\naturals}\setU_i\mright\},
\end{align}
respectively, where in \eqref{eq:stablower}  and \eqref{eq:stabupper}  the infima are over all possible  coverings $\{\setU_i\}_{i\in\naturals}$ of $\setU$ by nonempty compact sets $\setU_i$. 
\item \label{Mpropmanifold}
Every  $s$-dimensional $C^1$-submanifold  \cite[Definition 5.3.1]{krpa08} $\setU$ of $\reals^{m}$ has $\dim_\mathrm{MB}(\setU)=s$.
\item\label{Mpropmonotonic}
$\underline{\dim}_\mathrm{MB}(\cdot)$ and $\overline{\dim}_\mathrm{MB}(\cdot)$ are monotonically nondecreasing with respect to $\subseteq$.
\item\label{MpropMltB}
$\underline{\dim}_\mathrm{MB}(\setU)\leq \underline{\dim}_\mathrm{B}(\setU)$ %\label{eq:MBp1}\\
and $\overline{\dim}_\mathrm{MB}(\setU)\leq\overline{\dim}_\mathrm{B}(\setU)$ %\label{eq:MBp2}. 
 for all nonempty  sets $\setU$. 
\item \label{MpropL2} If $f$ is Lipschitz, then 
\begin{align}
\overline{\dim}_\mathrm{MB}(f(\setU))\leq \overline{\dim}_\mathrm{MB}(\setU)\\
\underline{\dim}_\mathrm{MB}(f(\setU))\leq \underline{\dim}_\mathrm{MB}(\setU)
\end{align}
for all nonempty subsets $\setU$ in the domain of $f$.
\item \label{Mpropcountablestable}
$\overline{\dim}_\mathrm{MB}$ and $\underline{\dim}_\mathrm{MB}$ are countably stable, i.e., 
\begin{align}\label{eq:stableupper}
\overline{\dim}_\mathrm{MB}\mleft(\bigcup_{i\in\naturals}\setU_i\mright)=\sup_{i\in\naturals}\overline{\dim}_\mathrm{MB}(\setU_i)
\end{align}
and 
\begin{align}\label{eq:stablelower}
\underline{\dim}_\mathrm{MB}\mleft(\bigcup_{i\in\naturals}\setU_i\mright)=\sup_{i\in\naturals}\underline{\dim}_\mathrm{MB}(\setU_i)
\end{align}
for all countable collections of nonempty sets $\setU_i$, $i\in\naturals$. 
\item\label{MpropL} Let  $f\colon \reals^m\to\reals^n$ be locally Lipschitz. Then,     
\begin{align}
\overline{\dim}_\mathrm{MB}(f(\setU))\leq \overline{\dim}_\mathrm{MB}(\setU)\\
\underline{\dim}_\mathrm{MB}(f(\setU))\leq \underline{\dim}_\mathrm{MB}(\setU)
\end{align}
for all nonempty subsets $\setU\subseteq\reals^m$. 
\end{enumerate}
\end{lem}
\begin{proof} 
Property \ref{Mequiv} is by invariance of $\overline{\dim}_\mathrm{B}$ and $\underline{\dim}_\mathrm{B}$ under  set closure \cite[Chapter 2.2]{fa14} and the Heine-Borel theorem \cite[Theorem 2.41]{ru89}. 

Properties \ref{Mpropmanifold}--\ref{MpropMltB} follow from Definition  \ref{dfndimlocal} and 
the properties of lower and upper Minkowski dimension listed in \cite[Chapter 2.2]{fa14}. 

We prove Properties \ref{MpropL2}--\ref{MpropL} for $\underline{\dim}_\mathrm{MB}$ only. The corresponding arguments for $\overline{\dim}_\mathrm{MB}$ are along the same lines.

To prove Property \ref{MpropL2} for $\underline{\dim}_\mathrm{MB}$, let $f\colon\setD\to\reals^n$ be Lipschitz with domain $\setD\subseteq\reals^m$ and consider $\setU\subseteq\reals^m$.  
By \cite[Theorem 2]{ma95}, there exists a Lipschitz mapping  $g\colon\reals^m\to\reals^n$ with $g|_\setD= f$. We have 
%We can thus assume, w.l.o.g., that $\setD=\reals^m$.  We have 
\begin{align}
\underline{\dim}_\mathrm{MB}(f(\setU))
&=\underline{\dim}_\mathrm{MB}(g(\setU))\\
&=\inf\mleft\{\sup_{i\in\naturals} \underline{\dim}_\mathrm{B}(\setV_i) : g(\setU)\subseteq \bigcup_{i\in\naturals}\setV_i\mright\}\label{eq:lopB1}\\
&\leq\inf\mleft\{\sup_{i\in\naturals} \underline{\dim}_\mathrm{B}(g(\setU_i)) : \setU\subseteq \bigcup_{i\in\naturals}\setU_i\mright\}\label{eq:lopB2}\\
&\leq\inf\mleft\{\sup_{i\in\naturals} \underline{\dim}_\mathrm{B}(\setU_i) : \setU\subseteq \bigcup_{i\in\naturals}\setU_i\mright\}\label{eq:lopB3}\\
&=\underline{\dim}_\mathrm{MB}(\setU),\label{eq:lopB4}
\end{align}
 where  \eqref{eq:lopB1} is by Property \ref{Mequiv} with the infimum   over all possible  coverings $\{\setV_i\}_{i\in\naturals}$ of $g(\setU)$ by nonempty compact sets $\setV_i\subseteq\reals^n$,  
in \eqref{eq:lopB2} we used that, by Lemma \ref{lem:fK}, Lipschitz images of compact sets are again compact  with the infimum  over all possible  coverings $\{\setU_i\}_{i\in\naturals}$ of $\setU$ by nonempty compact sets $\setU_i\subseteq\reals^m$, 
in \eqref{eq:lopB3} we used that  $\underline{\dim}_\mathrm{B}(g(\setU_i))\leq \underline{\dim}_\mathrm{B}(\setU_i)$ for all $i\in\naturals$ as $g$ is Lipschitz \cite[Proposition 2.5, Property (a)]{fa14}, and 
 \eqref{eq:lopB4} is again by  Property \ref{Mequiv}. 

Property  \ref{Mpropcountablestable} is stated in \cite[Chapter 2.3]{fa14} without proof. 
% in fact, \eqref{eq:stableupper} follows by combining \cite[Equation (3.24)]{fa14}   and \cite[Proposition 3.9]{fa14}, but this argument does not allow to establish \eqref{eq:stablelower}. 
For the sake of completeness, we prove \eqref{eq:stablelower}.  
%and our proof technique applies to \eqref{eq:stableupper} as well. 
Let 
\begin{align}
\setU=\bigcup_{i\in\naturals}\setU_i
\end{align}
with $\setU_i\subset\reals^m$ nonempty for all $i\in\naturals$. 
By the monotonicity of $\underline{\dim}_\mathrm{MB}$, we have 
\begin{align}
\underline{\dim}_\mathrm{MB}\mleft(\setU\mright)\geq\sup_{i\in\naturals}\underline{\dim}_\mathrm{MB}(\setU_i). 
\end{align}
It remains to show that 
\begin{align}\label{eq:toshowstability}
\underline{\dim}_\mathrm{MB}\mleft(\setU\mright)\leq\sup_{i\in\naturals}\underline{\dim}_\mathrm{MB}(\setU_i). 
\end{align}
Suppose first that the $\setU_i$ are all bounded.  It follows that 
\begin{align}
\underline{\dim}_\mathrm{MB}\mleft(\setU\mright)
&=\inf\mleft\{\sup_{i\in\naturals} \underline{\dim}_\mathrm{B}(\setV_i) :\setU\subseteq \bigcup_{i\in\naturals}\setV_i\mright\}\label{eq:loclipMB1}\\
&=\inf\mleft\{\sup_{i\in\naturals} \underline{\dim}_\mathrm{MB}(\setV_i) :\setU\subseteq \bigcup_{i\in\naturals}\setV_i\mright\}\label{eq:loclipMB2}\\
&\leq \sup_{i\in\naturals} \underline{\dim}_\mathrm{MB}(\setU_i),  \label{eq:loclipMB3}
\end{align}
where \eqref{eq:loclipMB1} is by Definition  \ref{dfndimlocal} with the infimum   over all possible  coverings $\{\setV_i\}_{i\in\naturals}$ of $\setU$ by nonempty bounded sets $\setV_i\subseteq\reals^m$, 
in \eqref{eq:loclipMB2} we applied Lemma \ref{lem:MBB} below, and in 
 \eqref{eq:loclipMB3} we used that, by assumption,  $\{\setU_i\}_{i\in\naturals}$ is a covering of $\setU$ consisting of  nonempty bounded sets $\setU_i\subseteq\reals^m$,  which establishes \eqref{eq:toshowstability} for the case where the $\setU_i$ are all bounded. 

Next, suppose  that the $\setU_i$ are not necessarily all bounded. We can write 
\begin{align}
\setU=\bigcup_{i,j\in\naturals}\setC_{i,j}
\end{align} 
with $\setC_{i,j}=\setU_i\cap\setB_m(\veczero,j)$ for all $i,j\in\naturals$, so that  
\begin{align}
\underline{\dim}_\mathrm{MB}\mleft(\setU\mright)
&=\sup_{i,j\in\naturals} \underline{\dim}_\mathrm{MB}(\setC_{i,j})\label{eq:loclipMBG1}\\
&\leq \sup_{i\in\naturals} \underline{\dim}_\mathrm{MB}(\setU_{i})\label{eq:loclipMBG2},
\end{align}
where in \eqref{eq:loclipMBG1} we rely on Property  \ref{Mpropcountablestable} for nonempty bounded sets, and 
\eqref{eq:loclipMBG2} follows from $\setC_{i,j}\subseteq\setU_i$ for all $i,j\in\naturals$ and the monotonicity of $\underline{\dim}_\mathrm{MB}$; this establishes \eqref{eq:toshowstability} for the case where the $\setU_i$ are not necessarily  all bounded. 

It remains to establish Property  \ref{MpropL}. Consider a locally Lipschitz mapping $f\colon \reals^m\to\reals^n$  and  let  $\setU\subseteq\reals^m$ be nonempty. We have 
\begin{align}
\underline{\dim}_\mathrm{MB}(f(\setU))
&=\underline{\dim}_\mathrm{MB}\mleft(  \bigcup_{j\in\naturals}f(\setU\cap\setB_m(\veczero,j))\mright)\label{eq:dimMBlip1}\\
&=\sup_{j\in\naturals}\mleft(\underline{\dim}_\mathrm{MB}\mleft(  f(\setU\cap\setB_m(\veczero,j))\mright)\mright)\label{eq:dimMBlip2}\\
&\leq \sup_{j\in\naturals}\mleft(\underline{\dim}_\mathrm{MB}\mleft(\setU\cap\setB_m(\veczero,j)\mright)\mright)\label{eq:dimMBlip3}\\
&=\underline{\dim}_\mathrm{MB}\mleft(  \bigcup_{j\in\naturals}\setU\cap\setB_m(\veczero,j)\mright)\label{eq:dimMBlip4}\\
&=\underline{\dim}_\mathrm{MB}(\setU),
\end{align}
where \eqref{eq:dimMBlip2} and \eqref{eq:dimMBlip4} are by countable stability of $\underline{\dim}_\mathrm{MB}$, and  
in \eqref{eq:dimMBlip3} we applied Property \ref{MpropL2} to $f|_{\overline{\setU\cap\setB_m(\veczero,j)}}$, which is Lipschitz thanks to 
Lemma \ref{lem:lipcompact} with $\overline{\setU\cap\setB_m(\veczero,j)}$ compact by the  Heine-Borel theorem \cite[Theorem 2.41]{ru89}, for all $j\in\naturals$. 
\end{proof}

\begin{lem}\label{lem:MBB}
Let $\setU\subseteq\reals^m$ be nonempty. Then,  $\underline{\dim}_\mathrm{MB}(\setU)= \underline{F}(\setU)$ with 
\begin{align}
\underline{F}(\setU)=\inf\mleft\{\sup_{i\in\naturals} \underline{\dim}_\mathrm{MB}(\setU_i) : \setU\subseteq \bigcup_{i\in\naturals}\setU_i\mright\}
\end{align} 
and
$\overline{\dim}_\mathrm{MB}(\setU)= \overline{F}(\setU)$ with 
\begin{align}
\overline{F}(\setU)=\inf\mleft\{\sup_{i\in\naturals} \overline{\dim}_\mathrm{MB}(\setU_i) : \setU\subseteq \bigcup_{i\in\naturals}\setU_i\mright\},
\end{align} 
respectively, where in both cases the infimum  is over all possible  coverings $\{\setU_i\}_{i\in\naturals}$ of $\setU$ by nonempty bounded sets $\setU_i\subseteq\reals^m$. 
\end{lem}
\begin{proof}
We present a proof for $\underline{\dim}_\mathrm{MB}$ only, the  arguments for $\overline{\dim}_\mathrm{MB}$ are along the same lines. 
First note that $\underline{F}(\setU)\leq\underline{\dim}_\mathrm{MB}(\setU)$ as   $\underline{\dim}_\mathrm{MB}\leq \underline{\dim}_\mathrm{B}$ by Property \ref{MpropMltB} in Lemma  \ref{lem:Mpropmindim}. It remains to show that  $\underline{F}(\setU)\geq\underline{\dim}_\mathrm{MB}(\setU)$. 
Suppose, towards a contradiction, that $\underline{F}(\setU)<\underline{\dim}_\mathrm{MB}(\setU)$ and set 
$\Delta=\underline{\dim}_\mathrm{MB}(\setU)-\underline{F}(\setU)>0$. By definition of $\underline{F}(\setU)$, there must exist a collection $\{\setU_i\}_{i\in\naturals}$ of nonempty bounded sets $\setU_i\subseteq\reals^m$ such that 
\begin{align}
\setU\subseteq \bigcup_{i\in\naturals}\setU_i\label{eq:cupU}
\end{align} 
and 
\begin{align}\label{eq:sup1}
\sup_{i\in\naturals}\underline{\dim}_\mathrm{MB}(\setU_i)\leq\underline{F}(\setU)+\frac{\Delta}{3}. 
\end{align}
Similarly, by Definition  \ref{dfndimlocal}, for every $i\in\naturals$, there must exist a collection $\{\setV^{(i)}_j\}_{j\in\naturals}$ of nonempty bounded sets $\setV^{(i)}_j$ such that 
\begin{align}\label{eq:cupUi}
\setU_i\subseteq \bigcup_{j\in\naturals}\setV^{(i)}_j
\end{align} 
and 
\begin{align}\label{eq:sup2}
\sup_{j\in\naturals}\underline{\dim}_\mathrm{B}\mleft(\setV^{(i)}_j\mright)\leq\underline{\dim}_\mathrm{MB}(\setU_i)+\frac{\Delta}{3}. 
\end{align}
Combining \eqref{eq:cupU} with \eqref{eq:cupUi} yields 
\begin{align}\label{eq:cupU2}
\setU\subseteq \bigcup_{i,j\in\naturals}\setV^{(i)}_j.  
\end{align} 
Next, note that  
\begin{align}
\underline{\dim}_\mathrm{MB}(\setU)
&\leq \sup_{i,j\in\naturals}\underline{\dim}_\mathrm{B}\mleft(\setV^{(i)}_j\mright)\label{eq:step1F}\\
&=\sup_{i\in\naturals}\sup_{j\in\naturals}\underline{\dim}_\mathrm{B}\mleft(\setV^{(i)}_j\mright)\label{eq:step2F}\\
&\leq\sup_{i\in\naturals}\underline{\dim}_\mathrm{MB}(\setU_i)+\frac{\Delta}{3} \label{eq:step3F}\\
&\leq \underline{F}(\setU)+\frac{2\Delta}{3} \label{eq:step4F}\\
&<\underline{F}(\setU)+\Delta,
\end{align} 
where \eqref{eq:step1F} follows from Definition  \ref{dfndimlocal} and \eqref{eq:cupU2}, 
in \eqref{eq:step3F} we used \eqref{eq:sup2}, and  \eqref{eq:step4F} is by \eqref{eq:sup1}. 
This is in contradiction to $\Delta=\underline{\dim}_\mathrm{MB}(\setU)-\underline{F}(\setU)$.  
\end{proof}

\section{Properties of Set-Valued Mappings}\label{app:set}
A set-valued mapping $\Phi\colon\setT\to 2^{\reals^m}$ associates to each $x\in\setT$ a set $\Phi(x)\subseteq\reals^m$. Many  properties of ordinary mappings such as, e.g., measurability, can be extended to set-valued mappings. In this appendix, we first briefly review  properties of set-valued mappings and then state a result  needed in the existence proof of a measurable decoder in Section \ref{sec.existence}.  

\begin{dfn}(Closed-valuedness of set-valued mappings)\cite[Chapter 5]{rowe98}\label{dfn:setvaluedclosed}
A set-valued mapping $\Phi\colon\setT\to 2^{\reals^m}$ is closed-valued if, for every  $t\in\setT$, the set  $\Phi(t)$ is closed.
\end{dfn}

\begin{dfn}(Inverse image of a set-valued mapping)\cite[Chapter 14]{rowe98}
For a  set-valued mapping $\Phi\colon\setT\to 2^{\reals^m}$, the inverse image $\Phi^{-1}(\setA)$ of $\setA\subseteq\reals^m$ is 
\begin{align}
\Phi^{-1}(\setA)=\{t\in\setT:\Phi(t)\cap\setA\neq\emptyset\}. 
\end{align}
\end{dfn}

\begin{dfn}(Measurable set-valued mapping)\cite[Chapter 14]{rowe98}\label{dfn:setvaluedmeasurable}
Let  $(\setT,\colS(\setT))$ be a measurable space.  
A set-valued mapping $\Phi\colon\setT\to 2^{\reals^m}$ is $\colS(\setT)$-measurable if, for every  open set $\setO\subseteq\reals^m$, the inverse image $\Phi^{-1}(\setO)\in\colS(\setT)$.
\end{dfn}

\begin{lem}\cite[Theorem 14.3]{rowe98}\label{lem:setvalued}
Let  $(\setT,\colS(\setT))$ be a measurable space and $\Phi\colon\setT\to 2^{\reals^m}$  closed-valued. Then, $\Phi$ is $\colS(\setT)$-measurable
if and only if $\Phi^{-1}(\setK)\in\colS(\setT)$ for all compact sets $\setK\subseteq\reals^m$.
\end{lem}

\begin{lem}\cite[Corollary 14.6]{rowe98}\label{lem:section}
Let  $(\setT,\colS(\setT))$ be a measurable space and  $\Phi\colon\setT\to 2^{\reals^m}$ an $\colS(\setT)$-measurable closed-valued mapping. Then, there exists an $\colS(\setT)$-measurable mapping $f\colon\Phi^{-1}(\reals^m)\to\reals^m$ such that $f(t)\in\Phi(t)$ for all $t\in\Phi^{-1}(\reals^m)$. 
\end{lem}

\begin{dfn}(Normal integrand)\cite[Definition 14.27]{rowe98}\label{dfn:normal}
Let  $(\setT,\colS(\setT))$ be a measurable space.  
An extended real-valued function $f\colon\setT\times\reals^m\to \overline{\reals}$ is a  normal integrand with respect to $\colS(\setT)$ if its epigraphical mapping 
\begin{align}
S_f\colon\setT&\to2^{\reals^m\times\reals}\\
t&\mapsto\{(\vecx,\alpha)\in \reals^m\times\reals:f(t,\vecx)\leq\alpha\}
\end{align}
is closed-valued and $\colS(\setT)$-measurable. 
\end{dfn}

\begin{lem}\cite[Example 14.31]{rowe98}\label{lem:normal}
Let $\setT=\reals^{n\times m}\times\reals^n$ and suppose that $f\colon\setT\times\reals^m\to\overline{\reals}$ is continuous. Then,   
$f$ is a normal integrand with respect to $\colB(\setT)$.  
\end{lem}

\begin{lem}\cite[Proposition 14.33]{rowe98}\label{lem:normal2}
Let  $(\setT,\colS(\setT))$ be a measurable space. 
An extended real-valued function 
$f\colon\setT\times\reals^m\to\overline{\reals}$  is a normal integrand with respect to $\colS(\setT)$ if and only if,  
for every $\alpha\in\overline{\reals}$, the level-set mapping
\begin{align}
L_\alpha\colon\setT&\to2^{\reals^m}\\
t&\mapsto \{\vecx\in\reals^m: f(t,\vecx)\leq \alpha\}
\end{align}
is $\colS(\setT)$-measurable and closed-valued. 
\end{lem}

We can now state the result on set-valued mappings needed to prove the existence of a measurable decoder in Section \ref{sec.existence}. 

\begin{lem}\label{lem:normal3}
Let  $(\setT,\colS(\setT))$ be a measurable space,  
$\alpha\in\reals$, and 
$f\colon\setT\times\reals^m\to\overline{\reals}$.  Suppose that $f$ is a 
normal integrand with respect to $\colS(\setT)$ and $\setK\subseteq\reals^m$ is compact and nonempty.  
Then, the following properties hold.
\begin{enumerate}[label=\roman*)]
\item\label{lem:normal3i}
The set-valued mapping 
\begin{align}
P_\setK\colon\setT&\to2^{\reals^m}\\
t&\mapsto \{\vecx\in\setK: f(t,\vecx)\leq \alpha\}
\end{align}
is $\colS(\setT)$-measurable and closed-valued. 
\item\label{lem:normal3ii} 
$P_\setK^{-1}(\reals^m)=\{t\in\setT:\exists\,\vecx\in\setK\ \text{with}\ f(t,\vecx)\leq\alpha\}\in \colS(\setT).$
\item \label{lem:normal3iii}
There exists an $\colS(\setT)$-measurable mapping
\begin{align}
p_\setK\colon P_\setK^{-1}(\reals^m) &\to\reals^m\\
t&\mapsto p_\setK(t)\in P_\setK(t).
\end{align}
\end{enumerate}
\end{lem}
\begin{proof} 
We start with the proof of \ref{lem:normal3i}. 
For each  $t\in\setT$, we can write $P_\setK(t)=L_\alpha(t)\cap\setK$, where $L_\alpha$ is the $\colS(\setT)$-measurable closed-valued level-set mapping from Lemma \ref{lem:normal2}. 
Since $L_\alpha$ is closed-valued and the intersection of a closed set with a compact set is closed, $P_\setK$ is closed-valued. 
To prove that $P_\setK$ is $\colS(\setT)$-measurable it  suffices, thanks to Lemma \ref{lem:setvalued},  to show that $P_\setK^{-1}(\setA)\in\colS(\setT)$ for all compact sets $\setA\subseteq\reals^m$. To this end, let $\setA\subseteq\reals^m$ be an arbitrary but fixed compact set. 
Since the intersection of two compact sets is compact it follows that $\setK\cap\setA$ is compact. 
As $L_\alpha$ is $\colS(\setT)$-measurable and $\setK\cap\setA$ is compact,  $L_\alpha^{-1}(\setK\cap\setA)\in\colS(\setT)$ by Lemma \ref{lem:setvalued}. Therefore, as  $L_\alpha^{-1}(\setK\cap\setA)=P_\setK^{-1}(\setA)$, we can conclude that  $P_\setK^{-1}(\setA)\in \colS(\setT)$. 
Since $\setA$ was arbitrary, this proves \ref{lem:normal3i}. 
Now, $\colS(\setT)$-measurability of $P_\setK$  implies  $P_\setK^{-1}(\reals^m)\in \colS(\setT)$, and thereby \ref{lem:normal3ii}.  
Finally, the existence of  the  $\colS(\setT)$-measurable mapping $p_\setK$ in \ref{lem:normal3iii} follows from \ref{lem:normal3i} and Lemma \ref{lem:section}. 
\end{proof}

\section{Properties of Sequences of Functions in Several Variables}\label{app:seq}
In this appendix, we summarize properties of sequences of functions in several variables needed in the proof of Theorem \ref{prpcounter}. 
We start with a result that establishes a sufficient condition for  uniform convergence. 

\begin{thm}(Weierstrass $M$-test)\label{thm.M}\cite[Theorem 7.10]{ru89}
Consider a sequence   $\{f_n\}_{n\in\naturals}$  of functions 
$f_n\colon \reals^{k}\to\reals$   
and suppose that there exists a sequence $\{M_n\}_{n\in\naturals}$ of nonnegative real numbers    such that 
\begin{align}
|f_n(\vecx)|\leq M_n\quad\text{for all}\ \vecx\in\reals^k\ \text{and}\ n\in\naturals  
\end{align}
and $\sum_{n\in\naturals} M_n<\infty$. 
Then, the sequence   $\{s_n\}_{n\in\naturals}$ 
of partial sums 
\begin{align}
s_n=\sum_{i=1}^n f_i
\end{align}
converges  uniformly to a function $f\colon\reals^{k}\to \reals$.  
\end{thm}

Next, we state a result that allows us to interchange the order of summation and differentiation for certain sequences of differentiable functions.    
We start with the corresponding statement for differentiable functions in one variable. 

\begin{thm}\cite[Theorem 7.17]{ru89}\label{thm.derexchlimit}
Consider a sequence $\{f_n\}_{n\in\naturals}$ of functions   $f_n\colon \reals\to\reals$, each of  which  is differentiable on the closed interval $[a,b]\subseteq\reals$, such that $\lim_{n\to\infty} f_n(x_0)$ exists and is finite for at least one $x_0\in [a,b]$. Let $f_n^\prime$ denote the derivative of $f_n$, $n\in\naturals$. If $\{f_n^\prime\}_{n\in\naturals}$ converges uniformly on $[a,b]$, then $\{f_n\}_{n\in\naturals}$ converges uniformly on $[a,b]$ to a function $f\colon\reals\to \reals$ satisfying 
\begin{align}
\frac{\mathrm d f(x)}{\mathrm d x}=\lim_{n\to\infty} f_n^\prime(x)\quad\text{for all}\  x\in[a,b].
\end{align}
\end{thm}

\begin{cor}\label{cor.exchangesumder}
Consider a  sequence   $\{f_n\}_{n\in\naturals}$  of functions $f_n\colon\reals^{k}\to\reals$ converging uniformly to $f\colon\reals^{k}\to\reals$. 
Suppose that there exists an  $i\in\{1,\dots,k\}$ such that  
the partial derivatives $\partial f_n(\vecx)/\partial x_i$ exist and are finite for all $\vecx\in\reals^k$ and  $n\in\naturals$    
and  that  the sequence  $\{\partial f_n/\partial x_i\}_{n\in\naturals}$  converges uniformly.  Then, 
\begin{align}
%f^\prime(\vecx)
\frac{\mathrm \partial f(\vecx)}{\mathrm \partial x_i}
&=\lim_{n\to\infty} \frac{\mathrm \partial f_n(\vecx)}{\mathrm \partial x_i} \quad\text{for all}\  \vecx\in\reals^k. \label{eq:toshowderlimit}
\end{align}
\end{cor}
\begin{proof}
%Suppose that the sequence  $\{f_n^\prime\}_{n\in\naturals}$  converges uniformly. 
Let $\vecx=\tp{(x_1\mydots x_k)}\in\reals^k$ be arbitrary but fixed and denote by $\{g_n\}_{n\in\naturals}$ the sequence of functions  
defined according to $g_n(t)=f_n(\tp{(x_1 \mydots x_{i-1}\ t\ x_{i+1}\mydots x_k)})$.  
Since $\{f_n\}_{n\in\naturals}$ converges uniformly to  $f$ by assumption,  and $g_n(x_i)=f_n(\vecx)$ for all $n\in\naturals$, it follows that  
\begin{align}
 \lim_{n\to\infty} g_n(x_i)
 &=\lim_{n\to\infty} f_n(\vecx)\\
 &=f(\vecx).
\end{align}
%For each $n\in\naturals$, let $g_n^\prime$ denote the derivative of $g_n$.
Since  $\{\partial f_n/\partial x_i\}_{n\in\naturals}$ converges uniformly, %, again by assumption,  
so does $\{\mathrm d  g_n/\mathrm d  t\}_{n\in\naturals}$. In particular, 
there exists a closed interval $[a,b]\subseteq\reals$ containing $x_i$ such that 
$\{\mathrm d  g_n/\mathrm d  t\}_{n\in\naturals}$ converges uniformly on that interval 
$[a,b]$. Theorem \ref{thm.derexchlimit} therefore implies that $\{g_n\}_{n\in\naturals}$ converges uniformly on $[a,b]$ to a function $g\colon\reals\to \reals$ satisfying 
\begin{align}\label{eq:dgprime}
\frac{\mathrm d g(t)}{\mathrm d t}
&=\lim_{n\to\infty} \frac{\mathrm dg_n(t)}{\mathrm d t}\quad\text{for all}\ t\in [a,b]. 
\end{align} 
But $g(x_i)= f(\vecx)$ and $\mathrm d  g_n(x_i)/\mathrm d  x_i= \partial f_n(\vecx)/\partial x_i$ for all $n\in \naturals$,  which implies 
\begin{align}
\mathrm \partial f(\vecx)/\mathrm \partial x_i=\lim_{n\to\infty}  \partial f_n(\vecx)/\partial x_i. 
\end{align}
As  $\vecx$  was arbitrary, this finishes the proof. 
\end{proof}

\section{Properties of Real Analytic Mappings and  Consequences Thereof} \label{app:rea}
In this appendix, we review material on real analytic mappings, on which our strong converse result in Sections \ref{sec:analytic} and \ref{proof.converse} relies.  
We start with the definition of a real analytic mapping. 

\begin{dfn}(Real analytic mapping)\cite[Definition 2.2.1]{krpa92}\label{def:real}
Let $\setU$ be an open set in $\reals^m$.  
\begin{enumerate}[label=\roman*)]
\item 
A function  $f\colon\setU\to \reals$  is real analytic  on $\setU$ if, for each $\vecx\in\setU$, $f$ may be represented by a convergent power series (see \cite[Definition 2.1.4]{krpa92}) in some open neighborhood of $\vecx$; if $\setU=\reals^m$, then $f$ is real analytic.
\item 
A mapping  $f\colon\setU\to \reals^n$, $\vecx\mapsto \tp{(f_1(\vecx)\dots f_n(\vecx))}$ is real analytic on  $\setU$ if every component $f_i$, $i=1,\dots,n$, is  real analytic  on $\setU$; if  $\setU=\reals^m$, then $f$ is real analytic.
\end{enumerate}
\end{dfn}

\begin{lem}\cite[Corollary 1.2.4]{krpa92}\label{lem:analyticrad}
If a power series  
\begin{align}
f(x)=\sum_{j=0}^\infty a_j (x-\alpha)^j
\end{align}
converges on an  open interval $\setI\subseteq \reals$, then $f$ is real analytic on $\setI$. 
\end{lem}

\begin{lem}\cite[Proposition 2.2.2]{krpa92}\label{lem:analyticalg}
Let $\setU$ and $\setV$ be open sets in $\reals^m$. If $f\colon\setU\to\reals$ is real analytic on $\setU$ and $g\colon\setV\to\reals$ is real analytic on $\setV$, then $f+g$ and $f\cdot g$ are both real analytic on $\setU\cap\setV$. Furthermore, if $g(\vecx)\neq 0$ for all $\vecx\in\setU\cap\setV$, then  $f/g$ is  real analytic on $\setU\cap\setV$.
\end{lem}

\begin{cor}\label{cor:poly}
All polynomials on $\reals^m$ are real analytic. 
\end{cor}

\begin{lem} \cite[Proposition 2.2.3]{krpa92} \label{lem:analyticder}
Let  $\setU$ be an open set in $\reals^m$ and suppose that $f\colon\setU\to\reals$ is real analytic on $\setU$. Then, the partial derivatives---of arbitrary order---of $f$ are real analytic on $\setU$.
\end{lem}

\begin{lem}\cite[Proposition 2.2.8]{krpa92}\label{lem:analyticcomp1}
Let $\setU$ be an open set in $\reals^m$ and $\setV$ an open set in $\reals^n$. 
Suppose that $f\colon\setU\to \reals^n$  is real analytic on $\setU$ and  $g\colon\setV\to \reals^k$ is real analytic on $\setV$. 
Then, for every $\vecx\in\setU$ with $f(\vecx)\in\setV$,  there exists an open set $\setW\subseteq\reals^m$ such that $\vecx\in\setW$ and $(g\circ f)|_\setW\colon\setW\to \reals^k$ is real analytic on $\setW$. 
\end{lem}

\begin{cor}\label{lem:analyticcomp}
If  $f\colon\reals^m\to \reals^n$  and $g\colon\reals^n\to \reals^k$ are both  real analytic, then $g\circ f\colon\reals^m\to \reals^k$ is real analytic. 
\end{cor}

\begin{lem} \cite[p. 83]{krpa92} \label{lem:vanishrealanalytic}
A real analytic function $f\colon\reals^m\to \reals$ vanishes either identically or on a set of  Lebesgue measure zero.
\end{lem}

The following definition extends the notion of a diffeomorphism of class $C^r$ \cite[Definition 3.1.18]{fed69} to a real analytic diffeomorphism. 
\begin{dfn}(Real analytic diffeomorphism)
Let $\setU\subseteq\reals^n$ be an open set and $\rho\colon\setU\to\reals^n$ real analytic. The mapping  $\rho$ is a real analytic diffeomorphism if $\setV=\rho(\setU)$ is an open set and the inverse $\rho^{-1}$  exists on $\setV$ and is real analytic on $\setV$. 
\end{dfn}

\begin{thm}\cite[Theorem 2.5.1]{krpa92}\label{thm:inverse}
Let $\setU\subseteq\reals^m$ be open and $f\colon\setU\to\reals^m$ real analytic. If $J\! f(\vecx_0)>0$ for some  $\vecx_0\in\setU$, then $f^{-1}$ exists on 
an open set $\setV\subseteq\reals^m$ containing $f(\vecx_0)$ and  is real analytic on $\setV$. 
\end{thm}

\begin{cor}\label{thm:ivanalytic}
Let $\setU$ be an open set in $\reals^m$ and suppose that $f\colon\setU\to \reals^m$  is real analytic on $\setU$. 
If 
there exists an $\vecx_0\in\setU$ such that 
 $J\! f(\vecx_0)> 0$,  then there exists an  $r>0$ such that $f|_{\setB_m(\vecx_0,r)}$ is a real analytic diffeomorphism.
\end{cor}
\begin{proof}
Suppose that there exists an $\vecx_0\in\setU$ such that 
 $J\! f(\vecx_0)> 0$.  Theorem \ref{thm:inverse} then implies that there exists an open set $\setV\subseteq\reals^m$ such that $f(\vecx_0)\in\setV$, and  $f^{-1}$ exists on $\setV$ and is real analytic on $\setV$. 
Since $\setV$ is open and $f(\vecx_0)\in\setV$, there must exist an $\varepsilon>0$ such that $\setB_m(f(\vecx_0),\varepsilon)\subseteq\setV$. 
By continuity of $f$, which follows  from real analyticity,  there must exist an $r>0$ such that 
$f(\setB_m(\vecx_0,r))\subseteq\setB_m(f(\vecx_0),\varepsilon)\subseteq\setV$. We set 
$\setW=f(\setB_m(\vecx_0,r))$. 
Summarizing, $f$ is real analytic on $\setB_m(\vecx_0,r)$ and $f^{-1}$ exists on  $\setW=f(\setB_m(\vecx_0,r))$ and is 
real analytic on $\setW$. It remains to show that $\setW$ is open. This follows by noting that   $\setW=(f^{-1})^{-1}(\setB_m(\vecx_0,r))$ is    the inverse image of an open set under a real analytic (and therefore continuous) mapping, and is hence open. We conclude that $f|_{\setB_m(\vecx_0,r)}$ is a real analytic diffeomorphism. 
\end{proof}

\begin{cor}\label{cor:inverse}
Let $\setU$ be an open set in $\reals^m$ and suppose that $f\colon\setU\to \reals^n$  with $n\geq m$ is real  analytic on $\setU$. If 
there exists an $\vecx_0\in\setU$ such that $J\! f(\vecx_0)> 0$, then there exists an  $r>0$ such that  $f$ is one-to-one on $\setB_m(\vecx_0,r)$.  
\end{cor}
\begin{proof}
Suppose that there exists an $\vecx_0\in\setU$ such that 
 $J\! f(\vecx_0)> 0$. 
Then,  $n\geq m$ implies $\rank(Df(\vecx_0))=m$. Thus,  the $n\times m$ matrix $Df(\vecx_0)$ has $m$ linearly independent row vectors. Denote the indices of $m$ such linearly independent row vectors by   $\{i_1,\dots, i_m\}$, % of $Df(\vecx_0)$ 
and consider the  mapping 
\begin{align}
g\colon\setU&\to \reals^m\\
\vecx&\mapsto \tp{(f_{i_1}(\vecx)\dots f_{i_m}(\vecx))}. 
\end{align}
Since $f$ is real analytic on $\setU$, so is $g$.
Furthermore,  
 $\rank(Dg(\vecx_0))=m$ and hence  $Jg(\vecx_0)> 0$.  Corollary \ref{thm:ivanalytic} therefore implies the existence of an  $r>0$ such that $g|_{\setB_m(\vecx_0,r)}$ is a real analytic diffeomorphism. In particular, $g$ is one-to-one on $\setB_m(\vecx_0,r)$, which  in turn implies that $f$ is one-to-one on $\setB_m(\vecx_0,r)$.  
\end{proof}

Next, we show that the square root is a real analytic diffeomorphism on the set of positive real numbers. 

\begin{lem}\label{lem.sqrt}
The function $\sqrt{\phantom{\cdot}}\colon \reals_+\to\reals_+$, $x\mapsto \sqrt{x}$,
where $\reals_+=\{x\in\reals:x> 0\}$, is a real analytic diffeomorphism. 
\end{lem}
\begin{proof}
The function
$(\cdot)^2\colon \reals_+\to\reals_+$, $x\mapsto x^2$ 
is real analytic by Corollary \ref{cor:poly}. Let $y\in\reals_+$ be arbitrary but fixed and set $x=\sqrt{y}$. 
Since  $\mathrm d y/\mathrm dx=2x>0$, Theorem \ref{thm:inverse} implies that there exists an $r>0$ such that the inverse of ${(\cdot)}^2$, given by  $\sqrt{\cdot}$, exists on $(y-r,y+r)$ and is real analytic on $(y-r,y+r)$. 
As $y$ was arbitrary, it follows that  $\sqrt{\cdot}$ is real analytic on $\reals_+$. Finally, since ${(\cdot)}^2$ is the inverse of $\sqrt{\cdot}$ and $\sqrt{\reals_+}=\reals_+$ is open, $\sqrt{\cdot}$ is a real analytic diffeomorphism. 
\end{proof}

\begin{lem}\label{lem:lebJ}
If $f\colon\reals^m\to \reals^n$  is real analytic, so is $J\!f$. In particular, $J\!f$ 
vanishes either identically or on a set of  Lebesgue measure zero. 
\end{lem}
\begin{proof}
Suppose that $f\colon\reals^m\to \reals^n$  is real analytic. 
Recall that $J\!f(\vecx)=\sqrt{g(\vecx)}$, where   
$g\colon\reals^m\to \reals$,   
\begin{align}
g(\vecx)=
\begin{cases}
\det(Df(\vecx)\tp{(Df(\vecx))})&\text{if}\ n<m\\
\det(\tp{(Df(\vecx))}Df(\vecx))&\text{else}.
\end{cases}
\end{align}
Lemmata  \ref{lem:analyticalg} and \ref{lem:analyticder} imply that $g$ is real analytic. As $\sqrt{\cdot}$ is real analytic on $\reals_+$ by Lemma \ref{lem.sqrt}, 
real analyticity of $J\!f$ follows from Lemma  \ref{lem:analyticcomp1}. Finally, as $J\!f$  is real analytic, it vanishes by Lemma \ref{lem:vanishrealanalytic} either  identically or on a set of  Lebesgue measure zero. 
\end{proof}

We  have the following important properties of real analytic mappings. 
\begin{lem}\label{lem:immreal}
Let $h\colon\reals^s\to \reals^m$ be a real analytic mapping of $s$-dimensional Jacobian $Jh\not\equiv 0$. Then, the following properties hold.
\begin{enumerate}[label=\roman*)]
\item \label{immreal1}
The set $\setO=\{\vecz\in\reals^s:Jh(\vecz)>0\}$ is open  and satisfies  $\lebmeasure^s(\reals^s\mysetminus\setO)=0$. 
\item \label{immreal2}
For every  set $\setA\in\colB(\reals^s)$ of positive  Lebesgue measure, there exists a set $\setB\in\colB(\reals^s)$ of positive  Lebesgue measure such that $\setB\subseteq\setA$ and the mapping $h|_\setB$ is an embedding, i.e.,  $Jh(\vecz)> 0$ for all $\vecz\in\setB$ and $h$ is one-to-one on $\setB$.
\end{enumerate} 
\end{lem} 
\begin{proof}  
Openness of  $\setO$ follows from continuity of $Jh$, and  Lemma \ref{lem:lebJ} 
together with $Jh\not\equiv 0$ implies   $\lebmeasure^s(\reals^s\mysetminus\setO)=0$. To prove \ref{immreal2}, consider $\setA\in\colB(\reals^{s})$ of positive Lebesgue measure. As  $\lebmeasure^s(\reals^s\mysetminus\setO)=0$, with $\setO$ from \ref{immreal1}, it follows that 
$\lebmeasure^s(\setA\cap\setO)>0$. Thus, by 
 Lemma \ref{lem:loef},  there must exist a $\vecz_0\in\setA\cap\setO$ such that  
\begin{align}
\lebmeasure^s(\setB_s(\vecz_0,r)\cap\setA\cap\setO)>0\quad \text{for all}\ r>0.
\end{align}
Since $Jh(\vecz_0)>0$, which follows from  $\vecz_0\in\setO$,  Corollary \ref{cor:inverse}  implies that there must exist an $r_0>0$ such that $h$ is one-to-one on $\setB_s(\vecz_0,r_0)$. Setting $\setB=\setB_s(\vecz_0,r_0)\cap\setA\cap\setO$ concludes the proof. 
\end{proof}

\begin{dfn}(Real analytic submanifold)\cite[Definition 2.7.1]{krpa92}\label{dfn:analyticmanifold}
A subset $\setM\subseteq\reals^n$ is an $m$-dimensional real analytic submanifold of $\reals^n$ if, for each $\vecy\in \setM$, there exist an open set $\setU\subseteq\reals^m$ and a real analytic immersion $f\colon\setU\to\reals^{n}$ such that $\vecy\in f(\setU)$ and open subsets of $\setU$ are mapped onto relatively open subsets in $\setM$. Here, a subset $\setW\subseteq \setM$ is called relatively open in $\setM$ if there exists an open set $\setV\subseteq\reals^n$ such that 
$\setW=\setV\cap\setM$.
\end{dfn} 

Equivalent definitions of real analytic submanifolds are listed in \cite[Proposition 2.7.3]{krpa92}.

\begin{lem}\label{lem:man}
Let $\setM\subseteq\reals^s$ be a $t$-dimensional real analytic submanifold of $\reals^s$ and $\vecz_0\in\setM$. Then, there exist a real analytic  embedding  $\zeta\colon\reals^{t}\to\setM\subseteq\reals^s$ and an $\eta>0$ such that 
\begin{align}
\zeta(\veczero)&=\vecz_0,\label{eq:M1}\\ 
\setB_{s}(\vecz_0,\eta)\cap\setM&\subseteq\zeta\big(\reals^{t}\big),\label{eq:M2}
\end{align}
and $\zeta\big(\reals^{t}\big)$ is relatively open in $\setM$. %, i.e., there exists an open set $\setV\subseteq\reals^{s}$ such that $\zeta\big(\reals^{t}\big)=\setV\cap\setM$. 
\end{lem}
\begin{proof}
Since $\setM\subseteq\reals^s$ is a $t$-dimensional real analytic submanifold of $\reals^s$, Definition \ref{dfn:analyticmanifold} implies the existence of an open set $\setU\subseteq\reals^{t}$ and a real analytic immersion  $\xi\colon\setU\to\reals^{s}$ that maps open subsets of $\setU$ onto relatively open subsets in $\setM$ and satisfies $\vecz_0=\xi(\vecu_0)$ with $\vecu_0\in\setU$.  As  $\xi$ is an immersion,  
\begin{align}
\rank(D\xi(\vecv))=t\quad\text{for all}\ \vecv\in\setU.\label{eq:rankxi}
\end{align}
Since $\setU$ is open and $\vecu_0\in\setU$, there exists a $\rho>0$ such that $\setB_{t}(\vecu_0,\rho)\subseteq\setU$, which implies in turn  that 
\begin{align}\label{eq:ball}
\xi\big(\setB_{t}(\vecu_0,\rho)\big)\subseteq \xi(\setU). 
\end{align}
Using Corollary \ref{cor:inverse} and $J\xi(\vecu_0)>0$ (recall that $\xi$ is an immersion) we may choose $\rho$ sufficiently small for $\xi$ to be one-to-one on $\setB_{t}(\vecu_0,\rho)$. 
As $\xi(\setB_{t}(\vecu_0,\rho))$ is relatively open in $\setM$, there must exist an open set $\setV\subseteq\reals^{s}$ such that 
\begin{align}\label{eq:BV}
\xi\big(\setB_{t}(\vecu_0,\rho)\big)=\setV\cap\setM.
\end{align}
Now, as $\setV$ is open and $\vecz_0=\xi(\vecu_0)\in\setV$, we can find an $\eta>0$ such that $\setB_{s}(\vecz_0,\eta)\subseteq\setV$, which, together with \eqref{eq:BV}, yields 
\begin{align}\label{eq:ball2}
\setB_{s}(\vecz_0,\eta)\cap \setM\subseteq \xi\big(\setB_{t}(\vecu_0,\rho)\big).   
\end{align}
Let $\kappa\colon\reals^{t}\to \setB_{t}(\veczero,\rho)$ be the real analytic diffeomorphism constructed in  Lemma \ref{lem.extend} below and set 
\begin{align}
\zeta\colon\reals^t&\to\reals^s\\
\vecv&\mapsto \xi(\kappa(\vecv)+\vecu_0).
\end{align}
The mapping  $\zeta$ is  real analytic  by Lemmata  \ref{lem:analyticalg} and \ref{lem:analyticcomp1}. 
Clearly, $\zeta$ is one-to-one on $\reals^t$ as $\kappa$ is a diffeomorphism and 
$\xi$ is one-to-one on $\setB_{t}(\vecu_0,\rho)$. 
Since  $\zeta\big(\reals^{t}\big)=\xi(\setB_{t}(\vecu_0,\rho))$,  \eqref{eq:ball2} establishes \eqref{eq:M2} and \eqref{eq:BV}
proves that $\zeta\big(\reals^{t}\big)$ is  relatively open in $\setM$.
Finally,  \eqref{eq:M1} follows from  $\zeta(\veczero)=\xi(\kappa(\veczero)+\vecu_0)=\xi(\vecu_0)=\vecz_0$,  
where in the second equality we used \eqref{eq:kappa0} in Lemma \ref{lem.extend} below. 
It remains to show that $\zeta$ is an immersion, which is effected  by proving that $\rank(D\zeta(\vecv))=t$ for all $\vecv\in\reals^t$.    
The chain rule  implies 
\begin{align}
D\zeta(\vecv)
&= (D\xi)(\kappa(\vecv)+\vecu_0)D\kappa(\vecv)\quad\text{for all}\ \vecv\in\reals^t.
\end{align}
It now follows
\begin{enumerate}[label=\roman*)]
\item from \eqref{eq:kappaD0} in Lemma \ref{lem.extend} below that $\rank(D\kappa(\vecv))=t$ for all $\vecv\in\reals^t$, and \label{jac1}
\item from \eqref{eq:rankxi}  and $\kappa(\vecv)+\vecu_0\in\setB_{t}(\vecu_0,\rho)\subseteq\setU$ that $\rank((D\xi)(\kappa(\vecv)+\vecu_0))=t$ for all $\vecv\in\reals^t$.\label{jac2}
\end{enumerate}
Applying  Lemma \ref{lem:sylvester} to 
$(D\xi)(\kappa(\vecv)+\vecu_0)\in\reals^{s\times t}$ and $D\kappa(\vecv)\in\reals^{t\times t}$, and using 
\ref{jac1} and \ref{jac2} above  yields $\rank(D\zeta(\vecv))\geq t$ for all $\vecv\in\reals^t$,  which  in turn  implies 
$J\zeta(\vecv)>0$  for all  $\vecv\in\reals^t$, 
thereby concluding the proof. 
\end{proof}

\begin{lem}\label{lem.extend}
For $\rho>0$, the mapping 
\begin{align}
\kappa\colon\reals^k&\to\setB_k(\veczero,\rho)\\
\vecx&\mapsto \frac{\rho\vecx}{\sqrt{1+\|\vecx\|_2^2}}
\end{align}
is a real analytic diffeomorphism on $\reals^k$  satisfying  
\begin{align}
\kappa(\veczero)&=\veczero,\label{eq:kappa0}\\
\kappa(\reals^k)&=\setB_k(\veczero,\rho),\\
\rank(D\kappa(\vecx))&=k\quad\text{for all}\ \vecx\in\reals^k.\label{eq:kappaD0}
\end{align} 
\end{lem} 
\begin{proof}
It follows   from the definition of $\kappa$ that $\kappa(\veczero)=\veczero$.
The mapping $\kappa$ is real analytic thanks to Lemmata \ref{lem:analyticalg}, \ref{lem:analyticcomp1}, and \ref{lem.sqrt}.
Now, consider the mapping 
\begin{align}
\sigma\colon\setB_k(\veczero,\rho)&\to \reals^k\\
\vecy&\mapsto \frac{\vecy}{\sqrt{\rho^2-\|\vecy\|_2^2}}. 
\end{align}
Again, $\sigma$ is  real analytic thanks to Lemmata \ref{lem:analyticalg}, \ref{lem:analyticcomp1}, and \ref{lem.sqrt}. Since $(\kappa\circ\sigma)(\vecy)=\vecy$ for all $\vecy\in\setB_k(\veczero,\rho)$, it follows that $\kappa(\reals^k)=\setB_k(\veczero,\rho)$. 
Moreover, 
as $(\sigma\circ\kappa)(\vecx)=\vecx$ for all $\vecx\in\reals^k$, $\sigma$ is the inverse of $\kappa$, which establishes  that $\kappa$ is a real analytic diffeomorphism on $\reals^k$.  
Finally, since $(\sigma\circ\kappa)(\vecx)=\vecx$ for all $\vecx\in\reals^k$, the chain rule implies  $\matI_k=D(\sigma\circ\kappa)(\vecx)=(D\sigma)(\kappa(\vecx))D\kappa(\vecx)$ for all $\vecx\in\reals^k$, which yields 
\eqref{eq:kappaD0}. 
\end{proof}

\begin{prp}\cite[Proposition 2.7.3]{krpa92}\label{prp:analyticmanifold}
Let $\setM\subseteq\reals^n$. The following statements are equivalent:
\begin{enumerate}[label=\roman*)]
\item $\setM$ is an $m$-dimensional real analytic submanifold of $\reals^n$.
\item For each $\vecz\in\setM$ there exist an open set $\setU\subseteq\reals^n$ with $\vecz\in\setU$, a real analytic diffeomorphism $\rho\colon\setU\to\reals^n$, and an $m$-dimensional linear subspace $\setL\subseteq \reals^n$, such that
\begin{align}
\rho(\setM\cap\setU)=\rho(\setU)\cap\setL.
\end{align} \label{prop2analyticmanifold}
\end{enumerate}
\end{prp}
\vspace*{-9truemm}

Proposition \ref{prp:analyticmanifold} allows us to state the following sufficient condition on the inverse image of a point under a real analytic function to result in a real analytic submanifold. 

\begin{lem}\label{lem:summf}
Let $\psi\colon\reals^s\to\reals$ be a real analytic function, $y_0\in \psi(\reals^s)$, and set  
$\setM_{y_0}=\psi^{-1}(\{y_0\})$. Suppose that  $J\psi(\vecz)>0$ for all $\vecz\in\setM_{y_0}$. Then, $\setM_{y_0}$ is an $(s-1)$-dimensional real analytic submanifold of $\reals^s$. 
\end{lem}
\begin{proof}
We may assume, w.l.o.g., that  $y_0=0$ (if $y_0\neq0$, we set  $\tilde\psi(\vecz)=\psi(\vecz)-y_0$ and prove the lemma with  $\tilde\psi$ in place of $\psi$, noting that $\setM_{y_0}=\tilde\psi^{-1}(\{0\})$ and $J\tilde\psi(\vecz)=J\psi(\vecz)$ for all $\vecz\in\reals^s$).  
The proof is effected by verifying that  $\setM_0$ satisfies Statement \ref{prop2analyticmanifold} of Proposition \ref{prp:analyticmanifold}. 
Let $\vecz_0\in \setM_0$ be arbitrary but fixed and set $\tp{\veca_1}=D\psi(\vecz_0)$.
Since $J\psi(\vecz_0)>0$ by assumption, we must have $\veca_1\neq \veczero$. 
Choose  $\veca_2,\dots,\veca_s\in\reals^s$ such that $\matA=\tp{(\veca_1\mydots\veca_s)}$ is a regular matrix  
and  consider the mapping 
\begin{align}
\rho\colon\reals^s&\to\reals^s\\
       \vecz&\mapsto \tp{\big(\psi(\vecz)\ \tp{\veca_2}\vecz \mydots \tp{\veca_s}\vecz\big)}. \label{eq:defrho} 
\end{align}
Note that  $D\rho(\vecz_0)=\matA$.   
Since $\rank(\matA)=s$ by construction,  $J\rho(\vecz_0)>0$. 
It therefore follows from 
Corollary \ref{thm:ivanalytic} that there exists an $r>0$ such that $\rho|_{\setB_s(\vecz_0,r)}$ is a real analytic diffeomorphism. 
Finally, we can write 
\begin{align}
\rho(\setM_0\cap\setB_s(\vecz_0,r))
&=\rho(\{\vecz\in\setB_s(\vecz_0,r):\psi(\vecz)=0\})\label{eq:psi1}\\
&=\rho(\{\vecz\in\setB_s(\vecz_0,r):\tp{\vece_1}\rho(\vecz)=0\})\label{eq:psi2}\\
&=\rho(\setB_s(\vecz_0,r))\cap\setL,\label{eq:psi3}
\end{align}
where \eqref{eq:psi1} follows from $\setM_0=\psi^{-1}(\{0\})$,  \eqref{eq:psi2} is by  \eqref{eq:defrho}, 
and in \eqref{eq:psi3} we set $\setL=\{\tp{(w_1\mydots w_s)}\in\reals^s: w_1=0\}$.
\end{proof}

\begin{lem}(Sylvester's inequality)\cite[Chapter 0.4.5, Property (c)]{hojo13}\label{lem:sylvester}
If $\matA\in\reals^{m\times k}$ and $\matB\in\reals^{k\times n}$,  then 
\begin{align}
\rank(\matA)+\rank(\matB)-k
&\leq\rank(\matA\matB)\\
&\leq \min\{\rank(\matA),\rank(\matB)\}.
\end{align}
\end{lem}
\section*{Acknowledgments}
The authors are grateful to the reviewers for
very insightful comments, which helped  to  improve the paper.

%\balance
%\nocite{*}
\bibliographystyle{./IEEEtran}
\bibliography{references}

\end{document}